\documentclass[letterpaper,11pt,leqno]{amsart} 
\usepackage[T1]{fontenc}

\usepackage[T1]{fontenc}
\usepackage{amssymb}
\usepackage{amsthm}
\usepackage{amsmath}
\usepackage[utf8]{inputenc}
\usepackage{bera}

\newtheorem{knowntheorem}{Theorem}[section]

\usepackage{graphicx}
\usepackage{amssymb}                       
\usepackage{amsmath}
\usepackage{color}
\usepackage{times}

\usepackage[unicode,bookmarks,colorlinks]{hyperref}
\hypersetup{
  linkcolor=brickred,
}

\definecolor{mahogany}{cmyk}{0, 0.77, 0.87, 0}
\definecolor{salmon}{cmyk}{0, 0.53, 0.38, 0}
\definecolor{melon}{cmyk}{0, 0.46, 0.50, 0}
\definecolor{yellowgreen}{cmyk}{0.44, 0, 0.74, 0}
\definecolor{brickred}{cmyk}{0, 0.89, 0.94, 0.28}
\definecolor{OliveGreen}{cmyk}{0.64, 0, 0.95, 0.40}
\definecolor{RawSienna}{cmyk}{0, 0.72, 1.0, 0.45}
\definecolor{ZurichRed}{rgb}{1, 0, 0} 

\usepackage{amsmath,amssymb,amsthm,amsfonts,amsbsy,latexsym,dsfont,color,graphicx,enumitem, upgreek, xcolor, xfrac, times}
\usepackage[numeric,initials,nobysame]{amsrefs}

\usepackage[mathscr]{eucal}

\newcommand{\Sone}{\mathbb{S}^1}

\newcommand{\Soned}{\mathbb{S}^{d-1}}

\newtheorem{theorem}{Theorem}[section]
\newtheorem{corollary}[theorem]{Corollary}
\newtheorem{proposition}[theorem]{Proposition}
\newtheorem{lemma}[theorem]{Lemma}
\newtheorem{problem}[theorem]{Problem}
\newtheorem{definition}[theorem]{Definition}
\newtheorem{remark}[theorem]{Remark}

\newtheorem{example}[theorem]{Example}

\theoremstyle{remark}

\numberwithin{equation}{section}

\newcommand{\brk}[1]{\langle #1 \rangle}

\newcommand{\N}{\mathbb{N}}

\newcommand{\R}{\mathbb{R}}

\renewcommand{\P}{\mathbb{P}}

\newcommand{\p}{\mathsf{p}}

\newcommand{\bB}{\mathbb{B}}
\newcommand{\bC}{\mathbb{C}}

\newcommand{\bE}{\mathbb{E}}

\newcommand{\bM}{\mathbb{M}}
\newcommand{\bN}{\mathbb{N}}

\newcommand{\bP}{\mathbb{P}}

\newcommand{\bR}{\mathbb{R}}

\newcommand{\cA}{\mathcal{A}}
\newcommand{\cB}{\mathcal{B}}

\newcommand{\cE}{\mathcal{E}}
\newcommand{\cF}{\mathcal{F}}

\newcommand{\cT}{\mathcal{T}}

\newcommand{\BH}{\bold{H}}

\newcommand{\sF}{\mathscr{F}}

\newcommand{\wh}{\widehat}

\DeclareMathOperator{\sgn}{sgn}

\DeclareMathOperator{\sign}{sign}

\newcommand{\ip}[1]{\langle #1\rangle}

\newcommand{\BR}{\bold{R}}

\date{\today}

\begin{document}

\title{Cotlar martingale transforms and related singular integrals}
\author{Rodrigo Ba\~nuelos}
\address{Department of mathematics. Purdue University}
\email{banuelos@purdue.edu}

\begin{abstract} The ``magical" identity discovered by M.~Cotlar in 1955 for the Hilbert transform is established here in the setting of martingale transforms and, in particular, for conformal martingales. This, together with the probabilistic representation of the Riesz transforms, shows that, at the level of martingale transforms and in odd dimensions, they exhibit the same analytic-type structure as the Hilbert transform on the real line. Consequently, Cotlar’s proof of the sharp $L^p$ inequality for powers of $2$ applies. The significance of the martingale Cotlar identity, whose proof is entirely elementary, does not lie in providing an alternative proof of this well-known and relatively simple estimate, but rather in the structural viewpoint it reveals. This structure is explored further.

Independent of Cotlar’s identity, asymptotic bounds for the $L^p$ norm of the vector of Riesz transforms are investigated. It is shown that, in the limit as $p \to \infty$, this norm coincides asymptotically with that of the Hilbert transform on the real line.

The study of the Cotlar identity in the martingale setting is motivated by the desire to gain new insight into two longstanding open problems: T.~Iwaniec’s 1983 conjecture on the norm of the Beurling-Ahlfors operator and the problem of determining the sharp constant in E.~M.~Stein’s 1984 inequality for the vector of Riesz transforms. Related problems are also discussed.

The paper contains both a survey of known results and new contributions. An effort has been made to keep the exposition as self-contained as possible and to present the material in an accessible, largely expository style.

\end{abstract}

\maketitle
\tableofcontents 

\section{Introduction}  
The purpose of this work  is to present a Cotlar identity for martingales and to describe its relationship to the classical Riesz transforms on $\mathbb{R}^d$. Although, as we shall see,  a Cotlar identity does not hold  for the Riesz transforms, their representation \`a la Gundy-Varopoulos \cite{GV79} shows that, at the level of martingale transforms and in odd dimensions, a conformal Cotlar-type structure is nevertheless present. More precisely, when $d$ is odd, the Riesz transforms arise as conditional expectations of martingale transforms associated with certain $(d+1)\times(d+1)$ matrices which we call \emph{Cotlar matrices} and for which the Cotlar martingale identity holds.  

These matrices naturally extend to $\mathbb{C}^n$, with $n=d+1$, the structural features underlying the Hilbert transform and its connection with analytic functions in the complex plane; see Remark~\ref{Mart:CotlarHilbertProjection}, Example~\ref{Cot:Proj4D}, and equation~\eqref{THE:CID}. As a consequence of this connection, we obtain the known sharp $L^p$ bounds for the Riesz transforms when $p=2^n$, $n\in\mathbb{N}$, exactly as in Cotlar’s classical result for the Hilbert transform and without appealing to the method of rotations, the Pichorides inequality~\cite{Pic72}, martingale inequalities from~\cite{BW95} or Bellman function techniques \cite{Volberg1}.  These ideas also apply to other geometric settings where no method of rotations is available, including the discrete operators studied in~\cite{BanKwa, BanKwa1, BanKimKwa2026, BanKim2026} which use the inequalities form \cite{BW95}.

The now extensive literature, see for example  \cite{BanKwa, BanKimKwa, BanOsc15, BanBau13, BanBauLiYan, BanKimKwa, Permic24} and the many reference contained therein, on applications of martingale techniques to  Riesz transform and related singular integrals motivates the following questions. Is there a martingale version of Cotlar's identity? Could such a martingale identity, together with the probabilistic representation of the Riesz transforms, lead to applications in higher dimensions? In particular, might a martingale Cotlar formula shed new light on the over 40-year old unsolved problems:
\begin{enumerate}
\item the sharp $L^p$ bounds for the vector of Riesz transforms (E.\,M.~Stein~\cite{SteSome}), and 
\item the conjecture of T.~Iwaniec~\cite{Iwa82} concerning the $L^p$ norm of the Beurling-Ahlfors operator on $\bC$,  
\end{enumerate}
even for special values of $p$ such as, for example, powers of $2$?

Both problems, and especially~(2), have been studied extensively over the past four decades, with martingale inequalities playing a central role in many of the key developments. From the martingale point of view, which  fails to yield even the simplest case of $p=2$  that follows trivially from Fourier transform, the level of abstraction required to make further progress on these problems seems  remarkably similar as discussed below.

\subsection{Notation} 
Throughout  we will mainly work with smooth real-valued functions of compact support. That is, we assume that $f:\R^d\to \R$ and  $f\in C^{\infty}_0(\R^d)$.  Once the $L^p$  inequalities of interest are proved for this class of functions the usual density arguments give the results for all functions in $L^p$. There will also be occasions when $f:\R^d\to \bC$ and this will be explicitly stated.

The standard notion $\|f\|_{L^p(\R^d)}$ is used for the $p$-norm of functions in $L^p(\R^d)$. By abuse of notation,  $\|T\|_{L^p(\R^d)}$ will denote the norm of an operator $T:L^p(\R^d) \to L^p(\R^d)$. A notable exception to this will be the norm of the Hilbert transform $H$ which will always be denoted by $\|H\|_p$. For functions $f:\R^d\to \bC$ we will write $\|f\|_{L^p(\R^d;\, \bC)}$ and similarly for operator norms.  

For a  random variable $X$ defined on a probability space $(\Omega, \bP, \sF)$ we will use the notation $\|X\|_p$ for its $p$-norm. 
 
In what follows $1<p<\infty$ and  $p^*=\max\{p, q\}$, where $q$ is the conjugate exponent of $p$. Define 

\begin{equation}\label{BurkCons}
    p^*-1= \begin{cases}
    \frac{1}{p-1}, &1<p\leq 2,\\
    p-1, & 2\leq p<\infty.
    \end{cases}
\end{equation}
  The constant $p^*-1$,  which is the sharp constant  for martingale transforms \cite{Burk84},   is often called the Burkholder constant.  Similarly the constant 
  \begin{align}\label{pichcole}
  \|H\|_p=\cot\left(\frac{\pi}{2 p^*}\right)=
    \begin{cases}
        \tan\left(\frac{\pi}{2p}\right), & 1<p\le 2 \\
        \cot\left(\frac{\pi}{2p}\right), & 2\leq p<\infty,   
    \end{cases}
   \end{align} 
   which is the p-norm of the Hilbert transform  is often called the Pichorides constant \cite{Pic72} (independently found by B. Cole \cite{Gam}).

Our normalization for the Fourier transform is
\[
\widehat{f}(\xi)=\int_{\R^d} f(x) e^{2\pi x\cdot \xi} dx, \quad  f(x)=\int_{\R^d} \widehat{f}(\xi) e^{-2\pi x\cdot \xi} d\xi. 
\]

\section{Cotlar's "magical"  identity} 

The Hilbert transform $H$ is the most basic singular integral defined on  $\R$ by 
 \begin{align*}
Hf(x)=\frac{1}{\pi}\, p.v.\int_{\R} \frac{f(x-y)}{y} dy.
\end{align*} The following well-known and "magical" identity
was proved by M. Cotlar in  \cite[pg 159]{Cot55}: 
\begin{equation}\label{CotlarId} 
|Hf|^2=2H(f Hf)+|f|^2.
\end{equation}
There are other variants of this identity in the literature including those for discrete and free Hilbert transforms. See for example \cite{BanKwa, BanKwa1, GonPerXia, MeiRic} and references given there. 

Cotlar's  proof of \eqref{CotlarId} is simple using either analytic function arising from  the harmonic extensions of $f$ and $Hf$ to the upper half-space, or by taking Fourier transforms of both sides  and using the fact that 
\begin{equation}\label{Fouier:MulH}
\widehat{Hf}(\xi)=i\sign(\xi)\widehat{f}(\xi)
\end{equation}   

For completeness, and since this argument will be adapted to the setting of conformal martingales, we outline the  proof based on analytic functions.  Let $u_f(x, y)$ be the harmonic extension of $f$ to the upper half-space--its convolution with the Poisson kernel $P_y(x)$. Set $z=x+iy$ and  consider the analytic function $F(z)=u_f(z)+iv_f(z)$, with $v_f(z)$ the conjugate harmonic function of $u_f$. Since $F$ is analytic so is $F^2=u_f^2 - v_f^2+2iu_fv_f$. Set 
\[U=u_f^2-u_{Hf}^2, \quad  \text{and} \quad  V=2u_fu_{Hf}.
\] 
Recall that $v_f(z)=u_{Hf}(z)$,  or equivalently that $v_{Hf}=u_{-f}$,  using the fact that $H^2=-I$.
Taking boundary values we obtain $|f|^2-|Hf|^2=-2H(f\,Hf)$ and rearranging gives \eqref{CotlarId}.

An application of this identity gives a simple proof of the $L^p$ boundedness of $H$ with the optimal constant for $p=2^k$, $k\in \N$.   Indeed,   applying  the Minkowski  and Cauchy-Schwarz  inequalities, it follows from  \eqref{CotlarId} that  
\begin{align*}
 \| H f\|_{2 p}^2 = \|(H f)^2\|_p & \le \|f^2 \|_p + 2 \|H\left(f\cdot H f\right) \|_p \\
&\le \|f \|_{2 p}^2 + 2 \|H \|_p\, \|f \cdot H f \|_p \\
& \le \|f \|_{2 p}^2 + 2 \|H \|_p\,\| f \|_{2 p}\, \|H f \|_{2 p}\\
 & \le \|f \|_{2 p}^2 + 2 \|H \|_p \, \|H \|_{2 p}\, \|f \|_{2 p}^2.\\
 \end{align*} 
This gives the inequality 
\begin{align*}
\|H \|_{2 p}  \le \|H \|_p + \sqrt{1 + \|H \|_p^2}. 
\end{align*} 
Note that the last inequality is in fact valid for any $1<p<\infty$. 
From $\|H\|_2=1$, induction on $n$ and the trigonometric identity 
\[
 \cot \left(\frac{\alpha}{2}\right)= \cot (\alpha) + \sqrt{1 + \cot^2 (\alpha)},
 \]
it follows that
\begin{align}\label{sharpH} \|H \|_p \le \cot \left(\frac{\pi}{2 p}\right), \quad p = 2^k, \,\,\, k=1, 2, \dots. 
\end{align}
By duality, a similar bound holds for $p=\frac{2^k}{2^k-1}$ with $p$ in   $\cot(\cdot)$  replaced by its conjugate exponent.  That is, this argument gives Pichorides bound for powers of $2$. 

For some history of the use of the trigonometric identity in the context of Hilbert transforms in other settings,  see \cite{BanKwa1, GohbergKrein1970, GohbergKrupnik1968} and references therein. 

Except for the best constant that arises from the (clever) connection to the trigonometric identity the above argument and the Marcinkiewicz interpolation theorem is what Cotlar used to give a proof of the boundedness of the Hilbert transform on $L^p$, for  $1<p<\infty$.  

\begin{remark} It may be of interest to note that in his original proof of the boundedness of the Hilbert transform (conjugate function), Riesz~\cite{Riesz} first established the result for $p=2m$ by using the fact that if $f$ is analytic so is $f^p$,  and then applying the Marcinkiewicz interpolation theorem;  see \cite[p.~213]{Gra}.  
\end{remark}

For $j=1, 2, \dots, d$ and $f\in L^p(\R^d)$, the classical Riesz transforms on $\R^d$ are the singular integrals defined   by 
\begin{align}\label{RieszT}
    R_{j}f(x) = c_d\,\, p.v.\int_{\R^d}\frac{y_j}{|y|^{d+1}}f(x-y)  dy,\qquad c_d=\pi^{-\tfrac{d+1}{2}}\Gamma(\tfrac{d+1}{2}), 
    \end{align}
 with their  Fourier transform given by
 \begin{equation}\label{RieszForierM}
\widehat{{R_{j}f}}(\xi)=\frac{i\xi_j}{|\xi|}\widehat{f}(\xi).  
\end{equation}
When $d=1$ both formulations agree with the Hilbert transform.  Since the Riesz transforms or $\R^d$, $d>1$,  are Fourier multipliers extensions of the Hilbert transform, it follows from de Leeuw's extension theorem, \cite{DeL}, \cite[Theorem 2.5.15]{Gra},  that $\|H\|_p\leq \|R_j\|_{L^p(\R^d)}$. On the other hand, the Calder\'on-Zygmund method of rotations gives  that $\|R_j\|_{L^p(\R^d)}\leq  \|H\|_p$. Thus, 
\begin{equation}\label{IwMa2}
\|R_j\|_{L^p(\R^d)}=\|H\|_p=\cot\Big(\frac{\pi}{2p^*}\Big), \quad 1<p<\infty.
\end{equation} 
The equality \eqref{IwMa2} was first verified in  \cite{IwaMar} with this  argument.   This fact  (i.e., \eqref{IwMa2}) will be use several times below.

In \cite{BW95}, sharp  inequalities for orthogonal martingales are proved from which the  upper bound inequality \eqref{IwMa2} follows.  A fair questions that one may ask here is:~given the simplicity of the above proof based on the method of rotation, why the need for martingale inequalities? One advantage of the martingale approach is their applications to  Riesz transforms on different geometric settings beyond $\R^d$, as the literature already cited shows.

It follows from Cotlar's Fourier transform proof of \eqref{CotlarId}, \cite[p.~155]{Cot55},    that if $T_m$ is Fourier multiplier with the function $m\in L^{\infty}(\R^d)$ then the identity for $T_m$ replacing $H$ is equivalent to 
\begin{align}\label{CotFourier} \left(m(\xi+\gamma)-m(\xi)\right)\left(m(-\xi)-m(\gamma)\right)=0, \,\,\, \text{a.e.}\,\, \xi, \gamma\in \R^d. 
\end{align}
This imposes rigid conditions on the multiplier that are rarely satisfied outside of simple variations of the Hilbert transform. With $m_j(\xi)=\frac{i\xi_j}{|\xi|}$,  it is straightforward to verify that Cotlar’s identity fails for the Riesz transforms. For example, take  
$j=1$, $\xi=e_1$, $\eta=e_2$ (the standard unit vectors) and  use continuity at $(\xi,\eta )$ to find a neighborhood of this point where the identity does not hold. Similarly, there is no Cotlar identity for the Beurling-Ahlfors operator or various other classical Fourier multiplier operators. 

Despite the failure of Cotlar’s identity for Riesz transforms, it is shown here that a version of the identity does hold in odd dimensions  for a special subclass  of orthogonal martingales transforms. Moreover, this identity, together with their probabilistic representation, can be used to prove the sharp $L^p$-bound for the Riesz transforms when $p=2^n$ without appealing to the method or rotations, sharp martingale inequalities or other Burkholder-Bellman function techniques. 

The proof of the martingale Cotlar identity is entirely elementary, relying only on the basic form of Itô’s formula applied to the product of two martingales. Nevertheless, applying the martingale identity  to the Riesz transform requires viewing the martingale transform in a different light than what has typically been used in previous applications. In particular, the martingale Cotlar identity developed here isolates the role of conformality  at the martingale level for Riesz transforms in odd dimensions. 

Since the application treated here, the sharp $L^p$ bound for $p=2^k$ 
 for the Riesz transforms, is a special case of \eqref{IwMa2} which admits a simple analytic proof via the method of rotations and also follows from the martingale inequalities in \cite{BW95}, the question raised above on the need for martingale for the upper bound in \eqref{IwMa2} becomes even more compelling.
The value of the martingale Cotlar identity therefore does not lie in providing an alternative proof for the special case of this estimate, but rather in the structural viewpoint it reveals. Such identities expose hidden algebraic relationships at the martingale level and may extend to situations where analytic methods are unavailable. This perspective may be particularly relevant in settings such as in \cite{BanKwa}, where Cotlar-type identities and inequalities play a role in identifying the norm of the discrete Riesz–Titchmarsh Hilbert transform for $p=2k$, $k\in \N$.

\section{Cotlar's identity for martingales}

Let us first recall two inequalities for martingales with continuous paths. Fix $d\geq 2$  and let 
$B_t=(B_t^1,\dots,B_t^{d})$ 
be the standard d-dimensional Brownian motion equipped with its standard Brownian filtration $\mathcal F_t=\sigma(B_s:0\le s\le t)$ on the probability space $(\Omega,\mathcal F,\mathbb P)$.  Let  
$K_t=(K_t^1,\dots,K_t^{d})\in \R^d$ be predictable with
\[
\bE\int_0^t |K_s|^2\,ds<\infty,\,\,\, \text{for all }\,\,\, t>0. 
\]
For any two vectors $u, v\in \R^d$, we use  $u\cdot v=\ip{u, v}$ to denote  the  standard Euclidean inner product on $\R^d$. 
Define the  martingale  given by the stochastic integral 
\begin{align}\label{StochasticMartin}
M_t =M_0+\int_0^t K_s\cdot dB_s,\,\,\text{with quadrant variation}\,\, \ip{M}_t=\int_0^t |K_s|^2ds.
\end{align} 
We will often refer to martingales of this form as {\it Brownian martingales.} If $N_t$ is another martingale given by a predictable  process $H_t\in \R^d$,  the covariation  process of $M$ and $N$ is given by 
\[\ip{M, N}_t=\int_0^t M_s\cdot K_s ds.
\]
\begin{definition}
Consider two martingales $N_t$ and $M_t$ as above.  The martingale $M_t$ is said to be subordinate to $N_t$ if for all $t$, $\ip{M}_t\leq \ip{N}_t$.  They are said to be  orthogonal if  $\ip{M, N}_t=0$, for all $t$.  Throughout  the paper, $\|M\|_p=\sup_{t>0}\|M_t\|_p$. 
\end{definition} 

 \begin{knowntheorem}[\cite{Burk84, BW95}]\label{thm:A}
Suppose $M$ and $N$ are two martingales as above.    
\begin{enumerate}
\item Suppose $M$ is subordinate to $N$.  Then 
\begin{align}\label{Burkholder}
\|M\|_p\leq (p^*-1)\|N\|_p, \quad 1<p<\infty.
\end{align} 
\item Suppose $M$ and $N$ are orthogonal martingales and $M$ is subordinate to $N$.  Then 
\begin{align}\label{BanuelosWang}
\|M\|_p\leq \cot\left(\frac{\pi}{2 p^*}\right)\|N\|_p, \quad 1<p<\infty.
\end{align} 
\end{enumerate} 
Both inequalities are sharp. 
\end{knowntheorem}

Inequality \eqref{Burkholder} is a case of the celebrate inequalities of Burkholder \cite{Burk84} which have  had multiple applications in different areas of analysis, probability and related fields. Burkholder inequality  hold for Hilbert space-valued martingalesn with the same cosnatnt. The notion of subordinate orthogonal martingales was introduced in \cite{BW95}  where  the inequality is  proved.  This too  has  Hilbert space-valued versions \cite[Theorem A]{BanuelosWang1996}.  The main application in that paper is to prove the sharp inequality for Reisz transforms in \eqref{IwMa2}. An advantage of the martingale approach is their applications to  Riesz transforms on different geometric settings beyond $\R^d$, as the now large subsequent literature has shown. For parallel weak-type $(1, 1)$ inequalities, we refer the reader to \cite{BanWanDavis} and references therein.  

In addition to the literature already mention above, the interested reader is highly  encouraged to consult  \cite{Osekowski2012} where many extensions and applications are presented as well as the survey article \cite{Osekowski2013BellmanSurvey} where the Burkholder-Bellman function method, which is at the heart of the sharp martingale inequalities and their applications, is discussed.  

In what follows, our goal is to present a martingale analogue of Cotlar’s identity and to use it to derive \eqref{BanuelosWang} in the special case where $\ip{M}_t=\ip{N}_t$ and $p=2^k$, $k\in \N$,  paralleling Cotlar’s original argument for the Hilbert transform. 

With $M$ and $N$ as above, by It\^o's formula 
\begin{align}
N_tM_t&=\int_0^t M_s dN_s+\int_0^t N_s dM_s+\langle N, M\rangle_t\label{product}\\
&=\int_0^t M_s\,H_s\cdot dB_s + \int_0^t N_s K_s\cdot dB_s+\int_0^t H_s\cdot K_s ds.\nonumber
\end{align}
Taking   $M_t=N_t$ gives 
\begin{align}\label{square} 
M_t^2 &= 2\int_0^t M_s dM_s+\langle M\rangle_t =
2\int_0^t M_s K_s\cdot dB_s+\int_0^t |K_s|^2ds.
\end{align}

Let  $A=(a_{ij})$ be a $d\times d$ matrix with real coefficients.  Set $\|A\|^2=\sup\{|Av|^2: v\in \R^d, \, |v|=1\}$.  The martingale transform of $M_t$ by $A$ is defined as 
\begin{align}\label{stochasticMarinTras} 
A*M_t = \int_0^t AK_s\cdot dB_s. 
\end{align}

\begin{theorem}\label{CorlarMain} Suppose $A$ is a $d\times d$ matrix  such that $Av\bold{\cdot} v=0$  for all $v\in \R^d$.    Then 
\begin{align}
|A*M_t|^2=2A*(M(A*M))_t+\int_0^t |AK_s|^2 ds-2\int_0^t [M_sA^2K_s]\cdot dB_s\label{main1}
\end{align} 
\end{theorem} 

\begin{proof} 

Applying  \eqref{product} to the product of  $M_t$ and $(A*M)_t$ gives 
\begin{align*}
M_tA*M_t
&=\int_0^t A*M_s K_s\cdot dB_s+\int_0^t M_s AK_s\cdot dB_s+\int_0^t AK_s\cdot K_s ds\\
&=\int_0^t \tilde{K}\cdot dB_s, 
\end{align*} 
where $\tilde{K_s}=[A*M_s K_s+ M_sAK_s]$. The  assumption on $A$ gives $AK_s\cdot K_s=0$ and  
$\langle M, A*M\rangle_t, =0$.  Thus $M_t(A*M)_t$ is a martingale of the above form.  Its martingale transform by $A$ is 
\begin{align}\label{product2}
A*(M(A*M)_t=\int_0^t A\tilde{K}_s\cdot dB_s=\int_0^t [(A*M)_sAK_s + M_sA^2K_s]\cdot dB_s.
\end{align}

Similarly applying \eqref{square}  to  $(A*M)_t$ gives 
\begin{align}\label{square2}
|A*M)_t|^2=2\int_0^t [A*M_sAK_s]\cdot dB_s+\int_0^t |AK_s|^2 ds. 
\end{align} 
It follows  from \eqref{product2} and \eqref{square2} that 
\begin{align*}
2A*(M(A*M)_t=[(A*M_t]^2+2\int_0^t [M_sA^2K_s]\cdot dB_s-\int_0^t |AK_s|^2 ds.
\end{align*} 
Equivalently, 
\begin{align}\label{square3}
|A*M_t|^2=2A*M(A*M)_t-2\int_0^t [M_sA^2K_s]\cdot dB_s+\int_0^t |AK_s|^2 ds, 
\end{align} 
which is the claimed identity.  
\end{proof} 

\begin{corollary}\label{CotlarIneq} Suppose $A$ is as in the statement of the Theorem.   Then,  
\begin{align*}
|A*M_t|^2 &\leq 2A*M(A*M)_t+\|A\|^2|M_t|^2\label{main2}\\
&-2\int_0^t [M_sA^2K_s+\|A\|^2M_s K_s]\cdot dB_s.
\end{align*} 
\end{corollary} 
\begin{proof}
Notice that by \eqref{square3}, 
\begin{align*}
|(A*M)_t|^2&\leq 2(A*(M(A*M))_t\\
&-2\int_0^t [M_sA^2K_s]\cdot dB_s
+\|A\|^2 \int_0^t |K_s|^2 ds
\end{align*}
Applying \eqref{square} with $Y_t=M_t$ gives  
\begin{equation*} 
 \int_0^t |K_s|^2 ds=M_t^2-2\int_0^t M_s K_s\cdot dB_s
\end{equation*} 
and it follows that 
\begin{align*}
|(A*M)_t|^2 &\leq 2A*(M(A*M))_t-2\int_0^t [M_sA^2K_s]\cdot dB_s+\|A\|^2[M_t^2\\
&-2\int_0^t M_s K_s\cdot dB_s]\nonumber\\
&=2A*(M(A*M))_t+\|A\|^2M_t^2\\
&-2\int_0^t [M_sA^2K_s+\|A\|^2M_s K_s]\cdot dB_s,\nonumber
\end{align*} 
which gives competes the proof.  
\end{proof} 

\section{Martingale transforms} 

\subsection{Cotlar  martingale transforms}\label{TheKeyConnection}

\begin{definition} Supposed $d=2n$ is even and $A=(a_{ij})$ is a $d\times d$ real matrix.    We will say $A$ is a Cotlar matrix if it has the following properties: (1) $Av\cdot v=0$ for all $v\in \R^d$ and (2) $A^2=-I$, where $I$ denotes the identity matrix.  A  martingale transform by a Cotlar matrix will be called a {\it Cotlar  martingale transform}.
\end{definition} 
Note that since $\det(A^2)=(\det(A))^2=(-1)^d$, there are no such $d\times d$ matrices when $d$ is odd. 

The Cotlar martingale identity  now follows from Theorem \ref{CorlarMain}. 
 \begin{corollary}[Cotlar's martingale identity]\label{MarcCotlar1}
 Suppose $d$ is even and $A$ is a Cotlar matrix. Let $M_t$ be a martingale as in \eqref{StochasticMartin}. 
Then
\begin{align}\label{marcot1}
|A*M_t|^2=2A*(M\left(A*M)\right)_t+M_t^2. 
\end{align} 
\end{corollary}

\begin{proof} Since $A$ is a Cotlar matrix, 
\begin{align*}
\int_0^t |AK_s|^2 ds-2\int_0^t [M_sA^2K_s]\cdot dB_s =\int_0^t |K_s|^2 ds+2\int_0^t [M_sK_s]\cdot dB_s=M_t^2, 
\end{align*}
by \eqref{square}.   This completes the proof. 
\end{proof} 

Since $2A*(M\left(A*M)\right)_t$ is a martingale, its expectation is $0$.  Thus $\bE|A*M_t|^2=\bE|M_t|^2$.  From this and \eqref{marcot1} the exact same induction proof as in  \eqref{sharpH} gives  

  \begin{theorem}\label{Th:CotEven} Suppose $d=2n$ is even and A is a Cotlar matrix. Then,  
 \begin{align}\label{martcot2}
\|A*M_t\|_p\leq 
\cot\left(\frac{\pi}{2 p^*}\right)\|M_t\|_p, \,\,\,  p = 2^k \,\, \text{or}\,\, p=\frac{2^k}{2^-1},\,\,  k\in \bN.
   \end{align} 
 \end{theorem}

\begin{remark}
For the applications to Riesz transforms on $\R^d$ we will need to work with $(d+1)\times (d+1)$ matrices.  In order to make the connection to Cotlar matrices we need to restrict to odd dimensions so that $d+1$ is even. 
\end{remark} 

Here are some examples.  When  $d=1$,  the following Cotlar matrix 
 \begin{align}
{\BH_2}=  \left[\begin{matrix} 0 & 1 \\ -1 & 0 \end{matrix}\right]\label{CotlarHilber2} 
\end{align}
gives the probabilisitic representation of the Hilbert transform on $\R$. 

When $d=3$  the following three  $4\times4$  Cotlar matrices   
\begin{align}\label{d=3:cotlar}
\BH_4^1=\bmatrix
0&1&0&0\\
-1&0&0&0\\
0&0&0&1\\
0&0&-1&0
\endbmatrix ,\,\,  
\BH_4^2=\bmatrix
0&0&1&0\\
0&0&0&-1\\
-1&0&0&0\\
0&1&0&0
\endbmatrix, \,\,  
\BH_4^3=\bmatrix
0&0&0&1\\
0&0&-1&0\\
0&1&0&0\\
-1&0&0&0
\endbmatrix,  
\end{align} 
give the probabilisitic representation of the 
 three Riesz transforms: $R_1, R_2, R_3$ on $\R^3$, respectively. They also have the property that $\ip{H_4^jv, H_4^kv}=0$, for all $j\ne k$ and all vectors $v\in \bR^4$, which gives a complex structure on $\bC^2$. 
 
 In general,  suppose $d$ is odd so  that $d+1=2n$. Then  the following Cotlar matrix 
\begin{align}\label{HilbertMatrix}
\BH_{2n}=
\bmatrix
0&1&0&0&\cdots&0&0\\
-1&0&0&0&\cdots&0&0\\
0&0&0&1&\cdots&0&0\\
0&0&-1&0&\cdots&0&0\\
\vdots&\vdots&\vdots&\vdots&\ddots&\vdots&\vdots\\
0&0&0&0&\cdots&0&1\\
0&0&0&0&\cdots&-1&0
\endbmatrix  
=\mathrm{diag}(\,\underbrace{\BH_2, \BH_2,\dots,\BH_2}_{n-\text{times}}\,),
\end{align}
corresponds to the first Riesz   transform $R_1$ on  $\R^d$.  The other $(d-1)$ Riesz transforms arise from orthogonal rotations of this matrix.   

\subsection{Conformal martingales} 
Before we make the connection] to Riesz transforms on $\R^d$, we present  two other (and more general) version of Theorem \ref{Th:CotEven}. 

\begin{definition}[Conformal martingale]\label{def:conformal} Let $(\Omega,\cF, \{\cF_t, t\geq 0\},\bP)$ be a filtered probability space with the usual conditions; the filtration is right continuous and $\cF_0$ contains all sets of probability $0$. 
\begin{itemize}
\item [(i)]
A $\bC$-valued continuous local martingale $Z=X+iY$ is called \emph{conformal} if
\[
\ip{X}_t=\ip{Y}_t\,\,\,\, \text{and}\,\,\, 
\ip{X,Y}_t=0, \,\,\, 
\text{for all }t\ge 0.
\]
\item[(ii)]
For  any $n\geq 1$, a  $\bC^n$-valued continuous local martingale $Z=(Z^1,\dots,Z^n)$ is called conformal
if  $Z^j$ is conformal for every $j=1, \dots, n$.
\item[(iii)] A $\bC^n$-valued conformal local  martingale is said to satisfy the orthogonality property if 
\begin{align}\label{off:diagonal}
 \ip{Z^k_t, Z^j_t}=0,\,\,\,  \text{for all}\,\,\,  j,\,  k. 
 \end{align}
\end{itemize}
Note that we always have \eqref{off:diagonal} for $j=k$ for any $\bC^n$-valued conformal martingale
\end{definition}
\begin{remark} In many of the papers in the literature on sharp martingale inequalities and applications to singular integrals starting with \cite{BW95}, and \cite{BanuelosWang1996}, conformal martingales with the extra cross orthogonality condition  \eqref{off:diagonal} are  simple called orthogonal martingales.  

Since all the martingales here are assumed to be $L^p$-bounded we will omit the local martingale terminology and simple use martingale. 
\end{remark}

Conformal martingales and their connections to analytic functions in the plane and $\bC^n$, in general,  have been extensively studied in the literature. For some of these literature and applications, see \cite{GetoorSharpe1972, Uboe1986, Dur, Muller2020}. 

Throughout the rest of this section we assume that all our martingales start at zero. Notice that if $A$ is a Cotlar matrix  the martingale $Z_t=M_t+iA*M_t$ on the Brownian filtration is a conformal martingale.   The proof of \eqref{marcot1} and that of Cotlar's original identity \eqref{CotlarId} using analytic  functions immediately give a similar identity for $\bC$-valued conformal martingales.

We start with a lemma connecting conformal martingales to analytic functions on $\bC$. The  lemma is a special  case of Proposition 5.4 in \cite[pg. 291]{GetoorSharpe1972} valid for holomorphic functions on $\bC^n$. For our needs here it suffices to take of $f:\bC\to \bC$ given by $F(z)=z^2$.  However, the proof of the case $z^2$ is really the same as the proof for  the case $z^k$ and for completeness we give the proof. 

\begin{lemma}\label{thm:power-conformal}
If $Z=X+iY$ is a $\bC$-valued conformal martingale and $k\in\{1,2,3,\dots\}$. Then 
\[
Z^{(k)}_t := (Z_t)^k
\]
is again a continuous conformal martingale. Writing $Z^{(k)}_t=U_t+iV_t$, we have
\[
\ip{U}_t=\ip{V}_t,\qquad \ip{U,V}_t=0.
\]
Moreover,
\[
d\ip{Z^{(k)}}_t = k^2 |Z_t|^{2k-2}\, dA_t,
\quad\text{where } \quad A_t=\ip{X}_t=\ip{Y}_t.
\]
\end{lemma}
 \begin{proof} 
 Set $F(z)=z^k$. Since $f$ is analytic we write $F(z)=f(x, y)=U(x, y)+iV(x, y)$, where $U$ and $V$ are conjugate harmonic functions satisfying the Cauchy-Riemann equations. 
It follows from It\^o's  formula that both 
$U_t=U(X_t, V_t)$ and $V_t=V(X_t, Y_t)$ are martingales. Applying \eqref{product} and \eqref{square} with these two martingales gives: 
\begin{align*}
d\brk{U}_t = |\nabla U|^2 dA_t,\quad  d\brk{V}_t = |\nabla V|^2\, dA_t,
\end{align*} 
\[d\ip{U,V}_t = (U_xV_x+U_yV_y)\, dA_t.
\]
By the Cauchy--Riemann equations, 
\[|\nabla U|^2=|\nabla V|^2
\]  
\[
U_xV_x+U_yV_y = U_x(-U_y)+U_y(U_x) = 0 
\]
Thus $\brk{U}_t=\brk{V}_t$ and $\ip{U, V}_t=0.$ Recalling that $|\nabla U(z)|^2=|F'(z)|^2$ completes the proof of the Lemma. 
\end{proof} 
\begin{definition}
 For any $t>0$ and  $1<p<\infty$, define
\begin{align}\label{Pic:Constant}
\beta_p := \sup \Big\{\frac{\|Y_t\|_p}{\|X_t\|_p}\Big\}, 
\end{align}
where the $\sup$ is taken over all pairs $(X, Y)$ of martingales such that $Z=X+iY$ is a conformal martingale.
\end{definition}
 Since $\ip{X}=\ip{Y}$ for all conformal pairs, $\beta_2=1$.   The fact that $\beta_p$ is finite follows from any one of the multiple versions of the the classical Burkholder-Gundy inequalities.
 
  Our interest here is in showing that  
 \[\beta_p\leq \cot\left(\frac{\pi}{2 p}\right), \quad   p = 2^k,  \,\, k\in \N.  
 \] 
By Lemma \ref{thm:power-conformal},  
\begin{align}\label{Case:k2}
Z^2=X^2-Y^2+2iXY=U+iV
\end{align}
 is a $\bC$-valued conformal martingale with real and imaginary parts given by
\[
 U_t=X_t^{\,2}-Y_t^{\,2}, \,\,\,\, V_t=2X_tY_t, 
\]
\begin{align*}X_t^2 - Y_t^2 = 2 \int_0^t \big( X_s\, dX_s - Y_s\, dY_s \big), \,\,\,\, 
2X_tY_t = 2 \int_0^t \big( X_s\, dY_s + Y_s\, dX_s \big).
\end{align*}
In this setting,
\begin{equation}\label{eq:cotlar-scalar}
Y_t^{\,2} = X_t^{\,2} - \Re{Z_t^2}
\end{equation}
is the needed Cotlar identity.

Note that without conformality we would have 
\[Z_t^2=2\int_0^t Z_s dZ_s+\ip{Z}_t, 
\]
with $\ip{Z_t}\neq 0$ and   $Z^2$ would not be a martingale. 

\begin{theorem}\label{C^2:version} Let $Z=X+iY$ be a conformal martingale. Then 
\begin{align}\label{eq:Conformal}
\|Y\|_p\leq 
\cot\left(\frac{\pi}{2 p}\right)\|X\|_p, \,\,\,  p = 2^k,\,\,\,  n\in \bN.
   \end{align} 
\end{theorem}
\begin{proof}  

Suppose we can show that for every $p> 1$,
\begin{equation}\label{eq:recurrence}
\beta_{2p} \le \beta_p + \sqrt{1+\beta_p^2}, 
\end{equation}
where $\beta_p$ is the constant in \eqref{Pic:Constant}. Since  $\bE(Y_t^2)=\bE\ip{Y}_t=\bE\ip{X}_t)=\bE(X_t^2)$, it follows that $\beta_2=1$.  The induction argument as  before gives that for $p=2^k,$ 
$\beta_{p}\le \cot\!\left(\frac{\pi}{2p}\right),$
and \eqref{eq:Conformal} follows. 

To show \eqref{eq:recurrence},  applying the definition of $\beta_p$ to the pair $(U, V)$ in \eqref{Case:k2} we have 
\[
\|X_t^2-Y_t^2\|_p \le \beta_p\,\|2X_tY_t\|_p.
\]
Writing $$Y_t^2=X_t^2-(X_t^2-Y_t^2)$$ and applying  the Minkowski  and Cauchy–Schwarz  inequalities, exactly as in Cotlar's proof, we obtain 
\begin{align*}
\|Y_t\|^2_{2p}=\|Y_t^2\|_p & \leq \|X_t^2\|_p+\beta_p\|2X_tY_t\|_p\\
& \leq \|X_t^2\|_p+2\beta_p\|Y_t\|_{2p}\|X_t\|_{2p}\\
&\le \|X_t^2\|_p+2\beta_p\beta_{2p}\|X_t\|_{2p}^2\\
&= \|X_t\|_{2p}^2+2\beta_p\beta_{2p}\|X_t\|_{2p}^2. 
\end{align*}
Hence, 
\[\frac{\|Y_t\|^2_{2p}}{\|X_t\|_{2p}^2}\leq 1+2\beta_p\beta_{2p},
\]
which again together with the definition of $\beta_p$ gives 
\[\beta_{2p}^2\leq 1+2\beta_p\beta_{2p}. 
\] 
This competes the proof of  \eqref{eq:recurrence} and \eqref{eq:Conformal} follows.

\end{proof} 
The inequality \eqref{eq:Conformal} is a especial  ase of the inequality in \cite{BW95} valid for $1<p<\infty$. 
We end this section by pointing out the version of Theorem \ref{C^2:version} for $\bC^n$-valued conformal martingales for any $n\geq 1$.  This directly follows from the next lemma. 

\begin{lemma}\label{Corr:v-valued} Suppose $Z=(Z^1, \dots, Z^n)\in \bC^n$  is a conformal martingale with the orthogonality property \eqref{off:diagonal}.
Set $Z^k=X^k+iY^k$.  Then 
\[\Gamma_t :=Z_t\cdot Z_t=\sum_{k=1}^n (Z^k_t)^2=\sum_{k=1}^n [(X^k)^2-(Y^k)^2+2iX^kY^k]=U_t+iV_t
\]
is a complex-valued conformal martingale with 
\begin{align}\label{eq:cotlar-cector}
U_t=\|X_t\|_{\ell^2}^2-\|Y_t\|_{\ell^2}^2
\end{align} 
and 
\begin{align}
V_t=2X_t\cdot Y_t&=2\sum_{k=1}^n  \int_0^t \big( X_s^k\, dY_s^k + Y_s^k\, dX_s^k \big)\\
&=2\int_0^t (X_s\cdot dY_s+Y_s\cdot dX_s).\nonumber
\end{align} 
 \end{lemma} 
 \begin{proof}
 Since the function $F(z)=z_1^2+\dots+z_k^2$ is holomorphic, it follows from It\^o's formula that $\Gamma$ is a complex-valued conformal martingale. This is again  a special case of  Proposition 5.4 in \cite[pg. 291]{GetoorSharpe1972} which states that if $F:\bC^n\to\bC$ is holomorphic and $Z=(Z^1, \dots, Z^n)$ is a $\bC^n$-valued conformal martingale satisfying \eqref{off:diagonal}, then $F(Z)$ is a $\bC$-valued conformal martingale. 
 
 However, verifying $F(Z)$ is a conformal martingale for our function $F$ is very simple without appealing to the general result.  Indeed, by Lemma \ref{thm:power-conformal}, $(Z^k)^2$ is a martingale and hence so is their sum $\Gamma$.  It remains to show that $\Gamma$ is conformal. Computing the covariance and using  \eqref{off:diagonal},

\begin{align*}
\ip{\Gamma}_t
= \ip{\sum_{j=1}^n (Z^j)^2,\ \sum_{k=1}^n (Z^k)^2}_t
&= \sum_{j,k=1}^n \ip{(Z^j)^2,\, (Z^k)^2}_t\\
&=4\sum_{j,k=1}^n Z^j_t\, Z^k_t\, \ip{Z^j \, Z^k}_t.\\
&=0. 
\end{align*} 
\end{proof} 
\begin{remark}
For a concrete example of a $\bC^n$, $n=2d$, valued conformal martingales  with the orthogonal  property that  arises from a Beurling-Ahlfors operators on $\R^d$, see \eqref{concrete:ecample} and the discussions that follows. 
\end{remark} 
  From Lemma \ref{Corr:v-valued} and the same proof as that of Theorem \ref{C^2:version} we have the vector-valued version. 
\begin{theorem}\label{thm:dyadic-Cn}
Let $Z=X+iY$  conformal $\bC^n$-valued martingale with the orthogonality property.  
Then for $p=2^k, k\in \bN$,
\[
\|\,\|Y\|_{\ell^2}\|_p
\le \cot\Big(\frac{\pi}{2p}\Big)\ \|\,\|X\|_{\ell^2}\,\|_p. 
\]
\end{theorem}

\subsection{Riesz transforms as projections of Cotlar  martingale transforms}
Before we return to the Cotlar-Hilbert martingale transforms, we explain how the martingale transform on $\R^d$ arise as projections operators, equivalently conditional expectations,  of the martingale transforms.  This is the celebrated work of Gundy-Varopoulos \cite{GV79}  that has been so extensively used in the literature.  Here,  we use the notation in \cite{Ban, Ban86}.  Let $B_t=(X_t, Y_t)=(X_t^1, \dots , X_t^d, Y_t)$ be $(d+1)$ dimensional Brownian  motion in the upper half-space of $\R_+^{d+1}=\R^d\times (0, \infty)$ starting with the Lebesgue measure  on the hyperplane $\{(x, y): x\in \R^d, y=a\}$, $a>0$. Let $\tau=\inf\{t>0: Y_t=0\}$ be its exit time. We identify $B_{\tau}=(X_{\tau}, 0)$ with $X_{\tau}$. 
If $\mathbb P_{(x, a)}$ is  the probability measure associated with $B_t=(X_t, Y_t)$ starting at the point $(x,a)$ with $x\in \R^d$ and $a>0$, define the  measure  $\bP^a$  by 
\[
\bP^a(B_{t \wedge \tau})\in \Theta)=\int_{\R^d} \mathbb P_{(x,a)}(B_{t \wedge \tau}\in \Theta)d\mu(x),
\]
for any Borel set $\Theta\in  \R^d\times \R^+$. In particular, for any Borel set $\Theta\subset \R^d$, $\mathbb P_y(X_{\tau}\in \Theta)=m(\Theta)$, where $m$ is the Lebesgue  measure on $\R^d$. Hence, 
\begin{equation}\label{eq:3.3.2}
\bE^a(f(X_{\tau}))=\int_{\R^d} f(x)d\mu(x).
\end{equation} 

In the same way, by independence, the transition probability of $\{B_t; \tau>t\}$ is the product of the heat kernel in $\R^d$ and the heat kernel for the half-line $(0, \infty)$.  Integrating away the heat kernel in the $x$-variable (since the Brownian motion has  the Lebesgue measure as its initial distribution) and computing the
Green's function for the half-line  gives 
\begin{equation}\label{eq:3.3.3}
\bE^{a} \int_0^{\tau} F(B_s)\,ds
=\int_0^{\infty} \bE^a[F(B_s); \tau>t]=
2 \int_0^\infty \int_{\mathbb{R}^d} (y \wedge a)\, F(x,y)\,dx\,dy,
\end{equation}
for all nonnegative functions $F$ on $\R^{d+1}_+$. Both
\eqref{eq:3.3.2} and \eqref{eq:3.3.3} continue to hold for those $f$ and $F$
for which the integrals are finite. 

At this point we could apply the martingale inequalities with respect to  the probability measure of Brownian motion starting at the $(x, a)$ and then let $a\to \infty$ to get our results for their projections. This is done in \cite{Ban86}.  Here we will proceed as in the original paper of Gundy and Varopoulos \cite{GV79} where they  used a time-reversal argument to let $a\to\infty$ and construct a filtered probability space and a process
$\{B_t=(X_t, Y_t)\}$ indexed by $t \in (-\infty,0]$, which they called the {\it background
radiation process.} Heuristically speaking, the paths of $B_t$ are
Brownian paths which originate from "$\{a=\infty\}$"  at time $t=-\infty$
and exit $\mathbb{R}^{d+1}_+$ at time $t=0$ with Lebesgue measure as their 
distribution. Letting $\bE$ be the expectation with respect to the
measure associated with the background radiation process, the
identities \eqref{eq:3.3.2} and \eqref{eq:3.3.3} become, respectively, 
\begin{equation}\label{eq:3.3.4}
\bE f(B_0) = \int_{\mathbb{R}^d} f(x)\,dx
\end{equation}
and
\begin{equation}\label{eq:3.3.5}
\bE \int_{-\infty}^0 F(B_s)\,ds
=
2 \int_0^\infty \int_{\mathbb{R}^d} y F(x,y)\,dx\,dy.  
\end{equation}

For $f\in C^{\infty}_0(\R^d)$, let $U_f(x, y)=P_yf(x)$ be the convolution of $f$ with the Poisson kernel 
\[P_y(x)=\frac{c_d\, y}{\left(|x|^2+|y|^2\right)^{\frac{d+1}{2}}}, 
\]
where $c_d$ is the constant as in the  definition of the Riesz transforms \eqref{RieszT}.
We  set  
\begin{align*}
 \nabla U_f(x, y)
&=\left(\partial_yU_f(x, y), \partial_{x_1}U_f(x, y), \dots, \partial_{x_d}U_f(x, y)\right)\\
&=\left(\partial_y{P_yf(x)}, \partial_{x_1}{{P_yf(x)}, \dots, \partial_{x_d}{P_yf(x)},}\right)
\end{align*}
and 
$$
dB_s=\left(dY_s, dX_s^1, \dots, dX_s^{d}\right). 
$$
Apply It\^o's formula to get the martingale 
\[M_t^f=U_f(X_t, Y_t)=\int_{-\infty}^t\nabla U_f(X_s, Y_s)\cdot dB_s, \quad t\in (-\infty, 0). 
\] Note that  
\begin{equation}\label{eq:f(B_0} M_0^f=f(B_0)=\int_{-\infty}^0\nabla U_f(X_s ,Y_s)\cdot dB_s. 
\end{equation} 
 For any $(d+1)\times (d+1)$ matrix $A$ define the martingale transform  
\begin{align*}
A*M^f_t=\int_{-\infty}^tA\nabla U_f(X_s, Y_s)\cdot dB_s, \quad t\in (-\infty, 0]
\end{align*}
and its projection operator (conditional expectation) on $\R^d$ by 
\begin{align}\label{projection:TA}
T_Af(x)=\bE[A*M^f_0\, \big| B_0=(x, 0)]=\bE\left(\int_{-\infty}^0A\nabla U_f(X_s , Y_s)\cdot dB_s\, \big| X_0=x\right). 
\end{align}

Since the conditional expectation is a contraction on $L^p$ for any $1< p<\infty$ and the distribution of $B_0$ is the Lebesgue measure, we immediately have 
\begin{align}\label{Mart:CotlarProjection} 
\|T_Af\|_{p}\leq \left(\bE|A*M_0^f|^p\right)^{1/p}=\|A*M^f_0\|_p\,\,\, 1< p<\infty.
\end{align}
From this and  Theorem \ref{Th:CotEven} we have 

\begin{corollary}\label{Cor:Cot}
Suppose $d$ is odd and $A$ is a $(d+1)\times (d+1)$  Cotlar matrix.  Then  
\begin{align}\label{eq:Cotlar3}
\|T_Af\|_{p}\leq \cot\left(\frac{\pi}{2 p^*}\right)\|f\|_{p}, \,\,\,  p = 2^k\,\,  \text{or}\,\,  p=\frac{2^k}{2^n-1}, n\in \bN.
\end{align} 
\end{corollary} 
 
\begin{remark}\label{Mart:CotlarHilbertProjection} An important distinction for the case $d=1$ comes from the fact that the Cauchy-Riemann equations give 
$$\BH_2\nabla U_f =\nabla U_{Hf}.$$ Consequently, 
\begin{align} 
\BH_2*M_0^f=M^{Hf}_0=Hf(B_0). 
\end{align} 
Thus the conditional expectation plays no role in this case and equality holds in \eqref{Mart:CotlarProjection}. 
\end{remark} 

For any $d>1$,  consider the  $(d+1)\times (d+1)$ matrices $\BR_j$,   $j=1, 2, \dots, d$,  defined by 
 \begin{align}\label{Rieszd=2}
\BR_{j} = [a^{j}_{lm}] = 
\begin{cases}
    1,  & l=1, m=j+1 \\
    -1, & l=j+1, m=1 \\
    0,  & \text{otherwise}, 
\end{cases}
\end{align} 
and for $j,k = 1,2,\dots,d,\ j\neq k$,  the matrices  $\BR_{(j, k)}$ defined by 
\[\BR_{(j, k)} = \big[a^{(j, k)}_{\ell m}\big]=
\begin{cases}
-1, & \ell = k+1,\ m = j+1,\\
-1, & \ell = j+1,\ m = k+1,\\
0,  & \text{otherwise}. 
\end{cases}
\]
\begin{example}
 For example, with $d=3$, we have the following: 
\begin{align}\label{d=3}
\BR_1=\bmatrix
0&1&0&0\\
-1&0&0&0\\
0&0&0&0\\
0&0&0&0
\endbmatrix ,
\BR_2=\bmatrix
0&0&1&0\\
0&0&0&0\\
-1&0&0&0\\
0&0&0&0
\endbmatrix,  
\BR_3=\bmatrix
0&0&0&1\\
0&0&0&0\\
0&0&0&0\\
-1&0&0&0. 
\endbmatrix   
\end{align}
and 
\[
\BR_{(1, 2)}=
\bmatrix
0&0&0&0\\
0&0&-1&0\\
0&-1&0&0\\
0&0&0&0
\endbmatrix,  
\BR_{(1, 3)}=
\bmatrix
0&0&0&0\\
0&0&0&-1\\
0&0&0&0\\
0&-1&0&0
\endbmatrix, 
\BR_{(2, 3)}=
\bmatrix
0&0&0&0\\
0&0&0&0\\
0&0&0&-1\\
0&0&-1&0
\endbmatrix. 
\]
 In Addison, $\BR_{(1,2)}=\BR_{(2,1)}$,  $\BR_{(1, 3)}=\BR_{(3, 1)}$ and $\BR_{(2, 3)}=\BR_{(3, 2)}$. 
\end{example}
The above matrices can be decompose as the sum of two matrices  in the form of  
\[\BR_j=\BR_j^1+\BR_j^2,\,\,\,  \text{and}\,\,\, \BR_{(j, k)}=\BR_{(j, k)}^1+\BR_{(j, k)}^2
\]   where, 
\begin{align}
\BR_j^1=\begin{cases}
    1,  & l=1, m=j+1 \\
    0,  & \text{otherwise}, 
\end{cases}, \hskip.2cm  \BR_{j}^2=
\begin{cases}
    -1, & l=j+1, m=1 \\
    0,  & \text{otherwise}, 
\end{cases}
\end{align}
and 
\begin{align}\label{RiszMatr2}
\BR_{(j, k)}^1 =
\begin{cases}
-1, & \ell = k+1,\ m = j+1,\\
0,  & \text{otherwise},
\end{cases}, 
\end{align}
\begin{align}\label{RieszMatr3}
\BR_{(j, k)}^2 =
\begin{cases}
-1, & \ell = j+1,\ m = k+1,\\
0,  & \text{otherwise}.
\end{cases}
\end{align}
 
For $\{j, k =1, \dots, d\}$, the second order Riesz transforms are the Fourier multiplier  operators with 
\[\widehat{R_{(j, k)}f}(\xi)=\frac{-\xi_1 \xi_2}{|\xi|^2}\widehat{f}(\xi).\]

\begin{lemma}\label{Lemma:Cotlar} For all $f\in C_0^{\infty}(\R^d)$ the following hold. \\ 
\begin{itemize}
\item[(i)]  For all 
$j\in \{1, \dots, d\}$, 
\[T_{\BR_j^1}f=T_{\BR_j^2}f=\frac{1}{2}R_jf. \quad \text{Hence}\quad T_{\BR_j}f=R_jf.\]
\item[(ii)] For all $j\neq k \in \{1, \dots, d\},$
 \[T_{\BR_{(j, k)}^1}f=T_{\BR_{(j, k)}^2}f=\frac{1}{2}R_{(j, k)}f.  
 \quad \text{Hence}\quad T_{\BR_{(j, k)}}f=R_{(j, k)}f.
 \]
 \end{itemize}
\end{lemma}

\begin{proof} This follows from Gundy-Varopoulos \cite{GV79}. Here we give the proof of (i)  following the computation 
in \cite[pg. ~817]{Ban}. The proof of (ii) is exactly the same. Recall that 
\[ \widehat{P_yf}(\xi)=e^{-2\pi y|\xi|}\widehat{f}(\xi)\,\,\, \text{and}\,\,\, \widehat{\partial_{x_j}f}(\xi)=-2i\pi \xi_j\widehat{f}(\xi)
\]
Let  $g$ be another smooth function of compact support.  Set 
\[N=g(B_0)=\int_{-\infty}^0\nabla U_g(X_s ,Y_s)\cdot dB_s\] and 
\begin{align*} 
M=\BR_j^1*M_0^f&=\int_{-\infty}^0\BR_j^1\nabla U_f(X_s , Y_s)\cdot dB_s=\int_{-\infty}^0\partial_{x_j}U_f(X_s, Y_s)dY_s. 
\end{align*} 

By \eqref{product} and \eqref{eq:3.3.5} we have 
\begin{align*}\bE[NM]&=\bE\ip{N, M}=\bE\int_{-\infty}^0 \partial_{x_j}U_f(X_s, Y_s)\partial_{y}U_g(x, y)ds\\
&=\int^{\infty}_0\int_{\R^d} 2y\,  \partial_{x_1}U_f(x, y)\partial_{y}U_g(x, y)dx\,dy. 
\end{align*}
This gives, 
\begin{align*}
\int_{\R^d} g(x)T_{\BR_j^1}f(x)dx&=\bE[g(B_0)T_{\BR_j^1}f(B_0)]\\
&=\bE\left(\bE\left( g(B_0)\int_{-\infty}^0\partial_{x_j}U_f(X_s, Y_s)dY_s\,\right) \Big| B_0\right)\\
&=\bE[NM]=\bE\int_{-\infty}^0 \partial_{x_j}U_f(X_s, Y_s)\partial_{y}U_g(x, y)ds\\
&=\int^{\infty}_0\int_{\R^d} 2y\,  \partial_{x_1}U_f(x, y)\partial_{y}U_g(x, y)dx\,dy\\
&=2\int_0^{\infty}y\, \int_{\R^d}\widehat{{\partial_{x_j}U_f(\xi, y)}}\overline{\widehat{\partial_{y}U_f(\xi, y),}}\, d\xi,\\
 &=8\pi^2 \int_{\mathbb{R}^d}
i\xi_j |\xi|\, \widehat{f}(\xi)\,\overline{\widehat{g}}(\xi)
\left( \int_{0}^{\infty} y e^{-4\pi y|\xi|}\, dy \right)
d\xi 
\widehat{R_j f}(\xi)\,\overline{\widehat{g}}(\xi)\, d\xi\\
&= \frac{1}{2} \int_{\mathbb{R}^d}
R_j f(x)\, g(x)\, dx .
\end{align*} 
This shows that $T_{\BR_j^1}f=\frac{1}{2}R_jf$ and competes the proof of  (i). The proof of (ii) is exactly the same. 
\end{proof} 
\begin{remark} The use of the background radiation process can be entirely avoided by defining the  operators 
\[\cT^{a}_{A}f(x)=\bE^{a}\left(\int_{0}^\tau A\nabla U_f(X_s, Y_s)\cdot dB_s\big| X_{\tau}=x\right). 
\]
Taking $A=\BR_j$ and integrating against another smooth function as above and letting $a\to\infty$ shows that $\cT^{a}_{\BR_j}f\to R_jf$ in $L^2$. Since all the $L^p$-bounds of the operators $\cT^{a}_{\BR_j}$ are independent of $a$, they remain valid for $R_j$. For more details on this construction, see \cite{Ban86}. 
\end{remark} 

The lemma immediately gives the following 
\begin{corollary}
Consider the $(d+1)\times (d+1)$ matrix 
\[\tilde{\BR}_{(j, k)} =\big[\tilde{a}^{(j,k)}_{\ell m}\big]=
\begin{cases}
1, & \ell = j+1,\ m = k+1,\\
-1, & \ell = k+1,\ m = j+1\\
0,  & \text{otherwise},
\end{cases}
\]
then 
\begin{align}\label{Proj-0}
T_{\tilde{\BR}_{(j, k)}}f=\frac{1}{2}\left(R_{(j, k)}f-R_{(k, j)}f\right)=0.
\end{align}    
\end{corollary}

\begin{theorem}\label{CorOdd} Suppose $d>2$ is odd so that $d+1=2n$ for and integer  $n>1$. Let $R_1$ be the first Riesz transform on $\R^d$.  Consider the Cotlar matrix $\BH_{2n}$ given in \eqref{HilbertMatrix}.  Then the projection of the Cotlar martingale transform $\BH_{2n}*M_t^f$ is $R_1f$. That is,   
\begin{align}\label{Main:Result}
T_{\BH_{2n}}f=R_1f
\end{align}
and 
\begin{align}\label{Ineq:Riesz}
\|R_1f\|_p\leq \cot\left(\frac{\pi}{2 p^*}\right)\|f\|_p, \,\, p = 2^k\,\,  \text{or}\,\,  p=\frac{2^k}{2^k-1}, k\in \bN.
\end{align}
\end{theorem}

\begin{proof} Observe that the matrix $H_{2n}$ is the sum of the elementary skew symmetric matrices generating the
$2\times2$ blocks on the index pairs $(1,2)$, $(3,4)$, $\dots$, $(2n-1,2n)$.  For $k=1,\dots,n-1$, define  as above  the $(d+1)\times (d+1)$ matrix 
$\widetilde{\mathbf R}_{(2k, 2k+1)}$ by
\[\widetilde{\mathbf R}_{(2k, 2k+1)}=
\big[\tilde{a}^{(2k, 2k+1)}_{\ell m}\big]=
\begin{cases}
1,  & \ell=2k+1, m=2k+2),\\[2pt]
-1, & \ell=(2k+2,m=2k+1),\\[2pt]
0,  & \text{otherwise}.
\end{cases}
\]
That is, $\widetilde{\mathbf R}_{(2k, 2k+1)}$ is zero except for a single $2\times2$
skew-symmetric block acting on rows and columns $2k+1$ and $2k+2$. Then 
$\BH_{2n}
=\mathbf R_1+
\sum_{k=1}^{\,n-1} \widetilde{\mathbf R}_{(2k, 2k+1)}$ and 
\begin{align}
\BH_{2n}*M_t^f=\BR_1*M_t^f+\sum_{k=1}^{\,n-1} \widetilde{\mathbf R}_{(2k, 2k+1)} *M_t^f. 
\end{align} 
Taking conditional expectation and employing  \eqref{Proj-0} we obtain 
\begin{align}\label{THE:CID}
T_{\BH_{2n}}f(x)=R_1f(x)+\sum_{k=1}^{\,n-1} T_{\widetilde{\mathbf R}_{(2k, 2k+1)}} f(x)= R_1f(x). 
\end{align}
The bound in \eqref{Ineq:Riesz} follows from Corollary \ref{Cor:Cot}, completing   the proof of the theorem. 
\end{proof} 
\begin{example}\label{Cot:Proj4D} Suppose $d=3$ and consider the Cotlar matrices $\BH_4^1, \BH_4^2, \BH_4^3$ given in \eqref{d=3:cotlar}.  The above construction gives:
\begin{enumerate} 
\item $T_{\BH_4^1}=R_1+\frac{1}{2}(R_{(2, 3)}-R_{(3, 2)})=R_1$
\item $T_{\BH_4^2}=R_2+\frac{1}{2}(R_{(1, 3)}-R_{(3, 1)})=R_2$
\item $T_{\BH_4^3}=R_3+\frac{1}{2}(R_{(1, 2)}-R_{(2, 1)})=R_3$
\end{enumerate}
\end{example} 

\begin{remark}\label{Riesz;Rotation}
Extending the  $L^p$ bound from  odd to even dimensions presents no problem. Recall that by  de Leeuw's extension theorem (\cite{DeL}, \cite[Theorem 2.5.15]{Gra}) the Riesz transforms on $\R^{d_2}$ are Fourier extensions of those on $\R^{d_1}$, $d_2>d_1$.  Hence the upper bound for $\|R_1\|_{p\to p}$  in the Corollary holds for all $d>1$.   Of course, we already knew that but here we have shown it  in the case of powers of 2 directly from the Cotlar identity. 
\end{remark}

\begin{remark}
By a rotation of the matrix $\BH_{2n}$, as in the example  for $d=3$, we can verify that $R_j$, $j=2, \dots, d$, are projection of a Cotlar martingale transforms and hence we have the same bound for $R_j$ as for $R_1$.  However, if the only thing we want is to get the same $L^p$ norm for $R_j$ as that of $R_1$, this is immediate.  Indeed, for $j=2,\dots,d$,   let $\rho$ be orthogonal matrix $\rho\in O(d)$ (the orthogonal group in $\R^d$) with $\rho e_1=e_j$
and define $(U_{\rho} f)(x)=f(\rho^{-1}x)$. Then $R_j = U_{\rho}^{-1} R_1U_{\rho}$. 
Consequently,
$\|R_j f\|_{p}=\|R_1 (f\circ \rho^{-1})\|_{p}$ and hence
$\|R_j\|_{p\to p}=\|R_1\|_{p\to p}.$
\end{remark} 

\subsection{On a theorem and a problem of R. Durrett} As we showed above, for any $d>1$ the $(d+1)\times (d+1)$ matrices $\BR_j$, $j=1, \dots, d$ in \eqref{Rieszd=2} have the property that $T_{\BR_j}=R_j$.  That is, the Riesz transforms $R_j$ are projections (conditional expectations) of the martingale transforms by the matrices  $\BR_j$ applied to the Brownian martingales obtained by composting the harmonic extensions of functions with Brownian motion in the upper half-space of $\R^{d+1}$. Motivated by Davis \cite{Davis1973} probabilistic proof of Stein and Wiess \cite{SteWei1971} theorem that the distribution of the the conjugate function on the circle, equivalently of Hilbert transform, of an indicator function of a Lebesgue measurable set depends only on measure of the set, Durrett proved the following 
 
\begin{knowntheorem}[R.~Durret \cite{Dur}]\label{Du:Thm} Let $B\in \cF_{\infty}$, $0<\P(B)<1$.  There are constants $C$ and $\gamma$ which only depend on $\bP(B)$ and the matrices so that 
\begin{align}\label{Du:Ine}
\bP\left(\sup_{j}\sup_{t}|\BR_j*1_{B}|>y\right) \geq C e^{-\gamma y}. 
\end{align}
\end{knowntheorem}

\begin{problem}
The validity of such an inequality for the Riesz transforms is raised in  \cite[p.~34]{Dur}.  To the best of our knowledge this problem remains open.  
\end{problem}

 As we have already seen,  when $d$ is odd, equivalent $d+1=2n$, the Riesz transforms on $\R^d$ are given by the projection operators $T_{\BH_1}, \dots, T_{\BH_d}$ where the matrices $\BH_1, \dots, \BH_d$ are $2n\times 2n$, $d+1=2n$, and satisfy 
 \begin{enumerate}
 \item {\bf Cotlar property}\
\begin{align*}
x\cdot \BH_j x=0, \quad |\BH_j x|=|x|,  \quad\text{for all }x\in\mathbb{R}^{2n}.  
\end{align*}
\item {\bf Mutual orthogonality}
\begin{align*}
(\BH_i x)\cdot(\BH_j x)=0
\quad i\neq j,\quad\text{for all } x\in\R^{2n}.
\end{align*}
\end{enumerate}

 Using these matrices a sharp version of \eqref{Du:Ine} with an explicit constants is possible. This fact was already indicated in Durrett's paper following Davis's argument for the Stein-Weiss result. Here we also the compute the exact distribution of the martingale transform.

 Towards that end, let  $B_t=(B_t^1,\dots,B_t^{2n})$ 
be the standard $2n$-dimensional Brownian motion equipped with the standard Brownian filtration $\mathcal F_t=\sigma(B_s:0\le s\le t)$ on the probability space $(\Omega,\mathcal F,\mathbb P)$. Let $\cE\in\mathcal F_\infty$ with $0<\bP(\cE)<1$.  Consider the martingale 
\[
M_t=\bE(1_{\cE}\mid\cF_t) 
\]
which can be written as 
\begin{equation}\label{DurRep:1}M_t:=\bP(\cE)+\int_0^t K_s\cdot d B_s, 
\end{equation}
where  $K_s\in \R^{2n}$ predicable. 
To simplify notation set  $Y_t^j=(\BH{_j}*1_{\cE})_t$, the martingale transforms by the matrices $\BH{_j}$, $j=1, \dots , 2n$, and $Y_t^0=M_t-\bP(\cE)$.  
By  properties (1) and (2) the martingale 
\[
Z_t:=(Y^0, Y_t^1,\dots,Y_t^{2n})=(Y_t^0, Y_t^1, \dots, Y_t^{2n})\, \in \R^{2n+1}
\]
is an orthogonal martingale. That is,  $\ip{Y^j, Y^k}_t=0$, for all $j\neq k$ and $d\ip{Y^j, Y^j}_t=|K_t|^2 dt$.  Hence
\[
d\ip{Z}_t= |K_t|^2 I_{2n+1}\,dt .
\]
By Dambis-Dubins-Schwarz theorem (see \cite[(1), p.~78]{Dur}) there exists
an $(2n+1)$-dimensional Brownian motion $\widetilde B_t$ such that $
Z_t=\widetilde B_{\langle X\rangle_t}.
$ Moreover, if we set
$
\tau:=\langle X\rangle_\infty,
$
then
\[
(Y_\infty^1,\dots,Y_\infty^{2n})
\stackrel{d}{=}
(\widetilde B_\tau^{(1)},\dots,\widetilde B_\tau^{(2n)}),
\]
and since $M_{\infty}=1_{\cE}\in \{0, 1\}$, 
$\tau$ is the first exit time of one dimensional Brownian motion 
from the interval $(-\bP(\cE),1-\bP(\cE))$.
\begin{lemma} 
Set $\beta=\bP(\cE)$. Let $\tau$ be the first exit time of planar
Brownian motion $(B_t^1,B_t^2)$ started at $(0,0)$ from the strip
\[
S=(-\beta,\,1-\beta)\times\mathbb R.
\]
Then the 
distribution of the  vertical coordinate $B_\tau^2$ 
has density 
\begin{align}\label{DisDen:Strip}
f_\beta( \lambda)
=
\frac{\sin(\pi \beta)\cosh(\pi  \lambda)}
{\sinh^2(\pi \lambda)+\sin^2(\pi \beta)},
\qquad y\in\mathbb R.
\end{align}
which in particular depends  only on the probability of the event $\cE$.

\end{lemma}

These type of lemmas are usually proved by conformally  mapping the upper half-space,  whose Poisson kernel is  
\[\frac{1}{\pi}\frac{y}{(|x|^2+y^2)},
\]
onto the strip by the exponential map $F(z)=e^{i\pi z}$.  We leave this to the reader as an easy exercise in basic complex analysis. For examples of these  type of applications, see \cite[Theorem~3]{Davis1973} and  \cite[Lemma~2,1]{Jan2004}. 

Applying the Lemma with the pair $(Y_t^0, Y_t^j)$, $j\geq 1$, it follows that 
\begin{align}\label{Dis:HJ}
P(|Y_{\infty}^{(j)}|> \lambda)=\bP(|\widetilde B_\tau^{(j)}|> \lambda)
=
2\int_\lambda^\infty f_\beta(x)\,dx
=\frac{2}{\pi}
\arctan\!\left(
\frac{\sin(\pi \beta)}{\sinh(\pi  \lambda)}
\right).
\end{align}
For the last equality, we make the substitution 
 $t=\sinh(\pi x),$ $dt=\pi\cosh(\pi x)\,dx$ and get 
 \[
2\int_{\lambda}^\infty f_\beta(x)\,dx
=
\frac{2\sin(\pi\beta)}{\pi}
\int_{\sinh(\pi y)}^\infty \frac{dt}{t^2+(\sin(\pi\beta))^2}=\frac{2}{\pi}
\arctan\!\left(
\frac{\sin(\pi \beta)}{\sinh(\pi  \lambda)}
\right).
\]

\begin{proposition} 
For every \(0<\beta<1\) and every \( \lambda> 0\),
\begin{align}\label{Claim:1}
\frac{2\arctan 2}{\pi}\,\sin(\pi\beta)e^{-\pi \lambda}
&\;\le\;
\frac{2}{\pi}\arctan\!\left(\frac{\sin(\pi\beta)}{\sinh(\pi \lambda)}\right)\\
&
\;\le\;
\min\!\left\{1,\frac{8}{\pi}\sin(\pi\beta)e^{-\pi \lambda}\right\}\nonumber.
\end{align}
\end{proposition}
\begin{proof}
To simplify notation, set \(a=\sin(\pi\beta)\in(0,1)\) and \(t=\pi \lambda> 0\).

\medskip
\noindent
{\it Upper bound:}
Since \(\arctan(x)\le \pi/2\) for \(x\ge 0\),
\begin{equation}\label{UpperB:1}
\frac{2}{\pi}\arctan\!\left(\frac{a}{\sinh t}\right)\le 1.
\end{equation}
Moreover, for \(t>0\), \(\sinh t>0\), so \(\frac{a}{\sinh t}> 0\) and, using $\arctan(x)\leq x$, for  $x\geq 0$, we have 
\[
\frac{2}{\pi}\arctan\!\left(\frac{a}{\sinh t}\right)
\le
\frac{2}{\pi}\frac{a}{\sinh t}.
\]
If \(t\ge \tfrac12\log 2\), then
\[
\sinh t=\frac{e^t-e^{-t}}{2}
=\frac{e^t}{2}(1-e^{-2t})
\ge \frac{e^t}{4},
\]
and hence
\[
\frac{a}{\sinh t}\le 4a e^{-t}.
\]
Therefore,
\[
\frac{2}{\pi}\arctan\!\left(\frac{a}{\sinh t}\right)
\le
\frac{8}{\pi}a e^{-t}.
\]
Combining this estimate with \eqref{UpperB:1} gives
\[
\frac{2}{\pi}\arctan\!\left(\frac{a}{\sinh t}\right)
\le
\min\!\left\{1,\frac{8}{\pi}a e^{-t}\right\}.
\]
which is the righthand side of \eqref{Claim:1}

\medskip
\noindent
{\it Lower bound.}
For all \(t\ge 0\),
\[
\sinh t=\frac{e^t-e^{-t}}{2}\le \frac{e^t}{2},
\]
so
\[
\frac{a}{\sinh t}\ge 2a e^{-t}.
\]
Since \(\arctan\) is increasing,
\[
\arctan\!\left(\frac{a}{\sinh t}\right)
\ge
\arctan\!\left(2a e^{-t}\right).
\]
For \(0\le x\le 2\), the function \(u\mapsto \arctan x/x\) is decreasing, hence
\[
\arctan x \ge \frac{\arctan 2}{2}\,x.
\]
Since  \(0\le 2a e^{-t}\le 2\), we obtain
\[
\arctan\!\left(2a e^{-t}\right)
\ge
\arctan 2 \, a e^{-t}.
\]
Therefore,
\[
\frac{2}{\pi}\arctan\!\left(\frac{a}{\sinh t}\right)
\ge
\frac{2\arctan 2}{\pi}\,a e^{-t}, 
\]
proving the lefthand side of \eqref{Claim:1} and the claim.  
\end{proof}

From \eqref{Dis:HJ} and \eqref{Claim:1}, it follows that for all $\lambda>0$, 
\[
\frac{2\arctan 2}{\pi}\,\sin(\pi\beta)e^{-\pi \lambda}\leq \mathbb P(|\BH_{j}*1_{\cE}|> \lambda)
\le
\min\!\left\{1,\frac{8}{\pi}\sin(\pi\beta)e^{-\pi \lambda}\right\}\nonumber.
\]
By the reflection principle for Brownian motion, we have 
\[
\mathbb P(|\widetilde{B}_\tau|> \lambda)
\le
\mathbb P\!\left(\sup_{0\le s\le \tau}|\widetilde{B}_{\tau}|> \lambda\right)
\le
2\,\mathbb P(|\widetilde{B}_\tau|> \lambda).
\]

\begin{align}\label{Durrett:UNSup}
\frac{2\arctan 2}{\pi}\sin(\pi\beta)e^{-\pi  \lambda}&\leq \bP\left(\sup_{t>0}|(\BH_j*1_{\cE})_t|> \lambda\right)\\
&\le
2\min\!\left\{1,\frac{8}{\pi}\sin(\pi\beta)e^{-\pi \lambda}\right\}. \nonumber
\end{align}
Finally,  
\begin{align}\label{Durrett:Sup}
\frac{2\arctan 2}{\pi}\sin(\pi\beta)e^{-\pi  \lambda}&\leq \bP\left(\sup_{1\leq j\leq 2n}\sup_{t>0}|(\BH_j*1_{\cE})_t|>\lambda\right)\\
&\le
4n\min\!\left\{1,\frac{8}{\pi}\sin(\pi\beta)e^{-\pi \lambda}\right\},\nonumber
\end{align}
gives the lower bound in \eqref{Du:Ine} with explicit constants. For a single $\BH_j$ the upper bound is also independent of dimension.  independent. 

We summarize  the above in the following theorem. 

\begin{theorem} Suppose $d$ is odd and set $d+1=2n$. Let $\BH_1, \dots, \BH_d$ be the above $2n\times 2n$ matices so that the operators $T_{\BH_1}, \dots T_{\BH_d}$ correspond to the Riez transforms $R_1, \dots R_d$ as in \eqref{THE:CID}.  Then for any set $\cE\in\mathcal F_\infty$ with $0<\bP(\cE)<1$ the distribution of $(\BH_j*1_{\cE})$ is given by \eqref{Dis:HJ} and in particular, it dependents only on the $\beta=\bP({\cE})$. Furthermore, the distribution of $\sup_{t}|(\BH_j*1_{\cE})_t|$ and $\sup_{1\leq j\leq 2n}\sup_{t>0}|(\BH_j*1_{\cE})_t|$ satisfy the exponential  bounds given by \eqref{Durrett:UNSup} and \eqref{Durrett:Sup}, respectively. 
\end{theorem} 

\begin{remark}\label{Du:remark}
The argument in \cite{Dur} is based on decomposing the martingale transform as
\[
(\BR_j*1_{\cE})_t = Y_t + Z_t,
\]
where $Y_t$ is the component parallel to the martingale direction and $Z_t$ is obtained by projecting the integrand $K_s$ in \eqref{DurRep:1}  onto its orthogonal complement. Then  $Z_t$ can be represented as a time–changed Brownian motion $B_t$ run up to the random time
$
u_0=\langle Z\rangle_\infty.
$
It is then shown that $u_0$ is independent of $\sigma\left(M_t; t\geq 0\right)$
and  conditioning on $u_0$,
\[
\mathbb P\!\left(\sup_t |(\BR_j*1_{\cE})_t|>\lambda \,\middle|\, u_0\right)\geq 
\frac{1}{2}\mathbb P\!\left(\sup_{0\le u\le u_0}|B_u|>\lambda \,\middle|\, u_0\right).
\]
The right-hand side is just the usual Brownian tail with time parameter $u_0$. Thus the constant $\gamma$ in \eqref{Du:Ine} ultimately depends on the distribution of $u_0$, and hence on the law of the underlying martingale, not merely on $\P(\cE)$.

This is in contrast with the Cotlar setting, where the exponential rate equals $\pi$ and it is independent of $\cE$, reflecting the exact identity rather than the decomposition. In particular, for degenerate matrices such as $\BR_1$, $u_0$ may be arbitrarily small, so no exponential rate depending only on $\P(\cE)$ can be extracted from this argument. See \cite{Dur} for details. 
\end{remark}

As it has  already being mentioned,  in \cite[Theorem~1]{Davis1973} Davis gave a probabilistic proof of classical result of Stein and Weiss \cite[p. 240]{SteWei1971} that the distribution of the conjugate function of an  indicator function of a measurable  set $A$ on the circle depends only on its measure.  This result follows from the explicit formula given above. We give the proof here for the Hilbert transform on $\R$. 

Let $B_t=(X_t, Y_t)$ be Brownian motion in the upper half-space  of $\R^2$ starting at the point $(0, a)$, $a>0$ and let $\tau=\inf\{t: Y_t=0\}$; its first exit time. Let $A\subset \R$ be Lebesgue measurable with  $0<|A|<\infty$. Its  harmonic measure at the point  $(0, a)$ is  
\[\bP_{(0, a)}(B_{\tau}\in A)=U_{1_A}(0, a)=\frac{1}{\pi}\int_{A}\frac{a}{|s|^2+a^2}\, ds=\beta_{a}. 
\]
It has the following properties:  
\[(1)\,\,\, (\pi a) \beta_{a}\to |A| \quad  \text{and} \quad (2)\,\,   \beta_{a}\to 0,\,\, \text{both as}\,\,\,\, a\to \infty
.\]
 From \eqref{Dis:HJ} and Remark \ref{Mart:CotlarHilbertProjection} (It\^o's formula plus the Cauchy-Riemann equations)  
\begin{align}
\bP_{(0, a)}\left( |H1_A(B_{\tau})-\beta_{a}|> \lambda\right)=\frac{2}{\pi}
\arctan\!\left(
\frac{\sin(\pi \beta_{a})}{\sinh(\pi  \lambda)}
\right).
\end{align}
Fix $\lambda>0$. Since, 
\begin{align*}
&\pi a\frac{2}{\pi}
\arctan\!\left(\frac{\sin(\pi \beta_{a})}{\sinh(\pi  \lambda)}\right)\\
&=
\frac{2}{\pi}\,
\frac{\arctan\!\left(\dfrac{\sin(\pi \beta_{a})}{\sinh(\pi  \lambda)}\right)}
{\dfrac{\sin(\pi \beta_{a})}{\sinh(\pi  \lambda)}}\,
\frac{\sin(\pi \beta_{a})}{\beta_{a}}\,
\frac{1}{\sinh(\pi  \lambda)}\,\,
{\pi a}\beta_{a}, 
\end{align*}
using (1) and (2), we have, as $a\to\infty$, 
\[ \frac{\arctan\!\left(\dfrac{\sin(\pi \beta_{a})}{\sinh(\pi  \lambda)}\right)}
{\dfrac{\sin(\pi \beta_{a})}{\sinh(\pi  \lambda)}}\to 1, \qquad 
\frac{\sin(\pi \beta_{a)})}{\beta_{a}}\to \pi. 
\]
This gives the following 
\begin{corollary}
\begin{equation}\label{DavisTheorem}
|\{x\in \R: |H1_A(x)|> \lambda\}|=
\lim_{a\to\infty}
\pi a\, \frac{2}{\pi}
\arctan\!\left(\frac{\sin(\pi \beta_{a})}{\sinh(\pi  \lambda)}\right)
=
\frac{2|A|}{\sinh(\pi  \lambda)}.
\end{equation}
\end{corollary} 
From the density formulas restricted to the two boundary components of the strip, 
for every $ \lambda>0$ we have
\[
\mathbb P_{(0,a)}\bigl(B_\tau\in A,\ |H1_A(B_\tau)-\beta_a|> \lambda\bigr)
=
\int_{|u|> \lambda}
\frac{\sin(\pi\beta_a)}
{2(\cosh(\pi u)+\cos(\pi\beta_a))}\,du,
\]
and
\[
\mathbb P_{(0,a)}\bigl(B_\tau\in A^c,\ |H1_A(B_\tau)-\beta_a|> \lambda\bigr)
=
\int_{|u|> \lambda}
\frac{\sin(\pi\beta_a)}
{2(\cosh(\pi u)-\cos(\pi\beta_a))}\,du.
\]

Multiplying by $\pi a$ and using
$\sin(\pi\beta_a)\sim \pi\beta_a$ and $\cos(\pi\beta_a)\to 1$, we obtain
\[
\pi a\,
\mathbb P_{(0,a)}\bigl(B_\tau\in A,\ |H1_A(B_\tau)-\beta_a|> \lambda\bigr)
\to
|A|\int_{|u|> \lambda}\frac{\pi}{2(\cosh(\pi u)+1)}\,du.
\]
By symmetry and the identity $\cosh t+1=2\cosh^2(t/2)$,
\begin{align*}
\int_{|u|> \lambda}\frac{\pi}{2(\cosh(\pi u)+1)}\,du
&=
\pi\int_y^\infty \frac{du}{\cosh(\pi u)+1}\\
&=
\frac{\pi}{2}\int_ \lambda^\infty \operatorname{sech}^2\!\Bigl(\frac{\pi u}{2}\Bigr)\,du\\
&=
\int_{\pi y/2}^\infty \operatorname{sech}^2 v\,dv\\
&=
1-\tanh\Bigl(\frac{\pi  \lambda}{2}\Bigr)
=
\frac{2}{e^{\pi y}+1}.
\end{align*}
Hence
\[
\lim_{a\to\infty}
\pi a\,
\mathbb P_{(0,a)}\bigl(B_\tau\in A,\ |H1_A(B_\tau)-\beta_a|>y\bigr)
=
\frac{2|A|}{e^{\pi y}+1}.
\]
On the other hand,
\[
\mathbb P_{(0,a)}\bigl(B_\tau\in A,\ |H1_A(B_\tau)-\beta_a|>y\bigr)
=
\frac1\pi\int_{A\cap\{|H1_A(s)-\beta_a|>y\}}
\frac{a}{s^2+a^2}\,ds,
\]
so
\[
\pi a\,
\mathbb P_{(0,a)}\bigl(B_\tau\in A,\ |H1_A(B_\tau)-\beta_a|>y\bigr)
=
\int_{A\cap\{|H1_A(s)-\beta_a|>y\}}
\frac{a^2}{s^2+a^2}\,ds.
\]
Letting $a\to\infty$ yields 
\begin{corollary}
\begin{align}\label{Laeng1}
\bigl|\{x\in A:|H1_A(x)|>y\}\bigr|
=
\frac{2|A|}{e^{\pi y}+1}.
\end{align}
Similarly,
\begin{align}\label{Laeng2}
\bigl|\{x\in A^c:|H1_A(x)|>y\}\bigr|
=
\frac{2|A|}{e^{\pi y}-1}.
\end{align}
\end{corollary}
Both equalities \eqref{Laeng1}, \eqref{Laeng2} were first proved in \cite{Lae2000} by different analysis methods. 

Versions of the above inequalities hold for the martingale transforms associated with the matrices $\BH_j$ in all odd dimensions constructed by the harmonic extension of indicator of Lebesgue sets in $\R^d$. These results do not apply, at least not directly, to the operators $R_j=T_{\BH_j}$ since their distributions need not be controlled by those of the corresponding martingale transforms, as is the case for $L^p$ norms, $1\leq p<\infty$.  For closely related $L^p$, $LLogL$ and versions of weak-type  inequalities that hold in all dimensions and apply to  Riesz transforms in other geometric settings,  see \cite{Ose2012, OseRiesz, BanOsc15}.

These types of arguments have been used to identify the best constants in the Kolmogorov weak-type inequality for the Hilbert transform, as originally done in \cite{Davis1973} for $p=1$, and subsequently for the general case of orthogonal martingales in \cite{Jan2004} for $1\leq p\leq 2$. While the weak-type  inequalities hold for the Hilbert transform as well as for martingale transforms associated with the matrices $\BH_j$, the problem of proving similar estimates for $R_j=T_{\BH_j}$ remains open. 

In particular, it is not known whether these operators satisfy a weak-type $(1,1)$ inequality with constant independent of the dimension, a problem raised by E.~M.~Stein \cite[Problem (b), p.~203]{Stein1987ICM} in his Berkeley 1986 ICM address. It was shown in \cite{Janakiraman2004} that the weak-type constant grows at most logarithmically in $d$, which, to the best of our knowledge, is the best result currently available in the literature. The proof uses a modification of the Calder\'on-Zygmund decomposition which in its classical form leads to exponential growth with in the dimension. 

In \cite{SpectorStockdale2021} a different proof of Janakiraman's result is presented which shows that the problem can be reduced to studying Riesz transforms applied to finite sums of Dirac masses. Even such logarithmic growth does not appear attainable using the current martingale methods.

\section{Problem 1: E.M.~Stein's inequality for vector of Riesz transforms} 

\subsection{Brief history}   
Set 
\[Rf=(R_1f, R_2f, \dots, R_df)\] 
and 
\begin{align}
\|R\|_{L^p(\R^d)}=\sup_{\|f\|_{L^p(\R^d)}\leq 1}\Big\|\left(\sum_{k=1}^d |R_kf|^2\right)^{1/2}\Big\|_{L^p(\R^d)}. 
\end{align}
It follows from Calder\'on-Zygmund  theory \cite{Stein70} that 
\begin{align}\label{Stein1}
\|R\|_{L^p(\R^d)}\leq 
\begin{cases} \frac{C_d}{p-1}, &  \quad 1<p\leq 2,\\
C_d(p-1), & \quad 2\leq p<\infty, 
\end{cases} 
\end{align}
where the constant $C_d$ grows exponentially with $d$.   The behavior in $p$, as $\to 1$ and $\p\to\infty$,  is best possible. In his landmark paper \cite{SteSome} E.~M.~Stein proved  \eqref{Stein1} 
with a constant $A_p$  independent of $d$.
His proof  uses $L^p$-inequalities for Littlewood-Paley square functions  \cite{SteLP}.  His constant, although independent of dimension, did not give the correct behavior in $p$. It uses the fact that $\|g(f)\|_{L^p(\R^d)}$  is comparable to  $\|f\|_{L^p(\R^d)}$, $1<p<\infty$, with constants depending only on $p$. Here  $g$ denotes the (horizontal, vertical or full) Littlewood-Paley function as in  \cite[Chapter IV]{Stein70}.  While it is well-known that $\|g(f)\|_{L^p(\R^d))}\leq b_p\|f\|_{L^p(\R^d)}$ with $b_p\sim O(\sqrt{p})$, as ${p}\to\infty$ and that this growth is best possible, the best known behavior for the constant $a_p$ in the inequality $\|f\|_{L^p(\R^d)}\leq a_p\|g(f)\|_{L^p(\R^d)}$ is $O(p)$, as ${p}\to\infty$.  Hence  the Littlewood-Paley function argument will give, at best, $A_p\sim p^{3/2}$, as $p\to\infty$.  

It is worth noting here that proving that the behavior of $a_p$ is $O(\sqrt{p})$ for the Littlewood-Paley function $g$ is a well-known open (and perhaps forgotten by now) problem  which is related  to the sub-gaussian behavior  of functions with bounded Luzin area integral studied by Chang, Wilson and Wolff \cite{ChaWilWol}. 
For more information on this type of behavior of constants for  square functions  on $\R^d$,  see \cite{BanMoore}. 

A different proof of \eqref{Stein1} (and a version for Wiener space)  was given by P.~A.~Meyer in \cite{Mey1} using the Burkholder-Gundy martingale square function inequalities to prove the needed Littlewood-Paley inequalities.  A similar argument was used in A.~Bennett's  1984 thesis \cite{BenA} written under Stein. These proofs did not yield the correct behavior in $p$ either.

Stein's papers on dimension-free bounds for Riesz transforms, maximal functions, and other classical operators in analysis generated great interest in the harmonic analysis community; interest that continues to this day. In particular, soon after \cite{SteSome} appeared, several alternative proofs of \eqref{Stein1} were obtained, including probabilistic proofs based on the stochastic integral representation of Gundy--Varopoulos \cite{GV79}, as  in \cite{Ban86}, as well as the approach in \cite{MR780616} using techniques from the method of rotations.  None of these proofs yields the correct behavior in $p$ both as $p\to 1$ and $p\to\infty$.

In \cite{Pis}, G.~Pisier gave an elegant proof based on the {\it transference method} of Coifman and Weiss which extends ideas from the method of rotations to a much broader setting (see, for example, ``Some classical examples of transference,'' \cite[pp.~5--9]{CoiWei1976}) together with a projection operator onto the first Wiener chaos. Pisier's argument gives the correct order as $p\to\infty$.  However, for $1<p\le2$ his  resulting estimate is of order $(p-1)^{-3/2}$, which is not optimal. This is also the same behavior giving in \cite{MR780616}.  

The correct behavior for the full range of $p$ first appeared in \cite{BanMich}, where it is derived from an  exponential vector-valued martingale  together with a sharp vector-valued good-$\lambda$ inequality that follows from it. Nevertheless, Pisier's proof remains remarkable: it not only captures  the correct growth as $p\to\infty$, but also provides the best known uniform upper bound currently available in the literature for the range $2\leq p<\infty$.  

\subsection{Known upper bounds}

Although there is a vast literature now on sharp inequalities for Riesz transforms and their variants in different geometric settings, the problem of identifying the sharp constant $A_p$ (not just its correct dependance on $p$) in Stein's \eqref{Stein1} inequality for any $d>1$ remains open, except for the trivial case of $p=2$.  Below we present a short survey of some known estimates with the caveat that the list is not exhaustive and that improved or more refined estimates that we are not aware of may be available in the literature. We restrict our discussion to the classical case of Riesz transforms on $\R^d$, which we regard as the fundamental setting and the natural starting point in the search for optimal bound. Thus there are many missing references in our discussion below to similar bounds for Reisz transforms in many different settings with a variety of different techniques such as  those in \cite{Wrobel2018, CarbonaroDragicevic2013, Arcozzi1998, AuscherCoulhonDuong2005, Bak1, Li2008, Bakry1987},  and references therein. 

Before we discuss Pisier's theorem, we present other developments related to Stein's inequality on $\R^d$.  

\begin{knowntheorem}[T~Iwaniec and G.~Martin \cite{IwaMar}]  Set $\BH_{\bC}=R_1+iR_2$ where  $R_1$ and $R_2$  are the Riesz transforms on $\R^2$. Then, 
\begin{align}\label{IwaMar1}
\|R\|_{L^p(\R^d)}\leq \sqrt{2}\|\BH_{\bC}\|_{L^p(\R^d)}, \quad 2\leq p<\infty. 
\end{align}
\end{knowntheorem}
Iwaniec and Martin  proved this and other related results by extending  the classical method of rotations on $\R^d$ to the complex setting on $\bC^{2d}$.  They called the operator $\BH_{\bC}$  {\it "the complex Hilbert transform."}  It is also called  the {\it ``complex Riesz transform"} in other places in the literature.  The  value of the norm of its powers, $\|(\BH_{\bC})^{k}\|_{L^p(\R^2)}$, $k\in \bN$,  and is applications,  has been investigated by several authors. For some of this literature we refer to  \cite{AstIwaMar}, \cite{CarDraKov2023}, \cite{Dra}, \cite{DraPetVol1}, \cite{IwaSbo} and the many references contained therein. 
 
\begin{remark}\label{Warning1}
For clarity of comparison with results in the literature (such as \cite{CarDraKov2023} and \cite{AstIwaMar}), we remind the reader that we use $\|T\|_{L^p(\R^d)}$ and $\|T\|_{L^p(\R^d; \, \bC)}$  to denote the norm of the operator $T$ acting on real-valued and complex-valued functions, respectively.
This is an important distinction when dealing with norms of operators that map real-valued functions to complex-valued functions such as $\BH_{\bC}$ and the 
Beurling-Ahlfors operators defined in \eqref{BA:Def} below.  
\end{remark}
 
 Since the operator $R_1+iR_2$ on $\R^d$ is a Fourier multiplier extension of $R_1+iR_2$ on $\R^2$,  it follows from de Leeuw, \cite{DeL}, \cite[Theorem 2.5.15]{Gra}, that 
 \begin{align}\label{IwanMart6}
 \|\BH_{\bC}\|_{L^p(\R^2)}\leq \|R_1+iR_2\|_{L^p(\R^d)}\leq \|R\|_{L^p(\R^d)},  \quad 1<p<\infty. 
 \end{align}
 Thus, 
 \begin{align} 
\|\BH_{\bC}\|_{L^p(\R^2)}\leq \|R\|_{L^p(\R^d)}\leq \sqrt{2}\|\BH_{\bC}\|_{L^p(\R^2)}, \quad 2\leq p<\infty.  
\end{align}

 Clearly a lower bound is given by the norm of either $R_1$ or $R_2$. Applying Minkowski
and using \eqref{IwMa2} gives  
\begin{align}\label{Complex:H1}
\cot\Big(\frac{\pi}{2p^*}\Big)\leq \|\BH_{\bC}\|_{L^p(\R^2)}\leq 2 \cot\Big(\frac{\pi}{2p^*}\Big), \quad 1<p<\infty
\end{align}
and hence,
\begin{align}\label{Upper:Lower}
\cot\Big(\frac{\pi}{2p}\Big) \leq \|R\|_{L^p(\R^d)}
\leq 2\sqrt{2} \cot\Big(\frac{\pi}{2p}\Big),\quad 2\leq p<\infty.
\end{align} 

In \cite[Corollary 4.5]{BW95}, the sharp  martingale inequalities of Burkholder \cite{Burk84} are used to prove that 
\begin{align}\label{BanWan1}
\|R\|_{L^p(\R^d)}\leq 2(p^*-1)= \begin{cases}
    \frac{2}{(p-1)}, &1<p\leq 2,\\
    2(p-1), & 2\leq p<\infty.
    \end{cases}
\end{align} 
This inequality holds for complex-valued functions. 

How does this compare with the bound in \eqref{Upper:Lower} for $p>2$? Setting $f(p)=(p-1)-\sqrt{2}\cot({\pi}/{2 p})$, it is not difficult to see that there exist a $p_0$ ($\approx 9$) such that $f(p_0)=0$ and $f(p)<0$ for $2<p<p_0$ and $f(p)>0$ for $p>p_0$. Thus there is some small improvement  for $2<p<p_0$. It is clear however that trivially estimating the norm of $\BH_{\bC}$ from above using 
Minkowski's inequality is not efficient. 

In what follows we aim to provide some better bounds for $\|\BH_{\bC}\|_{L^p(\R^2)}$ and $\|R\|_{L^p(\R^d)}$ for $p\geq 2$.  
We begin with the classical method or rotations. This method has been used in various palaces in the literature to estimate norm of $(\BH_{\bC})^k$ when $k$ is odd; see for example \cite{Dra,  CarDraKov2023, IwaSbo}. From  \eqref{RieszForierM}, $\BH_{\bC}$ has Fourier multiplier 
\[
\,i\,\frac{\xi_1+i\xi_2}{|\xi|}
\]
and therefore,
 $(H_{\bC})^k$ has multiplier
\[
(i)^k e^{ik\theta},
\qquad
\xi=(|\xi|\cos\theta,|\xi|\sin\theta).
\]
Let 
\[
H_\phi f(x) := \text{p.v.} \frac{1}{\pi} \int_{\mathbb{R}} \frac{f(x - t e_\phi)}{t} \, dt.
\]
be the Hilbert transform \(H_\phi\) in direction
\(e_\phi=(\cos\phi,\sin\phi)\). Its multiplier is
$
i \operatorname{sgn}(\cos(\theta-\phi)).
$
Using the Fourier expansion
\[
\operatorname{sgn}(\cos t)
=
\frac{4}{\pi}\sum_{m=0}^\infty \frac{(-1)^m}{2m+1}\cos((2m+1)t),
\]
we see that only odd angular frequencies occur. Thus, for each odd
integer \(k\ge1\),
\begin{equation}\label{k-rotation}
(H_{\bC})^k
=
\omega_k\,\frac{\pi}{2}\,k
\int_0^{2\pi} e^{ik\phi} H_\phi\,\frac{d\phi}{2\pi},
\qquad |\omega_k|=1.
\end{equation}
In particular, it follows immediately 
 \begin{equation}\label{Complex:H2}
\|\left(\BH_{\bC}\right)^k\|_{L^p(\R^2)}\leq k\frac{\pi}{2} \|H\|_p,\,\,\, 1<p<\infty, 
 \end{equation} 
 which for $k=1$ improves the bound in \eqref{Upper:Lower}. 
 
 For   $d=2$,  we obtain from \eqref{Rieszd=2} and Lemma \ref{Lemma:Cotlar} that with 
\begin{align}\label{RieszMart:Matrice} 
 \BR_1=\bmatrix 0 & 1 & 0 \\ -1 & 0 & 0\\ 0 & 0 & 0 \endbmatrix,\,\,  \BR_2= \bmatrix 0 & 0 & 1 \\ 0 & 0 & 0 \\ -1 & 0 & 0 \endbmatrix,
\end{align} 
we can write 
\[\BH_{\bC}=T_ {\BR_1}+i T_ {\BR_2}, 
\]
where $T_ {\BR_1}$ and $T_ {\BR_2}$ are the projection of the martingale transforms.  
Applying the vector valued inequality
\cite[Theorem A]{BanuelosWang1996}
gives a further improvement   over \eqref{Complex:H2}
\begin{equation}\label{Complex:H3}
\|\BR_1+i\BR_2\|_{L^p(\R^2)} \leq \sqrt{2}\cot\Big(\frac{\pi}{2p}\Big), \,\,\, 2\leq p<\infty.   
 \end{equation}
 
 Both, the method of rotations and the martingale method, are not any better than what follows from the elementary inequality  $(a^2+b^2)^{p/2}\leq 2^{p/2 -1}(a^2+b^2)$ which already improves  the Minkowski estimate \eqref{Upper:Lower} to $\sqrt{2}$.  On the other hand, both \eqref{Complex:H2} and \eqref{Complex:H3} hold for complex-valued functions in the range of $2\leq p<\infty$. Regardless, we have the following bound 

\begin{align}\label{Stein2} \cot\Big(\frac{\pi}{2p^*}\Big)\leq \|\BH_{\bC}\|_{L^p(\R^2)}\leq \|R\|_{L^p(\R^d)}\leq 
\begin{cases}
2(p^*-1) & 1<p\leq 2,\\
2\cot\Big(\frac{\pi}{2p}\Big), & 2\leq p<\infty,   
\end{cases} 
\end{align}

The problem of determining the exact  $L^p$-norm $\BH_{\bC}$ (which is the same as the best constant in Stein's inequality  for $d=2$)
remains open  both in the complex and real case, as far as we know. Thus,  further improvements on the upper bound for $\|R\|_{L^p(\R^d)}$
in terms of the norm of $\BH_{\bC}$ appear unlikely at present.  Moreover, the constant $\sqrt{2}$ in the Iwaniec-Martin inequality \eqref{IwaMar1}, which is obtained in two steps:  (1) from the Fourier extension to $\bC^n$  combined with de Leeuw’s extension theorem and (2)  from the complex method of rotations, also seems quite challenging.

Similarly, applying the vector value martingale inequality  from \cite{BW95} on $\R^d$, which gives the constant $\sqrt{2}$ in \eqref{Complex:H3} on $\R^2$,  would give the  dimension depending bound 
$\|R\|_{L^p(\R^d)}\leq \sqrt{d}\cot({\pi}/{2 p})$,  for any $2\leq p<\infty$.  In addition, the inequality $\|R\|_{L^p(\R^d)}\leq 2(p^*-1)$ proved  in \cite{BW95} uses the matrix representation for $\frac{1}{2}R_j$ as in Lemma \ref{Lemma:Cotlar}.  With those  matrices, which are no longer orthogonal but have the differential subordination property independent  of $d$, Burkholder's vector-valued inequality gives the bound $(p^*-1)$ and multiplying by 2 gives the bound in \eqref{BanWan1} valid for all $1<p<\infty$. This it is  difficult to see improvements with the techniques discussed above.

\subsection{G.~Pisier's bound} We now state Pisier’s theorem  inserting the explicit value of $C_p$ given on page 466.  For completeness, and because the technique will be discussed further below, we also include its proof. Our presentation follows Pisier’s original argument, supplemented with additional clarifications where appropriate in order to make the exposition as accessible as possible and self-contained.  The aim of \cite{Pis}, as stated in the introduction, was to give a proof of the inequality for the Ornstein-Uhlenbeck Riesz transforms; Meyer's  \cite{Mey1} theorem.  Here we only discuss  the classical case,  \cite[Theorem 1.1]{Pis}.  

The proof for the Ornstein–Uhlenbeck case \cite[Theorem 2.1]{Pis}, although more technical, follows a similar pattern. While the resulting universal bound is not as sharp as in the classical case, it exhibits the same asymptotic behavior as $p\to 1$ and $p\to\infty$. 
 See also \cite[Theorem 7.5]{Hu2017}, where Pisier’s proof of Meyer’s inequality is adapted to the setting of abstract Wiener space.

\begin{knowntheorem}[G. Pisier 1986 \cite{Pis}]\label{Pisier:Main} Let  $1<p<\infty$ and $q$ be its conjugate exponent. For a standard normal random variable $X$, i.e., $X\sim N(0, 1)$,  we set
 $\gamma(p)=\|X\|_p$. Then, 
\begin{equation}\label{eq:Cp-explicit}
\|R\|_{L^p(\R^d)}
\le
\begin{cases}
\frac{\gamma(q)}{\gamma(p)}\sqrt{\frac{\pi}{2}}\,\,\|H\|_p, &1<p\leq 2,\\
\sqrt{\frac{\pi}{2}} \,\,\|H\|_p, & 2\leq p<\infty. 
\end{cases}
\end{equation}
\end{knowntheorem}

Before presenting the proof, we give some upper bounds on $\frac{\gamma(q)}{\gamma(p)}$ for $1<p\leq 2$. By trivial computation using the normal distribution, 
\[
\gamma(q)=\left(
\frac{2^{q/2}\,\Gamma\!\left(\frac{q+1}{2}\right)}
{\sqrt{\pi}}
\right)^{\!1/q}.\]
By Stirling’s approximation of the gamma function, $\gamma(q)\sim 
 \sqrt{\frac{q}{e}},\quad q\to\infty.$
On the other hand $\gamma(p)\to \gamma(1)=\sqrt{2/\pi}$ as $p\to 1^+$. Thus
\[
\frac{\gamma(q)}{\gamma(p)}
\sim
\sqrt{\frac{\pi}{2e}}\,
\frac{1}{\sqrt{p-1}}, \quad  p\to 1^+.
\]
Form this and the fact that   $\|H\|_p=\cot\Big(\frac{\pi}{2p^*}\Big)$, the Big-O behavior,  $(1-p)^{-3/2}$, as $p\to 1$, and $p$ as $p\to\infty$ follow as stated in \cite[p.~ 488]{Pis}.   
One can go a step further and give the sharp universal bound on $\frac{\gamma(q)}{\gamma(p)}$ for  the case $1<p<2$. Here we discuss three estimates and give the bound that follows from the theorem in the Corollary \ref{Pisier:corollary} below. 
\bigskip

\noindent{\bf (1) Nelson--Gross Hypercontractivity \cite{Gross1975, Nelson1973}:}   Let $(P_t)_{t\ge 0}$ denote the Ornstein--Uhlenbeck (OU) semigroup on  $\R$ with the Gaussian measure $\gamma$.  Then for all  $1<r< s<\infty$, $\|P_t f\|_{L^s(\gamma)} \le \|f\|_{L^r(\gamma)}$ whenever $s-1=(r-1)e^{2t}.$ Applying this with $r=2$, $s-1=e^{-2t}$,  and  $f(x)=x$, which is an eigenfunction of $P_t$ with eigenvalue $e^{-t}$,  gives 
$\gamma(s)= e^{t}\,\|P_t f\|_{L^s} \le e^{t} = \sqrt{s-1},$ for  $s>2$.  Applying this with $s=q$ 
and using the fact that $\gamma(1)=\sqrt{2/\pi}\leq \gamma(p)$, since  $\gamma(r)$ is increasing  by Jensen's inequality,   it follows that 
\begin{equation}\label{Pisier:lower1}
\frac{\gamma(q)}{\gamma(p)}
\leq \sqrt{\frac{\pi}{2}}\sqrt{q-1}=\sqrt{\frac{\pi}{2}}\frac{1}{\sqrt{p-1}}, \quad 1<p<2.
\end{equation}
\bigskip

\noindent{\bf (2) Gautschi/Kershaw and Stirling\label{eq:kershaw}} A different way to get the upper bound above is as follows; set $q=2r$ so that  
\[
 \bE|X|^{2r}
=
\frac{2^{r}\Gamma\!\left(r+\tfrac12\right)}{\sqrt{\pi}}.\]
From  the  Gautschi/Kershaw and Stirling inequalities  it follows that 
\begin{equation*}\label{eq:kershaw}
\Gamma\!\left(r+\tfrac12\right) \le \sqrt{r}\Gamma(r)\le \sqrt{r} \sqrt{2\pi}\,r^{r-\frac12}e^{-r},  \,\,\, r\geq 1, 
\end{equation*}
This gives,
\begin{align*}
\bE|X|^{2r}
 \le 
\frac{2^{r}\sqrt{r}}{\sqrt{\pi}}\,
\sqrt{2\pi}\,r^{r-\frac12}e^{-r}
=2^{r}\,r^{r}\,\sqrt{2}\,e^{-r}
\leq 2^{r}\,r^{r}, \quad r\geq 1.  
\end{align*} 
Taking roots we have $\gamma(q)\leq \sqrt{q}$, for $2<q<\infty$,  equivalent for $1<p<2$. This gives the same upper bound as \eqref{Pisier:lower1}. 
\bigskip

\noindent{\bf (3) Hypercontractivity and $\log$-convexity of absolute moments \cite{DLMF-Gamma}:}  The Hypercontractivity inequality shows that the function $\varphi(r)=\frac{\gamma(r)}{\sqrt{r-1}}$ is monotone decreasing for $r>1$ which implies that $\varphi(r)=\frac{\gamma(r)}{\sqrt{r}}$ is decreasing for $r>2$. From the fact that the absolute moments $M(r)=\bE|X|^r$  is  log-convex  it can be shown that $\varphi(r)=\frac{\gamma(r)}{\sqrt{r}}$ is decreasing for $r>1$.  Thus for any $1<p<q$, 
$\frac{\gamma(q)}{\sqrt q}\le \frac{\gamma(p)}{\sqrt p}$. 
If $q=p/(p-1)$, then $\sqrt{q/p}=(p-1)^{-1/2}$ and it follows that  
\[
\frac{\gamma(q)}{\gamma(p)}
=\frac{\gamma(q)/\sqrt q}{\gamma(p)/\sqrt p}\,\sqrt{\frac{q}{p}}
\le \sqrt{\frac{q}{p}}=\frac{1}{\sqrt{p-1}}
\]
\begin{remark}
Although the above observations  (all known) may be interesting in the search for sharp constants, {\bf the real problem} is to remove the factor $1/\sqrt{p-1}$ even at the expense of another constant  that does  not match the bound for the case $2\leq p <\infty$. 
\end{remark}

We summarize the explicit bound in \eqref{eq:Cp-explicit} as a corollary. 
\begin{corollary}\label{Pisier:corollary}
\begin{equation}\label{eq:Cp-moreexplicit}
\|R\|_{L^p(\R^d)}
\le
\begin{cases}
\sqrt{\frac{\pi}{2}}\frac{1}{\sqrt{p-1}}\,\,\|H\|_p, &1<p\leq 2,\\
\sqrt{\frac{\pi}{2}} \,\,\|H\|_p, & 2\leq p<\infty 
\end{cases}
\end{equation}
\end{corollary}

 Combining the case $p\geq 2$ with the estaminet in \eqref{BanWan1}, we have
 \begin{corollary}\label{SteBestReal}
 \begin{equation}\label{SteBestReal:R}
 \|R\|_{L^p(\R^d)}
\le
\begin{cases}
\frac{2}{p-1}, &1<p\leq 2,\\
\sqrt{\frac{\pi}{2}} 
\,\,\|H\|_p, & 2\leq p<\infty. 
\end{cases}
\end{equation}
\end{corollary}

\begin{remark} Notice that for $2\leq p<\infty$ the constant $\sqrt{\frac{\pi}{2}}\approx   1.25331413$  is better than all others discussed  above. Of course, it still does not attain the sharp  constant $1$  at $p=2$.  It is natural to ask whether the same bound  $\sqrt{\pi/2}\,\|H\|_p$ might hold 
for the entire range $1<p<\infty$ in Corollary \ref{Pisier:corollary}.  That is, without the additional factor 
$(p-1)^{-1/2}$ appearing for $1<p<2$. This value is an integral part of the proof and removing it by some simple modification of the argument does not seem possible. Hence a different approach would be required. 
\end{remark}

\begin{proof}[Proof of Theorem \ref{Pisier:Main}] Consider the product space $(\mathbb R^n\times\mathbb R^n,\,dx\otimes d\gamma_d)$, where 
\[
d\gamma_d=(2\pi)^{-d/2}e^{-|y|^2/2}\,dy.
\]
For $f\in L^p(\mathbb R^d)$, let  the function $ f(x+t y)$ be define on $\mathbb R^d\times\R^d\times \R$. By the translation invariance of the Lebesgue measure,
\begin{equation}\label{eq:isometry}
\int_{\mathbb R^n}\int_{\mathbb R^d} |f(x+t y)|^p\,dx\,d\gamma_n
=
\int_{\mathbb R^d}|f(x)|^p\,dx,
\end{equation}
Thus, as Pisier states it,  the function $(x,y)\to f(x+ty)$  is an {\it isometric embedding}
$L^p(dx)\hookrightarrow L^p(dx\otimes d\gamma_d)$,  for each fixed $t$.

Define an operator $\mathcal H$ acting on functions of $(x,y)$ by
\begin{equation*}\label{eq:calH-def}
(\mathcal H f)(x,y)
:=
\frac1\pi\,\p.v.\!\int_{\mathbb R}\frac{f(x+t y)}{t}\,dt. 
\end{equation*}
Then by \eqref{eq:isometry}, 
\begin{equation}\label{eq:calH-Lp}
\|\mathcal H f\|_{L^p(dx\otimes d\gamma_d)}
\le \|H\|_p\,\|f\|_{L^p(dx)}.
\end{equation}
This is the transference argument as in \cite{CoiWei1976}.   

Pisier now uses the orthogonal projection $Q$ from $L^2(\gamma_d)$ onto the first the first Wiener chaos.  For completeness, we elaborate further.  Let $H_n$ denote the (probabilist's) Hermite polynomials in one variable,
\[
H_n(y)=(-1)^n e^{y^2/2}\frac{d^n}{dy^n}e^{-y^2/2}.
\]
For a multi-index $\alpha=(\alpha_1,\dots,\alpha_d)\in\mathbb{N}^d$,
define the multivariate Hermite polynomial
\[
H_\alpha(y)=\prod_{j=1}^d H_{\alpha_j}(y_j),
\qquad
|\alpha|=\alpha_1+\cdots+\alpha_d.
\]
The family $\{H_\alpha : \alpha\in\mathbb{N}^d\}$ forms an orthogonal
basis for $L^2(\gamma_d)$.
The $m$ Wiener chaos is defined as
$
\mathcal H_m
=\overline{\operatorname{span}}
\{H_\alpha : |\alpha|=m\}
$ and the following is the Wiener chaos decomposition on $\R^d$ (see, \cite[Theorem 5.1 \& Example 5.12]{Hu2017})/. That is,  
\begin{align}\label{chaos-decomp}
L^{2}(\gamma_{d})
=
\bigoplus_{m=0}^{\infty}
\operatorname{span}\{H_{\alpha}(x):|\alpha|=m\}.
\end{align}
 In particular, the \emph{first Wiener chaos} is
$
\mathcal H_1
=\operatorname{span}\{H_\alpha : |\alpha|=1\}.
$
Since $|\alpha|=1$ implies $\alpha=e_j$ for some $j=1,\dots,d$,
and since
$H_{e_j}(y)=H_1(y_j)=y_j$ we have 
$
\mathcal H_1
=\operatorname{span}\{y_1,\dots,y_d\}. 
$
Thus the first Wiener chaos consists of all linear functions $y\to a\cdot y$, $a\in R$. 
For $h\in L^2(\gamma_d)$ define the orthogonal projection onto $\mathcal H_1$ by 
\begin{align}\label{ProJe:Q}
(Qh)(y)=\sum_{j=1}^d \ip{h,  y_j}\,y_j=a\cdot y, \quad a\in \R^d, 
\end{align}
where 
\[a_j=\ip{h, y_j}=\int_{\R^d}h(y)\,y_j\,d\gamma_d, \quad j=1, \dots, d.\]
 
Since the coordinate functions $y_j$ are i.i.d standard normal random variables, 
 $Qh$ is a normal random variable $X\sim N(0, \sigma^2)$ with $\sigma^2=\|Qh\|_{L^2(\gamma_d)}^2$.
Thus for any $2\leq p<\infty,$ the boundedness  of $Q$ on $L^2(\gamma_d)$ and Jensen's inequality give that for $2\leq p<\infty$, 
\begin{align}\label{Q:estiamte} 
\|Qh\|_{L^p(\gamma_d)}&=\gamma(p)\,\|Qh\|_{L^2(\gamma_d)}\nonumber\\
&\le \gamma(p)\,\|h\|_{L^2(\gamma_d)} \\ 
&\le \gamma(p)\,\|h\|_{L^p(\gamma_d)}.\nonumber
\end{align}
By duality, for $1<2\leq p$, this gives 
\begin{equation}\label{eq:Q-Lp-ple2}
\|Qh\|_{L^p(\gamma_d)}
\le \gamma(q)\,\|h\|_{L^p(\gamma_d)},
\end{equation}
where $q$ is the conjugate exponent  of $p$. 

For $(x, y)\in \R^d\times \R^d$, denote by  $Q_y$ the operator $Q$ acting on  the $y$ variable. 
Pisier's crucial identity, his equation (1.5), is:
\begin{equation}\label{eq:pisier-1.5}
Q_y(\mathcal H f)(x,y))
=
\sqrt{\frac{2}{\pi}}\,\sum_{j=1}^d y_j\, (R_j f)(x).
\end{equation}

Let us assume this for the moment.  Taking the $L^p(dx\otimes d\gamma_d)$ norms of both sides, applying \eqref{eq:calH-Lp}, \eqref{Q:estiamte} and  \eqref{eq:Q-Lp-ple2} we have 
\begin{align}\label{Upper:allp}
\sqrt{\frac{2}{\pi}}\Big\|\sum_{j=1}^d y_j R_j f(x)\Big\|_{L^p(dx\otimes d\gamma_d)}
 \leq \begin{cases} \gamma(q)\,\, \|H\|_p\,\|f\|_{L^p(\R^d)},& \,\,\, 1\leq p\leq 2, \\\
 \gamma(p)\,\, \|H\|_p\,\|f\|_{L^p(\R^d)}, &\,\,\, 2\leq p<\infty.
 \end{cases} 
\end{align} 

Exactly as above, for each $x$, $\sum_{j=1}^d y_k R_j f(x)$ is he sum of i.i.d. standard  normals and hence  a mean 0 normal with variance $\sum_{j=1}^d |R_j f(x)|^2$. This gives
\begin{align}\label{GaussianAve:1}
\Big\|\sum_{j=1}^d y_k R_j f(x)\Big\|_{L^p(\gamma_n)}
=
\gamma(p)\Big(\sum_{j=1}^n |R_j f(x)|^2\Big)^{1/2}, \, \,\, 1<p<\infty
\end{align}
This identity is what we call Gaussian averaging and to which we will return below when we explore asymptotic behavior.  
Taking the $L^p(dx)$ norm of both sides \eqref{GaussianAve:1} gives 
\begin{equation}\label{eq:Cp-moreexplicit}
\|R\|_{L^p(\R^d)}
\le
\begin{cases}
\frac{\gamma(q)}{\gamma(p)}\sqrt{\frac{\pi}{2}}\,\,\|H\|_p, &1<p<2,\\
\sqrt{\frac{\pi}{2}} \,\,\|H\|_p, & 2\leq p<\infty, 
\end{cases}
\end{equation}
which is the statement of the Theorem. 

It remains to verify \eqref{eq:pisier-1.5}. As Pisier says, this ``is easy to check using the Fourier transform in the $x$ variable."  Fix $y\in\R^d$. Applying  \eqref{Fouier:MulH} the Fourier transform of the left-hands side of \eqref{eq:pisier-1.5} in the $x$ variable gives 
\begin{equation}
\widehat{Q_y(Hf(\cdot, y))}(\xi)=Q_y\big[i\sign(\xi\cdot y)\big]\,\widehat f(\xi)=i\, Q_y\big[\sign(\xi\cdot y)\big]\,\widehat f(\xi).
\end{equation}
By definition of the first Wiener chaos, 
\[Q_y\big[\sgn(\xi\cdot y)\big]=a\cdot y, \quad a\in \R.
\]
Let $\rho$ be orthogonal matrix $\rho\in O(d)$ (the orthogonal group in $\R^d$) which fixes $\xi$. Since $\sign(\xi\cdot \rho y)=\sign(\xi\cdot y)$
and $\gamma_d$ is rotationally invariant, we have 
\[
a\cdot (\rho y)
=
Q_y\big(\sgn(\xi\cdot \rho y)\big)
=
Q_y\big(\sgn(\xi\cdot y)\big)
=
a\cdot y.
\]
This shows that $a$ is invariant under all orthogonal transformations fixing
$\xi$ and hence it must be a multiple of $\xi$. This gives $a=\alpha(\xi)\,\xi$ and it follows that 
\begin{align}
Q_y\big(\sgn(\xi\cdot y)\big)=\alpha(\xi) \xi\cdot y, 
\end{align} 

To determine $\alpha(\xi)$, set $X=\xi\cdot y\sim N(0, |\xi|^2)$.  By orthogonality, 
\begin{align*}
&\bE\Big[\Big(\sign(\xi\cdot y)-Q_y(\sign(\xi\cdot y))\Big)\xi\cdot y\Big]\\
&=\bE\Big[\Big(\sign(\xi\cdot y)-\alpha(\xi)\,\xi\cdot y \Big)\xi\cdot y\Big]=0
\end{align*}
Equivalently,
\[\bE[\sign(X)X]=\bE|X|=\sqrt{\frac{2}{\pi}} |\xi|=\alpha(\xi)\bE|X|^2=\alpha(\xi)|\xi|^2
\]
Thus $\alpha(\xi)=\sqrt{\frac{2}{\pi}}\frac{1}{|\xi|}$ and 
\begin{align*}
\widehat{\displaystyle{Q_y(Hf(\cdot, y))}}(\xi)&=i\,Q_y\big[\sign(\xi\cdot y)\big]\,\widehat f(\xi)\\
&=\sqrt{\frac{2}{\pi}}\, \frac{i\,\xi\cdot y}{|\xi|} \widehat f(\xi)\\
&=\sqrt{\frac{2}{\pi}}\, \sum_{j=1}^d\frac{i\,\xi_j\, y_j}{|\xi|} \widehat f(\xi)\\
&=\sqrt{\frac{2}{\pi}}\, \sum_{j=1}^d\widehat{R_jf}(\xi)\,y_j, 
\end{align*} 
which gives \eqref{eq:pisier-1.5} and completes the proof of the theorem. 
\end{proof}

\begin{remark}[Complex-valued functions; Pisier's method] As Pisier points out in \cite[Remark, p.~498]{Pis}, his  method applies to functions taking values in a Banach space with the unconditional martingale  differences property, U.M.D. for short. Details on how the constant change are not given.  Here we explain the modification for the case of complex-values functions. 
\end{remark}
The identity
\begin{equation}\label{eq:pisier-id}
Q_y(\mathcal H f)(x,y)
=
\sqrt{\frac{2}{\pi}}\sum_{j=1}^d y_j\,R_j f(x)
\end{equation}
holds for complex-valued functions $f$, since both $\mathcal H$ and
$Q_y$ are linear operators.
The difference between the real and complex cases appears in the
Gaussian moment step.  For real coefficients $a_j$ one has the exact
identity 
\[\|\sum_{j=1}^d y_j a_j\|_{L^p(\gamma)}
=\gamma(p) \left(\sum_{j=1}^d a_j^2\right)^{1/2}.
\] 
When the coefficients $a_j$ are complex this identity no longer holds.
Instead we  use the  Kahane–Khintchine Gaussian inequality. 
Let $\alpha(p)=\Gamma\!\left(\frac{p+2}{2}\right)^{1/p}$. Then 
\begin{align}\label{Gau:Kih1}
\alpha(p)\,\left(\sum_{j=1}^d |a_j|^2\right)^{1/2}
\le
\left\|\sum_{j=1}^d y_j a_j\right\|_{L^p(\gamma)}
\le
\gamma(p)\,\left(\sum_{j=1}^d |a_j|^2\right)^{1/2}, \,\,  2\leq p<\infty,
\end{align}
\begin{align}\label{Gau:Kih2}
\gamma(p)\,\left(\sum_{j=1}^d |a_j|^2\right)^{1/2}
\le
\left\|\sum_{j=1}^d y_j a_j\right\|_{L^p(\gamma)}
\le
\alpha(p)\,\left(\sum_{j=1}^d |a_j|^2\right)^{1/2}, \,\, 1<p\leq 2
\end{align}

With this, \eqref{eq:Cp-moreexplicit} becomes 
\begin{equation}\label{eq:Cp-moreexplicit-Compex1}
\|R\|_{L^p(\R^d)}
\le
\begin{cases}
\sqrt{\frac{\pi}{2}}\;\alpha(p)\,\frac{\gamma(q)}{\gamma(p)},\,\,\|H\|_p, &1<p\leq 2,\\  
\sqrt{\frac{\pi}{2}}\; \frac{\gamma(p)}{\alpha(p)}\,\,\, \|H\|_p, \,\,\, & 2<p<\infty. 
\end{cases}
\end{equation}
The constants simplify to
\[
\frac{\gamma_p}{\sqrt{2/\pi}\,\alpha_p}
=
\sqrt{\pi}\,
\left(
\frac{\Gamma\!\left(\frac{p+1}{2}\right)}
{\sqrt{\pi}\,\Gamma\!\left(\frac{p+2}{2}\right)}
\right)^{1/p}
=
\sqrt{\pi}\,C_p(2)\leq \sqrt{\pi}, \quad  2\leq p<\infty,  
\]
where the inequality follows from Remark \ref{Asym:C_p} below.  Similarly, 
\[
\sqrt{\frac{\pi}{2}}\;\alpha_p\,\frac{\gamma_q}{\gamma_p}
\le
\sqrt{\frac{\pi}{2}}\,(p-1)^{-1/2}.
\]
We summarize the above in the following 

\begin{corollary}\label{PisierComplex:Corrolary} The $L^p$-norm of the vector of Riesz transforms when acting on complex-valued functions has the bounds 
\begin{equation}
\|R\|_{L^p(\R^d;\, \bC)}
\le
\begin{cases}
\sqrt{\frac{\pi}{2}}\, \frac{1}{\sqrt{p-1}}\,\,\|H\|_p &1<p<2,\\ 
\sqrt{\pi}\,\,\, \|H\|_p, \,\,\,\,\, & 2\leq p<\infty. 
\end{cases}
\end{equation}
\end{corollary}
 In combination with \eqref{BanWan1} which holds for complex-valued functions we have
 \begin{corollary}\label{SteBestCom}
 \begin{equation}\label{SteBestCom:C}
\|R\|_{L^p(\R^d;\, \bC)}
\le
\begin{cases}
\frac{2}{p-1}\,\,\, &1<p<2,\\ 
\sqrt{\pi}\,\,\, \|H\|_p, \,\,\,\,\, & 2\leq p<\infty. 
\end{cases}
\end{equation}

 \end{corollary}

We observe that for complex-valued functions and $d=2$, the  method of rotations \eqref{k-rotation} gives 
\[\|\BH_{\bC}\|_{L^p(\R^2;\,\bC)}\leq \frac{\pi}{2}\cot\Big(\frac{\pi}{2p}\Big), \,\,\, 2\leq p<\infty,\]
and the martingale inequality \eqref{Complex:H3} gives  
\[
\|\BH_{\bC}\|_{L^p(\R^2;\,\bC)}\leq \sqrt{2}\cot\Big(\frac{\pi}{2p}\Big), \,\,\, 2\leq p<\infty,  
\]
which are better than the estimate  in the Corollary \ref{SteBestCom} for the particular case of $d=2$. The "dimensional dependence" of the constant $\sqrt{2}$ in the second inequality will be clarified below.

\subsection{Asymptotic bounds for vector of Riesz transforms, averaging} A standard technique  in probability and analysis  when estimating norms of operators in Hilbert spaces is to   average with respect to Rademacher, Gaussian, or uniform distributions.  An example of this is Pisier's \eqref{GaussianAve:1} where the length of the vector $v=Rf(x)=\left(R_1f(x), \dots, R_df(x)\right)$ is written as the  expectation  of a mean 0 variance $|v|^2$  normal  random variable. The crux of Pisier’s proof is to combine this simple tool  with the deeper idea of projecting onto the first Wiener chaos, thereby reducing the dimensional dependence from $\R^d$ to $\R$.  In what follows we will explore averaging techniques to obtain sharper asymptotic behavior of $L^p$-norms as $p\to\infty$. We recall the following 
\bigskip

\noindent{\bf Conjecture} (Problem 6, p.~819 \cite{Ban})
\begin{equation}
\|R\|_{L^p(\R^d)}=\|H\|_p=\cot\Big(\frac{\pi}{2p^*}\Big).
\end{equation}
 Although this may be false,  we have the the following   
  \begin{theorem}
 \begin{align}\label{Asymp:R1}
 \lim_{p\to\infty}\frac{\|R\|_{L^p(\R^d)}}{\|H\|_p}=1.
 \end{align} 
 \end{theorem} 
 
 \begin{proof} To proof this, we will use the averaging technique. Given Theorem \ref{Pisier:Main} which produces such a good uniform bound one might expect that applying the averaging technique with Gaussians would yield the best asymptotic result of the  three methods mentioned above. However, as it turns out, averaging with  the Gaussian and uniform distributions give the same result while averaging with the Rademacher distribution is less effective. We explore all three averaging distributions. 
  
We begin with uniform averaging. We denote by $d\sigma$ the normalized surface measure on the $(d-1)$-dimensional sphere $\Soned$. For $d=2$  this is just the arc-length measure on the circle $\Sone$. That is,  $d\sigma(\theta)=\frac{d\theta}{2\pi}$ on $[0,2\pi)$.
Recall that (by rotational invariance of $d\sigma$) for any $x\in \R^d$, $0<p<\infty$, 
\[ \int_{\Soned}|\omega \cdot x|^p\, d\sigma(\omega) = |x|^p \int_{\Soned} |\omega_1|^p\, d\sigma(\omega). 
\]
Define
\begin{equation}\label{eq:cp-def}
C_{p}(d)=\|\omega_1\|_p=\left(\int_{\Soned}|\omega_1|^p\,d\sigma(\omega)\right)^{1/p}=\left(\frac{\Gamma\!\left(\frac d2\right)\Gamma\!\left(\frac{p+1}{2}\right)}
{\sqrt{\pi}\,\Gamma\!\left(\frac{d+p}{2}\right)}\right)^{1/p}.
\end{equation}

\begin{remark}\label{Asym:C_p} Fix $d$. Then $C_p(d)$, as a function of $p$, has the following simple properties that will be use several times below:  
\begin{itemize}
\item[(i)] $C_2(d)=\frac{1}{\sqrt{d}}$. This follows from the  fact that $\Gamma\!\left(\frac{d+2}{2}\right)=\frac{d}{2}\Gamma(\frac{d}{2})$, $\Gamma(3/2)=\sqrt{\pi}/2$. 
\item[(ii)] $C_p(d)$ is increasing, equivalently $\frac{1}{C_p(d)}$ is decreasing,  in $p$.  This follows from Jensen's inequality.  
\item[(iii)] $C_{p}(d)\to \|\omega\|_{\infty}=1$, as $p\to\infty$. This is  a simple exercise in introductory analysis or by looking at asymptotic behavior of the $\Gamma$ function.    
\end{itemize} 
\end{remark} 
  For $\omega=(\omega_1, \dots, \omega_d)\in\Soned$ define the directional Riesz transform
\[
R_\omega f:=\omega_1 R_1 f+\dots +\omega_d R_d f.
\]
Then, 
\[
\left(|R_1f(x))|^2+\dots+|R_df(x)|^2\right)^{p/2}
=\frac{1}{(C_p(d))^p}\int_{\Soned}|R_{\omega}f(x)|^p d\sigma(\omega).
\] 
 Integrating on $x$ gives 
 \[
\|Rf\|_{L^p(\R^d)}^p=\frac{1}{(C_p(d))^p}\int_{\Soned}\|R_{\omega}f(x)\|_{L^p(\R^d)}^p d\sigma(\omega).
\]
We claim that for every $\omega\in\Soned$ and $1<p<\infty$,
\[
\|R_{\omega}\||_{L^p(\R^d)}=\|R_1\|_{L^p(\R^d)}. 
\]
This is proved exactly as in Remark \ref{Riesz;Rotation}. Let  $U\in SO(d)$ be an orthogonal matrix with $Ue_1=\omega$. Define $(Uf)(x):=f(U^{-1}x)$.
Then $U$ is an $L^p$-isometry.  A Fourier multiplier computation shows
$R_{\omega} = U R_1 U^{-1}$. Indeed,
\[
\wh{R_\omega f}(\xi)
= \,i\,\frac{\omega\cdot\xi}{|\xi|}\,\wh f(\xi)
= \,i\,\frac{e_1\cdot(U^{-1}\xi)}{|U^{-1}\xi|}\,\wh f(\xi),
\]
and similarly,
\[
\widehat{U R_1 U^{-1} f}(\xi)
= \widehat{R_1(U^{-1}f)}(U^{-1}\xi)
= \,i\,\frac{e_1\cdot(U^{-1}\xi)}{|U^{-1}\xi|}\,\wh f(\xi).
\]
Thus we have 
\begin{align}\label{Stein:asym}
\|Rf\|_{L^p(\R^d)}\leq \frac{1}{C_p(d)} \|R_1\|_{L^p(\R^d)}=\frac{1}{C_p(d)}\|H\|_p, \,\,\,\, 1< p<\infty. 
\end{align}
 Combined with the trivial bound from below we have 
\[
1\leq \frac{\|R\|_{L^p(\R^d)}}{\|H\|_p}\leq \frac{1}{C_p(d)},
\]
and  \eqref{Asymp:R1} follows from (iii) in Remark \ref{Asym:C_p}, completing the proof of the proposition.  
\end{proof} 

For fix $d$,   \eqref{Stein:asym} can be used to give universal bounds for $\|R\|_{L^p(\R^d)}$ which,  unfortunately,  are not  dimension independent and  in fact they are of order  $\sqrt{d}$.  More precisely, since $\frac{1}{C_p(d)}$ is decreasing in $p$ we have

\begin{align}\label{Stein:asym1}
\|R\|_{L^p(\R^d)}\leq \frac{1}{C_1(d)}\|H\|_p,  \,\,\, 1<p<\infty\
\end{align}
or the explicit  bound bound using (i) 
\begin{align}\label{Stein:asym2}
\|R\|_{L^p(\R^d)}\leq \frac{1}{C_2(d)}\|H\|_p=\sqrt{d}\,\|H\|_p,  \,\,\, 2\leq p<\infty\
\end{align}

These type of bounds can just be obtained from $\ell^2, \ell^1$ comparisons. But it is  interesting to note that the bound for $p\geq 2$ in the preceding inequity is exactly the same bound  that follows from the vector-valued version of the orthogonal subordination argument using \cite[Theorem A]{BanuelosWang1996}.  It is precisely for this reason that Burkholder's  inequality is used there to obtain the bound $2(p^*-1)$ in \eqref{BanWan1}. More precisely consider the  matrices $\BR_j$ with the orthogonality property for which the projection operator  $T_{\BR_j}=R_j$,  and the matrices $\BR_j^1$ which do not have the orthogonal property but for which $T_{\BR_j^1}=\frac{1}{2}R_j$, as in Lemma \ref{Lemma:Cotlar}. For a real-valued martingale $M_t$, the quadratic variation of the vector valued martingale transform has $\sum_{j=1}^d\ip{\BR_j*M_t}\leq d\ip{M_t}$ and the inequality from \cite[Theorem A]{BanuelosWang1996} (also \cite{BW95}) gives \eqref{Stein:asym2}. On the other hand, $\sum_{j=1}^d\ip{\BR_j^1*M_t}\leq \ip{M_t}$ and Burkholder's vector-valued inequality, which does not require orthogonally give,  the bound $(p^*-1)$ for the vector of $1/2$ Riesz transforms.  This is the argument that proves \eqref{BanWan1}.

 Next we consider  Gaussian averaging. This, essentially,  follows from \eqref{GaussianAve:1}.  Here we re-do the identity with different notation to simplify matters a bit. Set $Z=(Z_1, \dots, Z_d)$, where $Z_j$ are i.i.d. $N(0, 1)$ random variables. Then $|Z|^2$ is a chi-square random variable with $d$-degrees of freedom, that is,  a Gamma($d/2$, $2$) distribution. From this it follows that the density of $|Z|$ is given by 
\[f_{|Z|}(r)
=
\frac{1}{2^{\frac d2-1}\Gamma\!\left(\frac d2\right)}
\, r^{d-1}
\, e^{-r^2/2}, \quad  r \ge 0
\]
 and  
\[\bE|Z|^p=2^{p/2}\frac{\Gamma(\frac{p+d}{2})}{\Gamma(\frac{d}{2})}, \quad 0<p<\infty.
\] 
Setting 
\[m(p)=\bE|N(0,1)|^p=
\frac{2^{p/2}\,\Gamma\!\left(\frac{p+1}{2}\right)}
{\sqrt{\pi}}
\]
we see that 
if  $v(x)=(R_1f(x),\dots, R_df(x))$, then $Z\cdot v\sim N(0,|v|^2)$. Hence 
\begin{equation}\label{eq:gaussian-pointwise}
\big(|R_1f(x)|^2+\dots+|R_df(x)|^2\big)^{p/2}
=
\frac{1}{m_p}\,
\bE\big|Z_1R_1f(x)+\dots+Z_dR_2f(x)\big|^p,
\end{equation}
for all $1<p<\infty,$ which is the same as identity \eqref{GaussianAve:1}.
Integrating both sides in $x$ and using Fubini's theorem 
\begin{align}\label{Uniform:ID1}
\|Rf\|_{L^p(\R^d)}^p
=
\frac{1}{m_p}\,
\mathbb E\big\|\,Z_1R_1f+\dots+Z_dR_2f\,\big\|_{L^p(\R^d)}^p.
\end{align}

Fix $\omega\in \Omega$ and let $z=Z(\omega)\in \R^d$, then (by taking  Fourier  transform) 
\[
z_1R_1+\dots+z_dR_d = |z|\,R_\psi,
\qquad
\psi=\frac{z}{|z|}\in\Soned. 
\]
As before, the  rotation invariance of the Riesz transforms gives 
\[
\|R_\psi\|_{L^p(\R^d)} = \|R_1\|_{L^p(\R^d)}.
\]
Therefore,
\[
\|(z_1R_1+\cdots +z_dR_d)f\|_{L^p(\R^d)}^p
\le |z|^p\,\|R_1\|_{L^p(\R^d)}^p\,\|f\|_{L^p(\R^d)}
\]
From  this and \eqref{Uniform:ID1}, 
\begin{align}\label{Uniform:ID2}
\|R\|_{L^p(\R^d)}
\le
\left(\frac{\bE|Z|^p}{m_p}\right)^{1/p}\,
\|R_1\|_{L^p(\R^d)},\quad 1<p<\infty.
\end{align} 
Hence, 
\[\left(\frac{\bE|Z|^p}{m_p}\right)^{1/p}=\left(\frac{\sqrt{\pi}\,\Gamma(\frac{p+d}{2})}{\Gamma(\frac{d}{2})\Gamma\left(\frac{p+1}{2}\right)}\right)^{1/p}=\frac{1}{C_p(d)}, 
\]
which is the same as the constant in \eqref{eq:cp-def}. Thus the  Gaussian averaging gives the same constant as uniform averaging on the sphere and also proves \eqref{Asymp:R1}

Finally, let $\left(\varepsilon_1, \dots, \varepsilon_d\right)$ be the i.i,d Rademacher random variables
$\bP\left(\varepsilon_j=\pm 1\right)=\frac{1}{2}$.  Recall Khintchine  inequality with the best constant:
\[\left(\sum_{j=1}^d a_j^2\right)^{1/2}\leq K_p\left(\bE\Big|\sum_{j=1}^d a_j\varepsilon_j\Big|^p\right)^{1/p}, 
\] 
\[
K_p
=
\begin{cases}
\left(
  \frac{\sqrt{\pi}}
       {2^{p/2}\,\Gamma\!\left(\frac{p+1}{2}\right)}
\right)^{1/p}, & 1< p \le 2, \,\, \\
 1, & 2\leq p<\infty.
\end{cases} 
\]

With this  and the same arguments as above, we obtain 
\begin{align}\label{Khi:1}
\|R\|_{L^p(\R^d)}\leq \sqrt{d}\,\, K_p\|R_1\|_{L^p(\R^d)}, \,\,\,\, 1<p<\infty. 
\end{align} 
Thus Rademacher averaging does not improve  asymptotic bounds as Gaussian and uniform averaging do.

Let us consider the case $d=2$. Checking that $C_1(2)=\frac{2}{\pi}$, inequalities \eqref{Stein:asym1},  \eqref{Stein:asym2} reduce  to 
\begin{align}\label{Stein:asym3}
\|\BH_{\bC}\|_{L^p(\R^2)}\leq \frac{\pi}{2}\|H\|_p,  \,\,\, 1<p<\infty, 
\end{align}
and 
\begin{align}\label{Stein:asym4}
\|\BH_{\bC}\|_{L^p(\R^2)}\leq \sqrt{2}\,\|H\|_p,  \,\,\, 2\leq p<\infty, 
\end{align}
respectively. 
The first is inequality \eqref{Complex:H2} obtained by the method of rotations and the second is inequality \eqref{Complex:H3} obtained by the martingale transforms. Both trivial form the elementary inequality as already mentioned and not as sharp as Pisier's which gives: 
\begin{align}\label{Stein:asym5}
\|\BH_{\bC}\|_{L^p(\R^2)}\leq \sqrt{\frac{\pi}{2}}\,\|H\|_p,  \,\,\, 2\leq p<\infty, 
\end{align}
In addition, \eqref{Asymp:R1} becomes 
 \begin{align}\label{Asymp:BH}
 \lim_{p\to\infty}\frac{\|\BH_{\bC}\|_{L^p(\R^2)}}{\|H\|_p}=1. 
 \end{align} 
 
In \cite[Conjecture 1.5, p. 3]{CarDraKov2023}, it is conjectured that  for all $k\in \N$ and $p\geq 2$, 
\[
\|(\BH_{\bC})^k\|_{L^p(\R^2; \bC)}=\frac{\Gamma(1/p)\,\Gamma(1/q+k/2)}{\Gamma(1/q)\,\Gamma(1/p+k/2)},
\]
where $ q=\frac{p}{p-1}$ and that when $1<p<2$ the identity should still hold, but with the roles
of $p$ and $q$ reversed.  

Taking $k=1$ and setting 
\[
D(p):=\frac{\Gamma(1/p)\,\Gamma(1/q+1/2)}{\Gamma(1/q)\,\Gamma(1/p+1/2)},
\]
one can check that 
\begin{align}\label{Drag:1}
\lim_{p\to \infty}\frac{D(p)}{\|H\|_p}=\frac{\pi}{4}. 
\end{align}

\section{Problem 2: T.~Iwaniec's conjecture} 
\subsection{Brief history} The Beurling-Ahlfors transform is the singular integral operator acting on complex-valued functions on $L^p(\bC)$, i.e., $f:\bC\to\bC$,  defined by 
\begin{align}\label{BA:Def}
    \cB f(z) =- \frac{1}{\pi}\int_{\bC}\frac{f(w)}{(z-w)^2}\, dm(w),
\end{align}
where $dm$ is the Lebesgue measure on $\bC$ identified as $\R^2$. As a Fourier multiplier, 
\begin{equation*}
  \widehat{\cB f}(\xi) = \frac{\overline\xi}{\xi} \widehat{f}(\xi)= \frac{{\overline \xi}^2}{|\xi|^2} \widehat{f}(\xi) 
\end{equation*}
 Identifying $\bC$ with $\R^2$, it follows that  
 \begin{align}\label{BA(2)}
(\BH_{\bC})^2=(R_1+iR_2)^2=
-(R_2^2-R_1^2-2iR_1R_2)=-\cB
 \end{align}
Thus he operator appearing in the Iwaniec-Martin inequality \eqref{IwaMar1} is the square root of $-\cB$.

The Beurling-Ahlfors transform has been extensively studied in the literature in large part due to its connections to many areas of analysis such as regularity theory for quasiconformal mappings,  partial differential equations and the well-known 1982 conjecture of T.~Iwaniec~\cite{Iwa82} concerning its $L^p$-norm.  Using explicit computations with (nearly) extremal functions, it was  proved by  O. Lehto (see for example \cite{Lehto})  that $(p^*-1)\leq \|\cB\|_{L^p(\bC; \,\bC)}$ and Iwaniec's conjecture asserts that  $\|\cB\|_{L^p(\bC;\, \bC)}=(p^*-1)$. As before, the case $p=2$ is trivial via the Fourier transform. For a sample on the large literature related to this conjecture which remains open for all $p\neq 2$, see ~\cite{Astala1, Ban, MR3018958, IwaMar, BanJan, Lehto, MR3558516, Volberg1, DraVol}, and the many references contained therein. To the best of our knowledge the following is the best known  upper bound valid for all $1<p<\infty$:
 \begin{align}\label{1.5bound}
 \|\cB\|_{L^p(\bC;\, \bC)}\leq 1.575(p^*-1).
\end{align}
This  bound is proved in~\cite{BanJan}. 

The first explicit bound $ \|\cB\|_{L^p(\bC;\,\bC)}\leq 4(p^*-1)$ was proved in \cite{BW95} using the representation  in Lemma \ref{Lemma:Cotlar} for second order Riesz transforms as projections of martingale transforms and  applying the $(p^*-1)$ bound for martingale transforms of 
Burkholder \cite{Burk84}; (1) in Theorem \ref{thm:A}.    The bound was improved to $2(p^*-1)$ in \cite{NazarovVolberg2003} using Bellman function techniques for the heat equation. In \cite{BanMen} the martingale arguments in \cite{BW95} are redone  to represent the second order Riesz transforms as martingale transforms of space-time Brownian martingales. This approach, which subsequently has seen many uses in different geometry settings beyond Euclidean spaces (see for example \cite{BBLS2021} and the many references contained therein) reproves the $2(p^*-1)$ for $\cB$ and gives the  bound $(p^*-1)$ for the operators $R_j^2-R_k^2$ and $2R_jR_k$ for any $j\neq k$ for any $d\geq 2$. These bounds were shown to be sharp in \cite{GesMonSak}. There are also estimates improving the bounds for $\|\cB\|_p$ asymptotically as $\p\to\infty$.  See for example \cite{DraVol4}.  We return to the asymptotics  below.
 
\subsection{Conformal pairs}

 \begin{definition} Let $A$ and $B$ be two  real-valued  $d\times d$ matrices.  We will say that the pair $(A, B)$ is $d\times d$ conformal if it satisfies the following two properties: 
\begin{align}\label{conformal-Pair}
(i)\quad  |Ax|=|Bx| \quad\text{and}\quad (ii) \quad \ip{Ax,Bx}=0,  \quad\text{for all }x\in\mathbb{R}^d. 
\end{align}
 \end{definition}  
 
 \begin{remark}\label{ConformalPairs} We arrived at this definition based on the example of interest given in \eqref{BA-matrices} below. 
The ``conformal pairs'' considered here are closely related to classical
constructions in matrix analysis.  The conditions $A^{T}A = B^{T}B$,  and that 
$A^{T}B$ is skew-symmetric, imply that $A$ and $B$ share the same positive
semidefinite factor in their polar decompositions, and that $B$ is obtained
from $A$ by composing with an orthogonal complex structure on the image of $A$.
 This is also related to the notion of partial isometries of  operators on Hilbert spaces; see  Horn and
Johnson's "Matrix Analysis" \cite{HornJohnsonMatrixAnalysis}. 
\end{remark}

\begin{definition}[Martingale transform norm]\label{def:OpNorm}
Let $A$ be a real $m\times n$ matrix. For $1<p<\infty$, define the operator norm
of the martingale transform induced by $A$ on $L^p(\Omega)$ by
\begin{equation}\label{Operator:NormA}
\|A\|_{L^p(\Omega)}
:=
\sup_{K\not\equiv 0}
\frac{\|A*M\|_p}{\|M\|_p},
\qquad
M_t=\int_0^t K_s\cdot dB_s,
\end{equation}
where $B_s$ is  $n$--dimensional Brownian motion and $K$ ranges over predictable
$\R^n$--valued processes with $\int_0^t |K_s|^2\,ds<\infty,$ a.s.
\end{definition}

 For a $A$  real $m\times n$ matrix, Burkholder's inequality \eqref{Burkholder} gives  
\begin{align}\label{Bur84}
\|A\|_{L^p(\Omega)}
\le (p^*-1)\,\|A\|_{\ell^2\to\ell^2}, \qquad 1<p<\infty,
\end{align}
where $\|A\|_{\ell^2\to\ell^2}=\sup_{|x|=1}|Ax|$ is the Euclidean operator norm. 

\begin{lemma}[Rotational invariance for conformal martingales]\label{lem:rotinv-conformal}
Let $Z_t=X_t+iY_t$ be a conformal martingale on the Brownian filtration.   
Then for every $\theta\in\R$ and every $t\ge0$,
\[
X_t\cos\theta+Y_t\sin\theta \ \stackrel{d}{=} \ X_t,
\]
and consequently, for $1<p<\infty$,
\[
\|X\cos\theta+Y\sin\theta\|_p=\|X\|_p.
\]
\end{lemma}

\begin{proof}
Set $\sigma_t:=\ip{X}_t=\ip{Y}_t$. By the
Dambis--Dubins--Schwarz theorem \cite[p.172]{Baudoin2014},  there exists a standard planar
Brownian motion $W$ such that $(X_t,Y_t)=(W^1_{\sigma_t},W^2_{\sigma_t})$.
Let $R_\theta\in SO(2)$ be the rotation matrix 
\[R_\theta =\bmatrix \cos(\theta) & \ \sin(\theta) \\ -  \sin(\theta) &  \cos(\theta) \endbmatrix
\]
Then
\[
(X_t\cos\theta+Y_t\sin\theta,\ -X_t\sin\theta+Y_t\cos\theta)
=R_\theta (X_t,Y_t)=R_\theta W_{\sigma_t}.
\]
Since $R_\theta W$ is again standard planar Brownian motion, $R_\theta W_{\sigma_t}$
has the same law as $W_{\sigma_t}$, hence 
$X_t\cos\theta+Y_t\sin\theta\stackrel{d}{=}X_t$. Taking $L^p$ norms proves  the lemma.
\end{proof}

\begin{theorem}[Conformal martingales transforms; real-valued]\label{thm:conformal-AB-R}  
Let $M_t=\int_0^t K_s\cdot dB_s$ be  a martingale on the filtration of $d$-dimensional Brownian motion as in \eqref{StochasticMartin}.   
 Let $(A,B)$ be a conformal $d\times d$ pair and suppose in addition that $\|A\|=\|B\|\leq 1$. 
Define $X_t:=A*M_t$ and $Y_t:=B*M_t.$ Set 
\begin{align*}
Z_t:=X_t+iY_t=(A+iB)*M_t. 
\end{align*}
Then $Z_t$ is a conformal martingale and 
\begin{align}\label{MainClaim:2}
\|A+iB\|_{L^p(\Omega)}\leq 
\frac{1}{C_p(2)}(p^*-1), \quad 1<p<\infty,  
\end{align}
where $C_p(2)$ is as in Remark \ref{Asym:C_p}. 
\end{theorem}

\begin{proof}
The fact that $Z_t$ is conformal follows  trivially from the fact that $(A, B)$ is a $d\times d$ conformal pair.  From Lemma \ref{lem:rotinv-conformal} and the averaging property we have 
\begin{align*}
\bE|Z_t|^p=\bE\left(|X_t|^2+|Y_t|^2\right)^{p/2}&=\frac{1}{C_p(2)^p}\int_{\Sone}\bE|X\cos\theta+Y\sin\theta|^p d\sigma\\
&=\frac{1}{C_p(2)^p}\|X\|_p^p, \quad 1<p<\infty 
\end{align*}
Thus, 
\begin{align}\label{Real-1}
\|(A+iB)*M_t\|_p &\leq  \frac{1}{C_p(2)}\|(A+iB)*M_t\|_p\\
& \leq \frac{1}{C_p(2)} (p^*-1)\|M_t\|_p, \quad 1<p<\infty, \nonumber
\end{align}
where we used Burkholder inequality for the last step. 
\end{proof}

\begin{theorem}[Conformal martingale transforms; complex-valued]\label{thm:conformal-AB-C} Let 
$M_t= \int_0^t H_s \cdot dB_s$ and 
$N_t= \int_0^t K_s \cdot dB_s$ 
be two martingales as in \eqref{StochasticMartin} and $(A,B)$ conformal $d\times d$ pair with $\|A\|=\|B\|\leq 1$.
Define
\begin{align*}
Z_t = X_t + i Y_t=\left(A+iB)*(M_t+iN_t\right),
\end{align*}
where 
\begin{align*}
X_t= \left(A*M_t - B*N_t\right), \qquad 
Y_t = \left(A*N_t + B*M_t\right). 
\end{align*} 
Then $Z_t$ is a conformal martingale and  
\begin{equation}\label{Complex-2}
\|(A+iB)*(M_t+iN_t)\|_p\leq \frac{\sqrt{2}}{C_p(2)}\, (p^*-1)\|M_t+iN_t\|_p, \quad 1<p<\infty. 
\end{equation} 
\end{theorem}

\begin{proof}
\[\ip{X}_t
=
\int_0^t|AH_s - BK_s|^2\,ds,
\qquad
\ip{Y}_t
=
\int_0^t|AK_s + BH_s|^2\,ds,
\]
and
\[
\ip{X,Y}_t
=
\int_0^t(AH_s - BK_s)\cdot(AK_s + BH_s)\,ds.
\]
Expanding the norms  we obtain
\[
|AH_s - BK_s|^2
=
|AH_s|^2 + |BK_s|^2 - 2(AH_s)\cdot(BK_s),
\]
\[
|AK_s + BH_s|^2
=
|AK_s|^2 + |BH_s|^2 + 2(AK_s)\cdot(BH_s).
\]
This together with  (i) and (ii) shows that $Z_t$ is a conformal.  

By Lemma \ref{lem:rotinv-conformal} we have
\[
\|Z_t\|_p\leq \frac{1}{C_p(2)}\|X_t\|_p, \qquad 1<p<\infty, 
\] 
which in this case is equivalent  to
\begin{equation}\label{Conf:3}
\|(A+iB)*(M_t+iN_t)\|_p\leq \frac{1}{C_p(2)} \|A*M_t - B*N_t\|_p, \qquad 1<p<\infty.
\end{equation}
For the pair $(A,B)$ $d\times d$ matrices define the $d\times 2d$ block matrix 
\[
T=[A\ \ B]
=
\begin{pmatrix}
a_{11} & \cdots & a_{1d} & b_{11} & \cdots & b_{1d}
\\
\vdots & & \vdots & \vdots & & \vdots
\\
a_{d1} & \cdots & a_{dd} & b_{d1} & \cdots & b_{dd}
\end{pmatrix}
\]
so that  $X_t=A*M_t - B*N_t= T*\bM_t$ where,
\[\bM_t:=\binom{M_t}{-N_t}.
\]
Since  $\|A\|=\|B\|\leq 1$, $\|T\|\leq \sqrt{2}$, \eqref{Conf:3} combined with Burkholder's  inequality give 
\[\|(A+iB)*(M_t+iN_t)\|_p\leq \frac{\sqrt{2}}{C_p(2)}(p^*-1)\|M_t+iN_t\|_p,\quad 1<p<\infty.
\]
\end{proof}

As we have already seen in Section~\ref{TheKeyConnection}, conformal martingales and their connections to holomorphic functions on $\bC$ and $\bC^n$ have been studied for many years. For many example on $\bC$ we refer the reader to \cite{Dur}.  
The connections of conformal martingales to singular integral, beyond the Hilbert transforms, first emerged with  
\begin{problem}{ \cite[p.~599]{BW95}} Determine the best constant $C_p$ in the inequality. 
\begin{align}\label{BetterBurk}
\|Y\|_p\leq C_p\|X\|_p,
\end{align}
when $Y$ is a conformal martingale subordinate to $X$.
\end{problem} 
 The inequality implies bounds for the $p$-norm of the Beurling-Ahlfors operator.  
 While Burkholder’s inequality already yields the bound $(p^*-1)$, one expects the conformal condition to lead to a smaller constant. This problem was the impetus for the investigations in \cite{BanJan}, where a slight modification of Burkholder’s arguments led to the following result

\begin{knowntheorem}(\cite{BanJan})\label{BanJanOriginalTheorem}
Suppose that $Y$ and $X$ are $\R^2$-valued martingales, with $Y$ conformal and subordinate to $X$, and $X$ arbitrary. Then
\begin{align}\label{BanJan:BurImprove}
\|Y\|_p\leq \sqrt{\frac{p(p-1)}{2}} \|X\|_p, \quad 2\leq p<\infty. 
\end{align} 
\end{knowntheorem} 
Applying this to the the martingales in Theorem \ref{thm:conformal-AB-R} (real case) and Theorem \ref{thm:conformal-AB-C} (complex case) give, respectively, 

\begin{align}
 \|(A+iB)M\|_p
    \le \sqrt{p(p-1)}\,\|M\|_p, \quad 2\leq p<\infty,\label{BanJan08-R} 
\end{align}
and 
\begin{align} 
 \|(A+iB)*(M+iN)\|_p   \le \sqrt{2p(p-1)}\,\|M+iN\|_p, \quad 2<p<\infty. \label{BanJan08-C}
\end{align}

These bounds, applied to conformal pairs $(A_0, B_0)$ in \eqref{Example:D2} together with an interpolation argument using the fact that $\|\cB\|_2=1$, yield the bound \eqref{1.5bound} proved in \cite{BanJan}. Here we apply the same method to general pairs of $d\times d$ conformal matrices to obtain the corresponding result on $\R^d$.

 As mentioned in Remark \ref{ConformalPairs}, our definition of conformal pairs $(A, B)$ was  motivated by the following example.
 \begin{example}  
Fix $j,\, k\in \{1,\dots,d\}$ with $j\neq k$.  Let $\cA^{rs}$ denote the
$d\times d$ matrix whose only nonzero entry is $(\cA^{rs})_{r,s}=-1.$ Define the diagonal and off--diagonal matrices, by  
\begin{align}\label{BA-matrices} 
A^{(j,k)}:=\cA^{kk}-\cA^{jj}, \quad B^{(j, k)}=\cA^{jk}+\cA^{kj}.
\end{align} 
Then for every $x\in \R^d$,
\[
(A^{(j, k)}x)_j=x_j,\quad (A^{(j, k)}x)_k=-x_k,
\qquad (B^{(j, k)}x)_j=-x_k,\quad (B^{(j, k)}x)_k=-x_j,  
\]
and all other components vanish. Consequently,
\[
|A^{(j, k)}x|^2=x_j^2+x_k^2=|B^{(j, k)}x|^2,
\quad
A^{(j, k)}x\cdot B^{(j, k)}x=0. 
\]
Thus  the pair $(A^{(j, k)},B^{(j,k)})$ satisfies the conformal conditions (i), (ii) in \eqref{conformal-Pair} and moreover  $\|A^{(j, k)}\|=\|B^{(j, k)}\|=1$. Hence, Theorems  \ref{thm:conformal-AB-R} and \ref{thm:conformal-AB-C} apply along with the  estimates \eqref{BanJan08-R} and \eqref{BanJan08-C}. 
\end{example} 
\begin{remark}\label{A0B0:Cpair}
The above matrices are $d\times d$ projections of the $(d+1)\times (d+1)$ matrices define in \eqref{Rieszd=2}-\eqref{RieszMatr3}. 

Of particular interest here is the case when $d=2$.  Define the $2\times 2$ matrices  $A_0$, $B_0$ and their sum $A_0+iB_0$ by 
\begin{align}\label{Example:D2} 
A_0=\bmatrix 1 & 0 \\ 0 & -1 \endbmatrix,\,\,  B_0= \bmatrix 0 & -1 \\ -1 & 0  \endbmatrix,\,\,  A_0+iB_0=\bmatrix 1 & -i \\ -i & -1 \endbmatrix. 
\end{align} 
Then $(A_0, B_0)$ is a conformal pair  and 
$A_0+iB_0$ is the matrix in \cite{BanMen} that gives  the Beurling-Ahlfors operator as a conditional expectation of a martingale transform. 
\end{remark}

\subsection{Projections of conformal martingales} 
As we saw in  Lemma~\ref{Lemma:Cotlar}, the second-order Riesz transforms can be written as conditional expectations of martingale transforms, as was originally shown by Gundy and Varopoulos in \cite{GV79},  using the harmonic extension $U_f(x,y)=P_yf(x)$, where $P_y$ is the Poisson semigroup (convolution with the Poisson kernel) composed with Brownian motion $(B_t, Y_t)$ in the upper half-space $\R^{d+1}_{+}=\{x\in \R^d, y>0\}$.  However, the alternative representation based on the heat extension $V_f(x, t)=T_tf(x)$ where $T_t$ is the heat semigroup (convolution with the heat kernel) and the underlying process is space–time Brownian motion $(B_t, a-s)$, $0<s\leq a$,  improves  their $L^p$-norm estimates by a factor of $2$. 
This simple and natural modification of the Gundy-Varopoulos representation first appeared in \cite{BanMen}.

Fix $a>0$ and let $S_t=(W_t,a-t)$, $0<s\leq a$, denote 
space-time Brownian motion started at $(0, a)$ in $\R^d\times(0,\infty)$ where 
$W_t$ is standard Brownian motion on the $\R^d$. 
Let $\bP_x$ and $\bE_x$ denote the probability and expectation
for processes starting at the point $(x, a)$ and $\bP^a$ denote the
``probability'' measure associated with the process with initial distribution
$m\otimes\delta_a$ and  by $\bE^a$ the corresponding expectation.

Let $F=f_1+if_2\in C_c^\infty(\R^d)$ and let $V_{f_1}(x,t)=T_tf_1(x)$,  and similarly for $f_2$,  be their heat extensions 
to the upper half-space.  We denote by $\nabla V_f$ the gradient of $V$ in the $x$ variable.  
 By It\^o's formula that for $0<t\leq a$, 
\begin{align*}
V_F\,(S_t)-V_F\,(S_0)
&=
\int_0^t \nabla V_F\,(S_t)\cdot dW_t\\
&=\int_0^t \nabla V_{f_1}(S_t)\cdot dW_t+i\int_0^t \nabla V_{f_2}(S_t)\cdot dW_t,  \\
&=M_t+iN_t
\end{align*}
hence  $M$ and $N$ are Brownian martingales of the form \eqref{StochasticMartin}. Let  $A^{(j, k)}$ and $B^{(j, k)}$ be as in  \eqref{BA-matrices}.   

We define the projection operator 
\begin{align}\label{HeatCondExpe-1}
\cT^a_{(A^{(j, k)}, B^{(j, k)})}F(x)&=\bE^a\Big[ (A^{jk}+iB^{jk})*[M_t+iN_t]\Big| S_a=(x, 0)\Big]\\
&=\bE^a\Big[ (A^{(j, k)}+iB^{(j, k)})*[M_t+iN_t]\Big| W_a=x \Big]\notag.
\end{align}
By Proposition 2.2 in \cite{BanMen}, 
\begin{align}\label{HeatCondExpe-2} 
\lim_{a\to\infty}\cT^a_{(A^{(j, k)}, B^{(j, k)})}F(x)=\bB_{(A^{(j, k)}, B^{(j, k)})}F=(R^{(jj)}-R^{(kk)}+2iR^{(jk)})F,\,\,\,  \text{in}\,\, L^2, 
\end{align} 
where $R^{jk}$ are the second order Riesz transforms 
\begin{align}\label{secondRiesz} 
    R^{(jk)}f(x)
    =\int_0^{\infty} \frac{\partial^2 V_f(x, t)}{\partial x_j\partial x_k}\, dt
    =\frac{\partial^2 }{\partial x_j\partial x_k}(-\Delta)^{-1}f(x), 
\end{align}
\[
    \widehat{{R^{(jk)}f}}(\xi)=-\frac{\xi_j\xi_k}{|\xi|^2}\widehat{f}(\xi).
\]

  Before we proceed to discuss $L^p$ boundedness properties,  let us look a bit more carefully at the martingale structure of these operators. Recall  that $j, k\in \{1, \dots, d\}$ and  $j\neq k$. To  simplify notation, set 
\begin{equation}\label{SimNot}
\cB^{jk}(d)=\left(R^{(jj)}-R^{(kk)}+2iR^{(jk)}\right)=(R_j+iR_k)^2
\end{equation}
so that $\cB^{12}(2)=(\BH_{\bC})^2=-\cB$. Since $(A^{jk}, B^{jk})$ is a conformal pair if and only if $(B^{jk}, A^{jk})$ is a conformal pair, the operators $\cB^{jk}$ are essentially the same as  $\cB^{kj}$. Therefore we may assume that $j<k$.  

Suppose further that $d$ is even, say $d=2n$, and identify $R^{2n}\cong C^n$ which decomposes orthogonally as a direct sum of coordinate
$2$--planes
\begin{equation}\label{eq:plane-decomp}
R^{2n}=\bigoplus_{m=1}^n E_m,
\quad
E_m:=\mathrm{span}\{e_{2m-1},e_{2m}\}.
\end{equation}
This gives the  operators 
\begin{align}\label{BA:1-d}
\cB_m(d)=(\BH_{\bC}(m))^2=(R_{2m-1}+iR_{2m})^2, \quad m=1, \dots n,
\end{align}
which are Beurling-Ahlfors operators on  $E_m\cong \bC$. Setting $z_m=x_{2m-1}+i x_{2m}$, $m=1,\dots,n$ and 
recalling the complex derivatives
\[
\partial_{z_m}=\frac12\big(\partial_{x_{2m-1}}-i\partial_{x_{2m}}\big),
\qquad
\partial_{\bar z_m}=\frac12\big(\partial_{x_{2m-1}}+i\partial_{x_{2m}}\big).
\]
A simple computation then gives that 
\begin{equation} 
\cB_m(d)
=
(R_{2m-1}+iR_{2m})^2
=
4\,\partial_{\bar z_m}^2\,(-\Delta)^{-1}, 
\end{equation}
exactly as in the complex plane $\bC$. 

These operators are the projection of the martingale transform of  $M_t+iN_t\in \bC$ obtained as above by composing the the heat extension of the complex-valued function $F=f_1+if_2$ with space-time Brownian motion and taking the martingale transform with the matrices 
\begin{align}
C_m=A^{(2m-1,2m)}+iB^{(2m-1,2m)}=
\begin{pmatrix}
0 &        &        &        &        \\
  & \ddots &        &        &        \\
  &        & A_0+iB_0 &      &        \\
  &        &        & \ddots &        \\
  &        &        &        & 0
\end{pmatrix},
\end{align}
where the only nonzero $2\times2$ block occurs in the rows and columns
$(2m-1,2m)$ and $A_0+iB_0$ is the $2\times 2$ matrix in 
\eqref{Example:D2}. 
Equivalently,
\[
C_m=A^{(2m-1,2m)}+iB^{(2m-1,2m)}
=
\mathrm{diag}\Bigl(
0,\dots,0,
\underbrace{A_0+iB_0}_{(2m-1,\,2m)\text{ block}}
0,\dots,0
\Bigr).
\]
From the fact that 
\[ A_0+iB_0=\bmatrix 1 & -i \\ -i & -1 \endbmatrix 
\]
it follows that for $z\in \bC^n$, 
\[C_m z
=
\bigl(0,\dots,0,\,
z_{2m-1}-i z_{2m},\,
- i z_{2m-1}-z_{2m},\,
0,\dots,0\bigr)
\]
and 
\[\|C_m\|_{\bC^n\to\bC^n}=2. 
\] In fact, due to the  block structure, the following matrix which is the sum of the $C_m$ also has norm 2.
\begin{align}
C=
\begin{pmatrix}
A_0+iB_0 &        &        &        &        \\
  & \ddots &        &        &        \\
  &        & A_0+iB_0 &      &        \\
  &        &        & \ddots &        \\
  &        &        &        & A_0+iB_0
\end{pmatrix},
\end{align}

By \eqref{HeatCondExpe-1} and \eqref{BA:1-d} we have the projection operators
\begin{align}
\cB_m(d)F=(R_{2m-1}+iR_{2m})^2 F, \quad 1\leq m\leq n.
\end{align}
\begin{theorem} For $m=1, \cdots n$, set 
\begin{align}
Z^m=C_m*(M+iN) 
\end{align} 
and 
\begin{align}\label{concrete:ecample}
Z=(Z^1, Z^2, \dots Z^n)\in \bC^n. 
\end{align}
$Z$ is a conformal martingale with the orthogonal property. Furthermore, the vector operator defined for $F:\R^d\to \bC$ by 
\[ 
\cB(d)F=\left(\cB_1(d)F, \dots, \cB_n(d)F\right)
\]
satisfies
\begin{align}\label{SteinBA:type}
\|\cB(d)\|_{L^p(\R^d; \bC)}\leq 2(p^*-1), 
\quad 1<p<\infty. 
\end{align}
\end{theorem} 

\begin{remark} This is the analogue  for the Beurling-Ahlfors operators $\cB$ on $R^{2d}$ of inequality \eqref{BanWan1} for the Riesz transforms on $\R^d$. In fact, as it was already pointed out, inequality \eqref{BanWan1} also holds for complex-valued functions.
\end{remark}

\begin{proof}
By Theorem \ref{thm:conformal-AB-C} each $Z_m$ is conformal and  as in (ii) of Definition \ref{def:conformal},  
\begin{align}\label{concrete:ecample}
Z=(Z^1, Z^2, \dots Z^n)\in \bC^n
\end{align}
is also  conformal.   It remains to show that   $\ip{Z^m, Z^l}=0$,  $l\neq m$ to satisfy property (iii)  of Definition \ref{def:conformal}.  This follows from the fact the pairs $(A^{(2m-1,2m)}, B^{(2m-1,2m)})$ act  on different blocks.  That is, with the matrix as above, 
\[
A^{(2m-1,2m)}+iB^{(2m-1,2m)}
=
\mathrm{diag}\Bigl(
0,\dots,0,
\underbrace{A_0+iB_0}_{(2m-1,\,2m)\text{ block}}
0,\dots,0
\Bigr),
\]
the block  structure guarantees that $\ip{Z^m, Z^{\ell}}_t=0$, $\ell\neq m$.  Thus $Z$ satisfies (iii) of Definition \ref{def:conformal}. Note that similarly $\ip{Z^m, \overline{Z^{\ell}}}=0,$ for all $\ell\neq m$.

 Since the matrix $C$, which is the sum of the matrices $C^m$, has norm 2, Burkholder's inequality (exactly as in \cite[(3.4.3), p.~815]{Ban}) immediately gives \eqref{SteinBA:type}. 
\end{proof}

Below we give two examples that show more explicitly  how the structure of the matrices $C^m$ give the orthogonality  property $\ip{Z^j, Z^\ell}_t=0$, for $\ell\ne j$. 

\begin{example}[$d=4$]\label{ex:d4}
Let $W=(W^1,W^2,W^3,W^4)$ be a $4$--dimensional Brownian motion and define
\[
M_t=\int_0^t H_s\cdot dW_s, 
\quad 
N_t=\int_0^t G_s\cdot dW_s,
\]
where $H_s,G_s\in\R^4$ are predictable processes.  Set
\[
U_s:=H_s+iG_s=(u_1,u_2,u_3,u_4)\in\bC^4 .
\]
 There are two pairs to consider, $(1, 2), (3, 4)$.
 
\medskip
\noindent
{\bf (1) The pair $(1,2)$.}
\[
A^{(1,2)}=
\begin{pmatrix}
1&0&0&0\\
0&-1&0&0\\
0&0&0&0\\
0&0&0&0
\end{pmatrix},
\qquad
B^{(1,2)}=
\begin{pmatrix}
0&-1&0&0\\
-1&0&0&0\\
0&0&0&0\\
0&0&0&0
\end{pmatrix}.
\]
Then
\[
(A^{(1,2)}+iB^{(1,2)})U
=\bigl(u_1-i u_2,\,-u_2-i u_1,\,0,\,0\bigr),
\]
and 
\[
Z_t^1:=\int_0^t (A^{(1,2)}+iB^{(1,2)})U_s\cdot dW_s .
\]

\medskip
\noindent
{\bf (2) The pair $(3,4)$.}
\[
A^{(3,4)}=
\begin{pmatrix}
0&0&0&0\\
0&0&0&0\\
0&0&1&0\\
0&0&0&-1
\end{pmatrix},
\qquad
B^{(3,4)}=
\begin{pmatrix}
0&0&0&0\\
0&0&0&0\\
0&0&0&-1\\
0&0&-1&0
\end{pmatrix}.
\]
Then
\[
(A^{(3,4)}+iB^{(3, 4)})U
=\bigl(0,\,0,\,u_3-i u_4,\,-u_4-i u_3\bigr),
\]
and
\[
Z_t^2:=\int_0^t (A^{(3,4)}+iB^{(3,4)})U_s\cdot dW_s .
\]
Thus
\[
d\langle Z^1,Z^2\rangle_t
=\Big((A^{(1,2)}+iB^{(1,2)})U_t\Big)\cdot
\Big((A^{(3,4)}+iB^{(3,4)})U_t\Big)\,dt =0, 
\]
and this gives  $\ip{Z^1, Z^2}_t=0$.
\end{example}
\begin{example}[$d=6$]\label{ex:d6}
Let $W=(W^1,\dots,W^6)$ be a $6$--dimensional Brownian motion and let
\[
M_t=\int_0^t H_s\cdot dW_s, 
\quad 
N_t=\int_0^t G_s\cdot dW_s,
\]
where $H_s,G_s\in\R^6$ are predictable processes.  Set
\[
U_s:=H_s+iG_s=(u_1,u_2,u_3,u_4,u_5,u_6)\in\bC^6 .
\]
There are three pairs to consider $(1,2), (3, 4), (5, 6)$ and this already shows the general proof as above. 
For $m=1,2,3$,  set 
\[
A_m:=A^{(2m-1,2m)},\quad B_m:=B^{(2m-1,2m)},
\quad 
Z_t^m:=\int_0^t (A_m+iB_m)\,U_s\cdot dW_s .
\]

\medskip
\noindent{\bf The three blocks.}
A direct computation (identical to the $d=4$ case, block by block) gives
\begin{align*}
(A_1+iB_1)U &= (u_1-i u_2,\,-u_2-i u_1,\,0,\,0,\,0,\,0),\\
(A_2+iB_2)U &= (0,\,0,\,u_3-i u_4,\,-u_4-i u_3,\,0,\,0),\\
(A_3+iB_3)U &= (0,\,0,\,0,\,0,\,u_5-i u_6,\,-u_6-i u_5).
\end{align*}
Thus $(A_m+iB_m)U$ is only non-zero on the coordinate pair
$\{2m-1,2m\}$.
This gives 
\[
d\langle Z_m,Z_\ell\rangle_t
=\Big((A_m+iB_m)U_t\Big)\cdot\Big((A_\ell+iB_\ell)U_t\Big)\,dt =0
\]
for $m\neq \ell$.
\begin{remark} Continuing with our assumption that $d=2n$, we note  that the operator $\cB(d)$ is not the same as the classical Beurling-Ahlfors operator $S$ acting on differential forms on $\R^d$ studied in \cite{DonaldsonSullivan1989}, \cite{IwaniecMartin1993}, and with martingale methods in \cite{BanLin}. Setting $S_1=S$ acting on $1$-forms $\Lambda^1$ we see from \cite[p.~234]{BanLin} that $S_1$ is a Fourier multiplier with symbol 
\[m_{S_1}(\xi)=\left( \frac{\xi \otimes \xi - |\xi|^2 I}{|\xi|^2} \right).
\]
That is, 
\[
S_1 = \mathcal{F}^{-1} \left( \frac{\xi \otimes \xi - |\xi|^2 I}{|\xi|^2} \right)\mathcal{F}. 
\]
We extend the operator $\mathcal B(d)$ componentwise to
vector-valued functions $F=(F_1,\dots,F_n):\R^{2n}\to\bC^n$ by
\[
\mathcal B(d)F
=
\bigl(
\mathcal B_1(d)F_1,\dots,\mathcal B_n(d)F_n
\bigr).
\]
With this we can write 
\[
\mathcal B(d)=P S_1 Q,
\]
where 
\[
\widehat{(P S_1 Q f)}(\xi)
=
P\,m_{S_1}(\xi)\,Q\,\widehat F(\xi). 
\]
The  operators $P$ and $Q$  are constant matrices with $|P\|=\|Q\|=1$.  Consequently, they commute with the Fourier transform. In fact, if we  identify \(\mathbb{C}^n\) with \(\mathbb{R}^{2n}\) via 
$z=(z_1,\dots,z_n), \quad z_j = x_j + i y_j
\;\longleftrightarrow\;
(x_1,y_1,\dots,x_n,y_n),  
$
$Q:\mathbb{C}^n \to \mathbb{C}^n$ is given by 
$Q(z_1,\dots,z_n) = (\overline{z_1},\dots,\overline{z_n})$ 
and  represented by the constant real matrix
\[
Q = \mathrm{diag}\!\left(
\begin{pmatrix}
1 & 0\\
0 & -1
\end{pmatrix},
\dots,
\begin{pmatrix}
1 & 0\\
0 & -1
\end{pmatrix}
\right) \in \mathbb{R}^{2n\times 2n}
\]
while 
 \[
P = \mathrm{diag}(I_2,\dots,I_2) \in \mathbb{R}^{2n\times 2n},
\]
where each \(I_2\) denotes the \(2\times2\) identity matrix.
 
 By  \cite[Theorem 1, p.~227]{BanLin}, using the heat extension rather than Poisson extension (see also \cite{Hytonen2009}), we have 
\[
\|S_1\|_{L^p(\R^d;\Lambda^1)\to L^p(\R^d;\Lambda^1)}\leq
\left(3-\frac{1}{n}\right)(p^*-1), \quad 1<p<\infty. 
\]
This gives
\begin{align}\label{SteinBA:type2}
\|\cB(d)\|_{L^p(\R^{2n}); \bC^n)}\leq \left(3-\frac{1}{n}\right)(p^*-1), \quad 1<p<\infty,  
\end{align} 
which is the off-diagonal version of  the inequality \eqref{SteinBA:type}. 
\end{remark} 
\end{example}

We now return to more precise information on the $L^p$-boundedness of the individual operators $\cB^{jk}(d)$ for $1\leq j<k\leq d$. From \eqref{HeatCondExpe-1} and \eqref{HeatCondExpe-2}, the contraction property of the conditional expectation on $L^p$ and Theorems \ref{thm:conformal-AB-R} and \ref{thm:conformal-AB-C},   it follows that 
\begin{equation}\label{Real-4}
\|\cB^{jk}(d)\|_{L^p(\R^d)}\leq \frac{1}{C_p(2)}(p^*-1), \quad 1<p<\infty, 
\end{equation}
when acting   on real-valued functions and that  
\begin{equation}\label{Complex-5}
\\|\cB^{jk}(d)\|_{L^p(\R^d; \bC)}\leq \frac{\sqrt{2}}{C_p(2)}(p^*-1), \quad 1<p<\infty, 
\end{equation}
when acting on complex-valued functions. 

Furthers more, from  inequalities \eqref{BanJan08-R} and \eqref{BanJan08-C} it follows that
\begin{equation}\label{Real-6}
\|\cB^{jk}(d)\|_{L^p(\R^d)}\leq \sqrt{p(p-1)},\quad 2\leq p<\infty, 
\end{equation}
when acting   on real-valued functions and that   
\begin{equation}\label{Complex-7}
\|\cB^{jk}(d)\|_{L^p(\R^d; \bC)}\leq \sqrt{2p(p-1)},\quad 2\leq p<\infty, 
\end{equation}
when acting on complex-valued functions. 

From the asymptotic behavior of   $C_p(2)$ given in (i)-(iii) of Remark \ref{Asym:C_p}, we see that 
\[\limsup_{p\to\infty}\frac{\|\cB^{jk}(d)\|_{L^p(\R^d)}}{(p-1)}\leq 1, 
\]
when  acting on real-valued functions and 
\[\limsup_{p\to\infty}\frac{\|\cB^{jk}(d)\|_{L^p(\R^d; \bC)}}{(p-1)}\leq \sqrt{2}, 
\]
when acting on complex-valued.  

For the case of the Beurling-Ahlfors operator $\cB$
when $d=2$, the asymptotic behavior was proved in \cite{DraVol} prior to the bounds obtained in \cite{BanJan}. Their argument, restricted to the operator $\cB$ is analytic  and different from the martingale argument given above. Clearly the asymptotic  behavior follows from \eqref{Real-6}
and \eqref{Complex-7}.  However, it is still interesting that asymptotic  can be obtained just from the Burkholder  $(p^*-1)$ result. 

Since,
\[\widehat{\cB^{jk}(d)f}(\xi)=\left(\frac{i\xi_j-\xi_k}{|\xi|}\right)^2\widehat{f}(\xi)\
\] 
we have $\|\cB^{jk}(d)\|_{L^2(\R^d)}=1$ and  the exact same interpolation argument as in \cite{BanJan} gives the more general inequality valid for all $d\geq 2$ 
\[
\|\cB^{jk}(d\|_{L^p(\R^d;\bC)}\leq 1.575 (p^*-1), \quad 1<p<\infty. 
\]

The interpolation argument in \cite{BanJan} that leads to the $1.575$ constant is so well  optimized that it does not seem  possible to improve it. For the operator acting on real-valued function the estimate  \eqref{Real-6} and the interpolation argument gives the bound  
\[\|\cB^{jk}(d)\|_{L^p(\R^d)}\leq 1.158 (p^*-1), \quad 1<p<\infty, 
\]
for $\cB^{jk}(d)$ acting on real-valued functions.

 With the improvement \eqref{BanJan:BurImprove} of Burkholder's inequality for conformal martingales given in \cite{BanJan}, two obvious questions arise.  (1) Is the inequality sharp and if not what is the sharp inequality?  (2) Is there a similar improved inequality in the range $1<p<2$? As for (1), we would venture to say that the  inequality is not sharp. Regarding  (2), this and related inequalities are investigated in \cite{BorJanVol2013}, \cite{BorJanVol2013b} and  \cite{BanOseAJM}.  In particular,    
 \begin{knowntheorem}[\cite{BorJanVol2013}]
 Suppose $Y=Y_1+iY_2$ is a conformal martingale and $X=X_1+iX_2$ is an arbitrary martingale. Then 
\begin{equation}\label{Conf:BurkImpro-1}
\|Y\|_p\leq C_p\|X\|_p, \quad 1<p\leq 2, \quad C_p=\frac{1}{\sqrt{2}}\frac{a_p}{1-a_p},
\end{equation} 
 where $a_p$ is the least positive root in the interval $(0, 1)$  of the bounded Laguerre function $L_p$. Furthermore,  the inequality is sharp.  
 \end{knowntheorem}
 With this inequality the authors in \cite{BorJanVol2013} improve the upper bound in \eqref{1.5bound} for large $p$. More precisely, they prove that 
 \begin{align}\label{JanVolDuke}
 \|\cB\|_{L^p(\bC;\, \bC)}\leq 1.3992\,p, \quad \text{for}\quad p\geq 1000. 
 \end{align} 
 Exactly as above, the same result holds $\cB^{jk}(d)$ for any $d\geq 2$.  In fact one can  improve the  $p\geq 100$ in \cite{BorJanVol2013} to $p\geq 500$. By duality and estimates on $a_p$ the authors obtain the following inequality (see \cite[Theorem 10.1, p.~516]{BorJanVol2013})
 \begin{equation}\label{eq:Thm10.1}
\|\cB\|_{L^p(\bC;\,\bC)}
<
\Big(\frac{p+3}{2}\,\pi\Big)^{\!1/(2p)}\,
\frac{p-Q}{Q},\quad p>2, 
\end{equation}
where 
\begin{equation}\label{eq:Qdef}
Q=
1-\sum_{n=2}^{\infty}\frac{(n-2)!}{(n!)^2}.
\end{equation}
 They estimate $Q$ numerically to be $\approx  0.718282$, \cite[(9.1), p.~515]{BorJanVol2013}. The authors than substitute $p=1000$ to get 
 \[
\Big(\frac{1003}{2}\pi\Big)^{1/2000}<1.004,
\quad \text{and}\quad 
1.4Q>1.005, 
\]
from which \eqref{JanVolDuke}  follows. 

As  might be expected,   a bit more careful argument to estimate the quantity on the right of \eqref{eq:Thm10.1} will lower the cutoff point. This is in fact the case. Observe that 
\[
\frac{(n-2)!}{(n!)^2}
=\frac{1}{n(n-1)\,n!}, 
\]
from which it  follows  that $Q=1-(3-e)=e-2$. Now consider the function 
 \begin{equation*}\label{eq:Rdef}
h(p)
=
\frac{1}{p}\,
\Big(\frac{p+3}{2}\,\pi\Big)^{\!1/(2p)}\,
\frac{p-Q}{Q},
\quad p>2.
\end{equation*}
 Taking $\log$ and differentiating it follows  that $h(p)$ is  decreasing for $p\geq 4$. With a target upper bound of $1.4$ testing various values of $p$ we  arrive at a safe interval where $h(p)\leq 1.4$. For example, $h(450)\approx 1.40016908$ and  $h(500)\approx 1.399999$. Thus,
 \begin{align}
 \|\cB\|_{L^p(\bC; \,\bC)}\leq 1.4\,p, \quad \text{for}\quad p\geq 500.  
 \end{align} 
 In fact, numerically the values are so close that the best interval is probably in a small neighborhood of $450$.  Finally,  $\lim_{p\to\infty}h(p)=\frac{1}{e-2}\approx 1.3922111$.  These estimates brake the $\sqrt{2}$ "barrier," albeit not by much.

For a the study of topics related to those presented in this section, and specially further structures and inequalities for  conformal martingales, the reader is referred to  \cite{JanaOrtho}. 
\subsection{Final Remarks}
 \begin{remark} While it may be possible to refine such numerical bounds further, a more compelling direction is to  seek arguments that establish Iwaniec’s conjecture, even for special values of $p$. This goal was a principal motivation for our study of martingale Cotlar identities beyond the classical setting. Even establishing the conjecture for $p=4$, which to the best of our knowledge is still unknown, may already yield valuable insight.     
\end{remark}

\begin{remark}
Martingale methods, in combination with the Burkholder-Bellman techniques  have produced some of the strongest bounds toward
Iwaniec's conjecture and Stein's inequality. Nevertheless, a remarkable
fact is that no martingale or other stochastic analysis argument is
known that recovers even the simplest case $p=2$ with the sharp constant
$1$;  a result that follows immediately from the Fourier transform.
This suggests that it is worthwhile to search for alternative
approaches even for special cases of $p$.
\end{remark}
\begin{remark}
A comparison between Pisier's treatment of the vector of Riesz
transforms and the probabilistic representation of the
Beurling-Ahlfors operator reveals the obstruction 
for $p=2$. The argument in \cite{Pis} relies on projecting onto the first Wiener
chaos, but this projection fails to preserve the $L^2$ norm and
introduces the factor $\sqrt{\pi/2}$. In the Beurling-Ahlfors setting
the analogous step is the conditional expectation arising in the
probabilistic representation of the operator, which is even less
efficient as already at the $L^2$ level it incurs a loss of $2$.
Consequently, a direct adaptation of \cite{Pis} to the
Beurling-Ahlfors operator seems unlikely to yield sharp bounds without
exploiting additional structure. Nevertheless, such an approach may
still lead to improvements over the currently known estimates and
provide further insight into the problem.
\end{remark}

\begin{remark} Motivated by Pisier's treatment of higher order Riesz transforms 
leads us to further speculate about a possible adaptation of this method to the
Beurling-Ahlfors operator.  
Unlike the Riesz transforms whose symbols are linear in $\xi$, the multiplier for $\cB$
 is quadratic in $\xi$. In the Gaussian framework underlying the 
 argument in \cite{Pis}, linear dependence on $\xi$ corresponds naturally to
the first Wiener chaos. The quadratic structure
suggests that the appropriate analogue of the  projection onto the
first Wiener chaos would be a projection onto a suitable subspace of
the second Wiener chaos. This approach would be similar to  Remark \cite[Remark, p.496]{Pis} concerning higher order Riesz transforms.  

 From \eqref{chaos-decomp} we have that the second Wiener chaos is 
\[\mathcal H_2
=\operatorname{span}\{H_\alpha : |\alpha|=2\}\]
with bases 
\[\{H_2(x_1)=x^2_1-1,\,\, H_2(x_2)=x^2_2-1, \,\, H_1(x_1)\,H_1(x_2)=x_1x_2\}.\]
Thus 
\[\overline{\xi}^2=(\xi_1-i \xi_2)^2=(\xi_1^2-\xi_2^2)-2i \xi_1 \xi_2\in \mathcal H_2.\]
Similarly, 
\[
(\xi_1,\xi_2)
\begin{pmatrix}
1 & -i\\
-i & -1
\end{pmatrix}
\begin{pmatrix}
\xi_1\\
\xi_2
\end{pmatrix}
=(\xi_1+i \xi_2)^2 \in \mathcal H_2,
\]
and this shows that the same quadratic combination arises from the
matrix $A_0+iB_0$ appearing in the martingale representation. Thus both the Gaussian and martingale viewpoints point  to the
 second-chaos structure.  

It should be noted that this approach, if applicable,  will produce an additional  factor of $\frac{1}{(p-1)}$ in the range $1<p<2$ in place of  $\frac{1}{\sqrt{p-1}}$ arising from hypercontractivity. In addition,  as  pointed out in \cite[Remark, p.~ 497]{Pis},  projecting onto the $k$-Wiener chaos to study norms of higher order Riesz transforms requires $k$ to be odd, similarly to the method of rotations for powers of $\BH_{\bC}$.  \end{remark}

\subsection*{Acknowledgments} Part of the material presented here was prepared during the Fall 2024 semester, while the author was teaching an advanced graduate course in harmonic analysis. Although none could be included in the course due to time constraints, the author nevertheless wishes to acknowledge the students’ stimulating discussions and presentations on related topics. The author is grateful to Prabhu Janakiraman for the insightful discussions on Iwaniec’s conjecture over the years, particularly during the past two years, and the many comments and corrections on earlier versions of these notes.

\bibliography{ref.bib}

@incollection{Bakry1987,
  author    = {Bakry, Dominique},
  title     = {\'Etude des transformations de {R}iesz dans les vari\'et\'es riemanniennes \`a courbure de {R}icci minor\'ee},
  booktitle = {S\'eminaire de Probabilit\'es, XXI},
  series    = {Lecture Notes in Math.},
  volume    = {1247},
  pages     = {137--172},
  publisher = {Springer, Berlin},
  year      = {1987}
}

@article{Li2008,
  author  = {Li, Xiang-Dong},
  title   = {Martingale transforms and {$L^p$}-norm estimates of Riesz transforms on complete Riemannian manifolds},
  journal = {Probability Theory and Related Fields},
  volume  = {141},
  number  = {1--2},
  pages   = {247--281},
  year    = {2008},
  doi     = {10.1007/s00440-007-0085-y}
}

@article{AuscherCoulhonDuong2005,
  author  = {Auscher, Pascal and Coulhon, Thierry and Duong, Xuan Thinh},
  title   = {Riesz transform on manifolds and heat kernel regularity},
  journal = {Annales Scientifiques de l'\'Ecole Normale Sup\'erieure},
  series  = {4},
  volume  = {38},
  number  = {6},
  pages   = {911--957},
  year    = {2005},
  doi     = {10.1016/j.ansens.2005.09.002}
}

@article{Arcozzi1998,
  author  = {Arcozzi, Natan},
  title   = {Riesz transforms on compact Lie groups, spheres and Gauss space},
  journal = {Arkiv f{\"o}r Matematik},
  volume  = {36},
  number  = {2},
  pages   = {201--231},
  year    = {1998},
  doi     = {10.1007/BF02384338}
}

@article{CarbonaroDragicevic2013,
  author  = {Carbonaro, Andrea and Dragi{\v{c}}evi{\'c}, Oliver},
  title   = {Bellman function and linear dimension-free estimates in a theorem of Bakry},
  journal = {Journal of Functional Analysis},
  volume  = {265},
  number  = {7},
  pages   = {1085--1104},
  year    = {2013},
  doi     = {10.1016/j.jfa.2013.05.005}
}

@article{BanWanDavis,
  author  = {Rodrigo Ba{\~n}uelos and Gang Wang},
  title   = {Davis's inequality for orthogonal martingales under differential subordination},
  journal = {Michigan Math. J.},
  volume  = {47},
  number  = {1},
  pages   = {109--124},
  year    = {2000}
}

@article{Osekowski2013BellmanSurvey,
  author  = {Adam Os{\k{e}}kowski},
  title   = {Survey article: Bellman function method and sharp inequalities for martingales},
  journal = {Rocky Mountain Journal of Mathematics},
  volume  = {43},
  number  = {6},
  pages   = {1759--1823},
  year    = {2013},
  doi     = {10.1216/RMJ-2013-43-6-1759}
}

@article{Wrobel2018,
  author  = {B{\l}a{\.z}ej Wr{\'o}bel},
  title   = {Dimension-free $L^p$ estimates for vectors of Riesz transforms associated with orthogonal expansions},
  journal = {Analysis \& PDE},
  volume  = {11},
  number  = {3},
  pages   = {745--773},
  year    = {2018},
}

@article{OseRiesz,
  author = {Adam Os{\k{e}}kowski},
  title = {Sharp logarithmic inequalities for Riesz transforms},
  journal = {J. Funct. Anal.},
  volume = {262},
  number = {5},
  pages = {2633--2660},
  year = {2012}
}

@article{IwaniecMartin1993,
  author    = {Tadeusz Iwaniec and Gaven Martin},
  title     = {Quasiregular mappings in even dimensions},
  journal   = {Acta Mathematica},
  volume    = {170},
  year      = {1993},
  pages     = {29--81},
  doi       = {10.1007/BF02392696}
}

@article{DonaldsonSullivan1989,
  author    = {Simon K. Donaldson and Dennis P. Sullivan},
  title     = {Quasiconformal 4-manifolds},
  journal   = {Acta Mathematica},
  volume    = {163},
  year      = {1989},
  pages     = {181--252},
  doi       = {10.1007/BF02392736}
}

@article{Hytonen2009,
  author  = {Tuomas P. Hytönen},
  title   = {On the norm of the Beurling--Ahlfors operator in several dimensions},
  journal = {Advances in Mathematics},
  volume  = {231},
  number  = {3-4},
  pages   = {1639--1649},
  year    = {2012},
  publisher = {Elsevier},
  doi     = {10.1016/j.aim.2012.06.003}
}

@article{BanLin,
  author  = {Rodrigo Ba{\~n}uelos and Arthur Lindeman},
  title   = {A Martingale Study of the Beurling--Ahlfors Transform in $\mathbb{R}^n$},
  journal = {Journal of Functional Analysis},
  volume  = {145},
  number  = {2},
  pages   = {224--265},
  year    = {1997},
  publisher = {Academic Press},
  doi     = {10.1006/jfan.1997.3000}
}

@article{SpectorStockdale2021,
  author  = {Daniel Spector and Cody B. Stockdale},
  title   = {On the dimensional weak-type $(1,1)$ bound for {R}iesz transforms},
  journal = {Communications in Contemporary Mathematics},
  volume  = {23},
  number  = {8},
  year    = {2021},
  pages   = {2050072},
  doi     = {10.1142/S0219199720500728}
}

@inproceedings{Stein1987ICM,
  author    = {Elias M. Stein},
  title     = {Problems in Harmonic Analysis Related to Curvature and Oscillatory Integrals},
  booktitle = {Proceedings of the International Congress of Mathematicians},
  editor    = {A. M. Gleason},
  volume    = {1},
  year      = {1987},
  pages     = {196--221},
  publisher = {American Mathematical Society},
  address   = {Providence, RI},
  note      = {Berkeley, California, USA, August 3--11, 1986}
}

@article{Janakiraman2004,
  author  = {Prabhu Janakiraman},
  title   = {Weak-type Estimates for Singular Integrals and the Riesz Transform},
  journal = {Indiana University Mathematics Journal},
  volume  = {53},
  number  = {2},
  pages   = {533--555},
  year    = {2004},
  publisher = {Indiana University Mathematics Department},
  doi     = {10.1512/iumj.2004.53.2327},
}

@article{Ose2012,
  author  = {Adam Os\k{e}kowski},
  title   = {On the action of {R}iesz transforms on the class of bounded functions},
  journal = {Bull. Lond. Math. Soc.},
  volume  = {44},
  year    = {2012},
  number  = {6},
  pages   = {1205--1214}
}

@article{Lae2000,
  author  = {Erik Laeng},
  title   = {On the $L^p$ norms of the Hilbert transform of a characteristic function},
  journal = {Studia Mathematica},
  volume  = {140},
  number  = {3},
  year    = {2000},
  pages   = {237--251},
  doi     = {10.4064/sm-140-3-237-251}
}

@book{SteWei1971,
  author    = {Stein, Elias M. and Weiss, Guido},
  title     = {Introduction to Fourier Analysis on Euclidean Spaces},
  series    = {Princeton Mathematical Series},
  volume    = {32},
  publisher = {Princeton University Press},
  address   = {Princeton, New Jersey},
  year      = {1971},
  pages     = {x+297}
}

@article{Davis1973,
address = {DURHAM},
author = {Davis, Burgess},
copyright = {Copyright 2016 Elsevier B.V., All rights reserved.},
issn = {0012-7094},
journal = {Duke mathematical journal},
language = {eng},
number = {3},
pages = {695-700},
publisher = {DUKE University Press},
title = {On the distributions of conjugate functions of nonnegative measures},
volume = {40},
year = {1973},
}

@article{Jan2004,
  author  = {Prabhu Janakiraman},
  title   = {Best weak-type $(p,p)$ constants, $1 \le p \le 2$, for orthogonal harmonic functions and martingales},
  journal = {Illinois Journal of Mathematics},
  volume  = {48},
  number  = {3},
  pages   = {909--921},
  year    = {2004}
}

@book{Osekowski2012,
  author    = {Adam Os\k{e}kowski},
  title     = {Sharp Martingale and Semimartingale Inequalities},
  series    = {Monografie Matematyczne},
  volume    = {72},
  publisher = {Birkh\"auser},
  address   = {Basel},
  year      = {2012},
  isbn      = {978-3-0348-0395-3}
}

@article{BanOseAJM,
  author  = {Rodrigo Bañuelos and Adam Osękowski},
  title   = {Burkholder inequalities for submartingales, Bessel processes and conformal martingales},
  journal = {American Journal of Mathematics},
  volume  = {136},
  number  = {2},
  year    = {2014},
  pages   = {481--520},
  doi     = {10.1353/ajm.2014.0012},
  mrnumber = {3188063}
}

@article{BorJanVol2013,
  author  = {Borichev, Alexander and Janakiraman, Prabhu and Volberg, Alexander},
  title   = {Subordination by orthogonal martingales in {$L^p$} and zeros of {L}aguerre polynomials},
  journal = {Duke Mathematical Journal},
  volume  = {162},
  number  = {5},
  year    = {2013},
  pages   = {889--924},
  doi     = {10.1215/00127094-2153108},
  mrnumber = {3043590}
}

@article{BorJanVol2013b,
  author    = {Alexander Borichev and Prabhu Janakiraman and Alexander Volberg},
  title     = {On Burkholder function for orthogonal martingales and zeros of Legendre polynomials},
  journal   = {Amer. J. Math.},
  volume    = {135},
  number    = {1},
  pages     = {207--236},
  year      = {2013},
  doi       = {10.1353/ajm.2013.0004},
  url       = {https://doi.org/10.1353/ajm.2013.0004},
}

@article{BanuelosWang1996,
  author  = {Ba{\~n}uelos, Rodrigo and Wang, Gang},
  title   = {Orthogonal martingales under differential subordination and applications to {R}iesz transforms},
  journal = {Illinois Journal of Mathematics},
  volume  = {40},
  number  = {4},
  year    = {1996},
  pages   = {678--691},
  doi     = {10.1215/ijm/1255985943},
}

@book{Baudoin2014,
  author    = {Fran{\c{c}}ois Baudoin},
  title     = {Diffusion Processes and Stochastic Calculus},
  publisher = {EMS Textbooks in Mathematics},
  year      = {2014},
  isbn      = {9783037192046},
  series    = {European Mathematical Society (EMS)},
}

@article{NazarovVolberg2003,
  author    = {Fedor Nazarov and Alexander Volberg},
  title     = {Heating of the Beurling operator and estimates of its norms},
  journal   = {St. Petersburg Mathematical Journal},
  volume    = {14},
  number    = {3},
  pages     = {???--???},
  year      = {2003},
  note      = {Translation of Russian original},
}

@book{HornJohnsonMatrixAnalysis,
  author    = {Roger A. Horn and Charles R. Johnson},
  title     = {Matrix Analysis},
  publisher = {Cambridge University Press},
  edition   = {2nd},
  year      = {2013},
  address   = {Cambridge},
  isbn      = {978-0-521-83940-2}
}

@book{Hu2017,
  author    = {Yaozhong Hu},
  title     = {Analysis on Gaussian Spaces},
  publisher = {World Scientific Publishing Company},
  year      = {2017},
  address   = {Singapore},
  isbn      = {978-9813142183},
}

@misc{DLMF-Gamma,
  author       = {{NIST Digital Library of Mathematical Functions}},
  title        = {Chapter 5: Gamma Function},
  howpublished = {\url{https://dlmf.nist.gov/5}},
  year         = {2025},
  note         = {log-convexity of $\Gamma$ and related inequalities (e.g., Gautschi, Kershaw).}
}

@article{Gross1975,
  author    = {Leonard Gross},
  title     = {Logarithmic Sobolev inequalities},
  journal   = {American Journal of Mathematics},
  volume    = {97},
  number    = {4},
  pages     = {1061--1083},
  year      = {1975},
  doi       = {10.2307/2373688}
}

@article{Nelson1973,
  author    = {Edward Nelson},
  title     = {The free Markoff field},
  journal   = {Journal of Functional Analysis},
  volume    = {12},
  number    = {2},
  pages     = {211--227},
  year      = {1973},
  doi       = {10.1016/0022-1236(73)90024-1}
}

@article{CarDraKov2023,
  author    = {Carbonaro, Andrea and Dragi\u{c}evi\'c, Oliver and Kova\v{c}, Vjekoslav},
  title     = {Sharp $L^p$ estimates of powers of the complex Riesz transform},
  journal   = {Mathematische Annalen},
  volume    = {386},
  number    = {1--2},
  pages     = {1081--1108},
  year      = {2023},
  doi       = {10.1007/s00208-022-02419-3},
  eprint    = {arXiv:2109.08369},
  url       = {https://doi.org/10.1007/s00208-022-02419-3},
}

@article{CoiWei1976,
  author    = {Coifman, R. R. and Weiss, G.},
  title     = {Transference methods in analysis},
  journal   = {Conference Board of the Mathematical Sciences Regional Conference Series in Mathematics},
  volume    = {31},
  year      = {1976},
  publisher = {American Mathematical Society},
  address   = {Providence, RI},
  note      = {With an appendix by W. Rudin},
  isbn      = {978-0-8218-2001-1},
  mrnumber  = {0447631},
}

@book{AstIwaMar,
  author    = {Kari Astala and Tadeusz Iwaniec and Gaven J. Martin},
  title     = {Elliptic Partial Differential Equations and Quasiconformal Mappings in the Plane},
  series    = {Princeton Mathematical Series},
  volume    = {48},
  publisher = {Princeton University Press},
  year      = {2009},
  pages     = {xviii+416},
  isbn      = {978-0-691-14148-2},
  doi       = {10.1515/9781400831477},
}

@misc{Dra,
      title={Analysis of the Ahlfors-Beurling operator (lecture notes for the summer school at the University of Seville, 2013)}, 
      author={Oliver Dragičević},
      year={2021},
      eprint={2109.04555},
      archivePrefix={arXiv},
      primaryClass={math.CA},
      url={https://arxiv.org/abs/2109.04555}, 
}

@article{IwaSbo,
  author       = {Tadeusz Iwaniec and Carlo Sbordone},
  title        = {Riesz transforms and elliptic PDEs with VMO coefficients},
  journal      = {Journal d’Analyse Mathématique},
  volume       = {74},
  year         = {1998},
  pages        = {183--212},
  doi          = {10.1007/BF02789569},
}

@article{GohbergKrupnik1968,
  author   = {Gohberg, I. Ts. and Krupnik, N. Ya.},
  title    = {Norm of the {H}ilbert transformation in the {$L^p$} space},
  journal  = {Funktsional\cprime ny{\u\i} Analiz i ego Prilozheniya},
  volume   = {2},
  number   = {2},
  year     = {1968},
  pages    = {91--92},
  note     = {In Russian},
  mrnumber = {0238584},
}

@book{GohbergKrein1970,
  author    = {Gohberg, Israel and Kre{\u\i}n, Mark G.},
  title     = {Theory and Applications of Volterra Operators in {H}ilbert Space},
  series    = {Translations of Mathematical Monographs},
  volume    = {24},
  publisher = {American Mathematical Society},
  address   = {Providence, RI},
  year      = {1970},
  mrnumber  = {0264447},
}

@book {Dur,
    AUTHOR = {Durrett, Richard},
     TITLE = {Brownian motion and martingales in analysis},
    SERIES = {Wadsworth Mathematics Series},
 PUBLISHER = {Wadsworth International Group, Belmont, CA},
      YEAR = {1984},
     PAGES = {xi+328},
      ISBN = {0-534-03065-3},
   MRCLASS = {60G46 (31B05 35-02 42Bxx 46N05 60J65)},
  MRNUMBER = {750829},
MRREVIEWER = {Terry\ R.\ McConnell},
}

@book{Muller2020,
  author    = {M{\"u}ller, Paul F. X.},
  title     = {Hardy Martingales},
  series    = {New Mathematical Monographs},
  volume    = {45},
  publisher = {Cambridge University Press},
  year      = {2020},
  isbn      = {9781108838672}
}

@article{Uboe1986,
  author    = {Ub{\o}e, Jan},
  title     = {Conformal Martingales and Analytic Functions},
  journal   = {Mathematica Scandinavica},
  volume    = {59},
  number    = {1},
  pages     = {75--88},
  year      = {1986},
  publisher = {Mathematica Scandinavica},
  url       = {https://www.mscand.dk/article/view/12186}
}

@article{GetoorSharpe1972,
  author    = {Getoor, R. K. and Sharpe, M. J.},
  title     = {Conformal Martingales},
  journal   = {Inventiones Mathematicae},
  volume    = {16},
  pages     = {271--308},
  year      = {1972},
  publisher = {Springer},
  doi       = {10.1007/BF01425714}
}

@article{GonPerXia,
  author       = {Gonz{\'a}lez-P{\'e}rez, Adri{\'a}n M. and Parcet, Javier and Xia, Runlian},
  title        = {Noncommutative Cotlar identities for groups acting on tree-like structures},
  journal      = {arXiv preprint},
  eprint       = {2209.05298},
  archivePrefix= {arXiv},
  primaryClass = {math.OA},
  year         = {2024},
  note         = {Preprint (2024).}
}

@article{MeiRic,
  author  = {Mei, Tao and Ricard, {\'E}ric},
  title   = {Free {H}ilbert Transforms},
  journal = {Duke Mathematical Journal},
  volume  = {166},
  number  = {11},
  year    = {2017},
  pages   = {2153--2182},
  doi     = {10.1215/00127094-2017-0007}
}

@article {DraPetVol1,
    AUTHOR = {Dragi\v cevi\'c, Oliver and Petermichl, Stefanie and Volberg,
              Alexander},
     TITLE = {A rotation method which gives linear {$L^p$} estimates for
              powers of the {A}hlfors-{B}eurling operator},
   JOURNAL = {J. Math. Pures Appl. (9)},
  FJOURNAL = {Journal de Math\'ematiques Pures et Appliqu\'ees. Neuvi\`eme
              S\'erie},
    VOLUME = {86},
      YEAR = {2006},
    NUMBER = {6},
     PAGES = {492--509},
      ISSN = {0021-7824},
   MRCLASS = {30E20 (42B15 47B38 47G10)},
  MRNUMBER = {2281449},
MRREVIEWER = {A.\ B\"ottcher},
       DOI = {10.1016/j.matpur.2006.10.005},
       URL = {https://doi.org/10.1016/j.matpur.2006.10.005},
}

@book{Gra,
  author = {Grafakos, Loukas},
  title = {Classical Fourier Analysis},
  edition = {3},
  publisher = {Springer},
  year = {2014}
}

@article{DeL,
  author = {de Leeuw, Karel},
  title = {On $L^p$ multipliers},
  journal = {Annals of Mathematics},
  volume = {81},
  pages = {364--379},
  year = {1965}
}

@article{Ban,
	author = {Ba\~{n}uelos, Rodrigo},
	fjournal = {Illinois Journal of Mathematics},
	issn = {0019-2082},
	journal = {Illinois J. Math.},
	mrclass = {60G46},
	mrnumber = {2928339},
	mrreviewer = {Jos\'{e} Villa-Morales},
	number = {3},
	pages = {789--868 (2012)},
	title = {The foundational inequalities of {D}. {L}. {B}urkholder and some of their ramifications},
	url = {http://projecteuclid.org/euclid.ijm/1336049979},
	volume = {54},
	year = {2010},
}

@book {SteLP,
    AUTHOR = {Stein, Elias M.},
     TITLE = {Topics in harmonic analysis related to the
              {L}ittlewood-{P}aley theory},
    SERIES = {Annals of Mathematics Studies},
    VOLUME = {No. 63},
 PUBLISHER = {Princeton University Press, Princeton, NJ; University of Tokyo
              Press, Tokyo},
      YEAR = {1970},
     PAGES = {viii+146},
   MRCLASS = {42.50 (22.00)},
  MRNUMBER = {252961},
MRREVIEWER = {R.\ E.\ Edwards},
}

@article {Permic24,
    AUTHOR = {Domelevo, Komla and Petermichl, Stefanie and \v Skreb,
              Kristina Ana},
     TITLE = {Continuous sparse domination and dimensionless weighted
              estimates for the {B}akry-{R}iesz vector},
   JOURNAL = {J. Reine Angew. Math.},
  FJOURNAL = {Journal f\"ur die Reine und Angewandte Mathematik. [Crelle's
              Journal]},
    VOLUME = {824},
      YEAR = {2025},
     PAGES = {137--166},
      ISSN = {0075-4102,1435-5345},
   MRCLASS = {42B20 (42B25 58J65)},
  MRNUMBER = {4926944},
       DOI = {10.1515/crelle-2025-0024},
       URL = {https://doi-org.ezproxy.lib.purdue.edu/10.1515/crelle-2025-0024},
}

@article {BanMen,
    AUTHOR = {Ba\~{n}uelos, Rodrigo and M\'{e}ndez-Hern\'{a}ndez, P. J.},
     TITLE = {Space-time {B}rownian motion and the {B}eurling-{A}hlfors
              transform},
   JOURNAL = {Indiana Univ. Math. J.},
  FJOURNAL = {Indiana University Mathematics Journal},
    VOLUME = {52},
      YEAR = {2003},
    NUMBER = {4},
     PAGES = {981--990},
      ISSN = {0022-2518},
   MRCLASS = {60G46 (30C62 47G10 60G44)},
  MRNUMBER = {2001941},
MRREVIEWER = {Jean-Claude Gruet},
       DOI = {10.1512/iumj.2003.52.2218},
}

@article {BanBauLiYan,
    AUTHOR = {Ba\~nuelos, Rodrigo and Baudoin, Fabrice and Chen, Li and
              Sire, Yannick},
     TITLE = {Multiplier theorems via martingale transforms},
   JOURNAL = {J. Funct. Anal.},
  FJOURNAL = {Journal of Functional Analysis},
    VOLUME = {281},
      YEAR = {2021},
    NUMBER = {9},
     PAGES = {Paper No. 109188, 37},
      ISSN = {0022-1236,1096-0783},
   MRCLASS = {60G44 (35A22 42B15 58J65)},
  MRNUMBER = {4295971},
MRREVIEWER = {H\'el\`ene\ Airault},
       DOI = {10.1016/j.jfa.2021.109188},
       URL = {https://doi-org.ezproxy.lib.purdue.edu/10.1016/j.jfa.2021.109188},
}

@article{Pic72,
  author  = {S. K. Pichorides},
  title   = {On the best values of the constants in the theorems of M. Riesz, Zygmund and Kolmogorov},
  journal = {Studia Mathematica},
  volume  = {44},
  number  = {2},
  pages   = {165--179},
  year    = {1972},
  doi     = {10.4064/sm-44-2-165-179},
}

@Book{Gam,
 Author = {Gamelin, T. W.},
 Title = {Uniform algebras and {Jensen} measures},
 FSeries = {London Mathematical Society Lecture Note Series},
 Series = {Lond. Math. Soc. Lect. Note Ser.},
 ISSN = {0076-0552},
 Volume = {32},
 Year = {1978},
 Publisher = {Cambridge University Press, Cambridge. London Mathematical Society, London},
 Language = {English},
 zbMATH = {3651205},
 Zbl = {0418.46042}
}

@article {BanKwa1,
    AUTHOR = {Ba\~{n}uelos, Rodrigo and Kwa\'{s}nicki, Mateusz},
     TITLE = {The {$\ell^p$} norm of the {R}iesz-{T}itchmarsh transform for
              even integer {$p$}},
   JOURNAL = {J. Lond. Math. Soc. (2)},
  FJOURNAL = {Journal of the London Mathematical Society. Second Series},
    VOLUME = {109},
      YEAR = {2024},
    NUMBER = {4},
     PAGES = {Paper No. e12888, 21},
      ISSN = {0024-6107,1469-7750},
   MRCLASS = {42A05 (39A12 42A50)},
  MRNUMBER = {4727420},
       DOI = {10.1112/jlms.12888},
       URL = {https://doi.org/10.1112/jlms.12888},
}

@article {ChaWilWol,
    AUTHOR = {Chang, S.-Y.\ A. and Wilson, J. M. and Wolff, T. H.},
     TITLE = {Some weighted norm inequalities concerning the {S}chr\"odinger
              operators},
   JOURNAL = {Comment. Math. Helv.},
  FJOURNAL = {Commentarii Mathematici Helvetici},
    VOLUME = {60},
      YEAR = {1985},
    NUMBER = {2},
     PAGES = {217--246},
      ISSN = {0010-2571,1420-8946},
   MRCLASS = {42B25 (47B38 47F05 81C10)},
  MRNUMBER = {800004},
MRREVIEWER = {Ron\ Kerman},
       DOI = {10.1007/BF02567411},
       URL = {https://doi.org/10.1007/BF02567411},
}

@book {BanMoore,
    AUTHOR = {Ba\~nuelos, Rodrigo and Moore, Charles N.},
     TITLE = {Probabilistic behavior of harmonic functions},
    SERIES = {Progress in Mathematics},
    VOLUME = {175},
 PUBLISHER = {Birkh\"auser Verlag, Basel},
      YEAR = {1999},
     PAGES = {xiv+204},
      ISBN = {3-7643-6062-3},
   MRCLASS = {31B05 (42B25 60G46)},
  MRNUMBER = {1707297},
MRREVIEWER = {Tao\ Qian},
       DOI = {10.1007/978-3-0348-8728-1},
       URL = {https://doi.org/10.1007/978-3-0348-8728-1},
}

@article {Cot55,
    AUTHOR = {Cotlar, Mischa},
     TITLE = {A unified theory of {H}ilbert transforms and ergodic theorems},
   JOURNAL = {Rev. Mat. Cuyana},
  FJOURNAL = {Revista Matem\'atica Cuyana},
    VOLUME = {1},
      YEAR = {1955},
     PAGES = {105--167},
      ISSN = {0484-7822},
   MRCLASS = {44.0X},
  MRNUMBER = {84632},
MRREVIEWER = {K.\ Yosida},
}

@article {MR780616,
    AUTHOR = {Duoandikoetxea, Javier and Rubio de Francia, Jos\'e L.},
     TITLE = {Estimations ind\'ependantes de la dimension pour les
              transform\'ees de Riesz},
   JOURNAL = {C. R. Acad. Sci. Paris S\'er. I Math.},
  FJOURNAL = {Comptes Rendus des S\'eances de l'Acad\'emie des Sciences.
              S\'erie I. Math\'ematique},
    VOLUME = {300},
      YEAR = {1985},
    NUMBER = {7},
     PAGES = {193--196},
      ISSN = {0249-6291},
   MRCLASS = {42B20},
  MRNUMBER = {780616},
MRREVIEWER = {Gerald\ B.\ Folland},
}

@article {BenA,
    AUTHOR = {Bennett, Andrew G.},
     TITLE = {Probabilistic square functions and a priori estimates},
   JOURNAL = {Trans. Amer. Math. Soc.},
  FJOURNAL = {Transactions of the American Mathematical Society},
    VOLUME = {291},
      YEAR = {1985},
    NUMBER = {1},
     PAGES = {159--166},
      ISSN = {0002-9947,1088-6850},
   MRCLASS = {42B20 (42A61 42B25 60J65)},
  MRNUMBER = {797052},
MRREVIEWER = {Rainer\ Wittmann},
       DOI = {10.2307/1999901},
       URL = {https://doi.org/10.2307/1999901},
}

@article {BanMich,
    AUTHOR = {Ba\~nuelos, Rodrigo},
     TITLE = {A sharp good-{$\lambda$} inequality with an application to
              {R}iesz transforms},
   JOURNAL = {Michigan Math. J.},
  FJOURNAL = {Michigan Mathematical Journal},
    VOLUME = {35},
      YEAR = {1988},
    NUMBER = {1},
     PAGES = {117--125},
      ISSN = {0026-2285,1945-2365},
   MRCLASS = {42B20 (60G46)},
  MRNUMBER = {931943},
MRREVIEWER = {Mario\ Milman},
       DOI = {10.1307/mmj/1029003685},
       URL = {https://doi.org/10.1307/mmj/1029003685},
}

@article {MR3558516,
    AUTHOR = {Strzelecki, Micha\l},
     TITLE = {The {$L^p$}-norms of the {B}eurling-{A}hlfors transform on
              radial functions},
   JOURNAL = {Ann. Acad. Sci. Fenn. Math.},
  FJOURNAL = {Annales Academi\ae\ Scientiarum Fennic\ae. Mathematica},
    VOLUME = {42},
      YEAR = {2017},
    NUMBER = {1},
     PAGES = {73--93},
      ISSN = {1239-629X,1798-2383},
   MRCLASS = {60G46 (42B20)},
  MRNUMBER = {3558516},
MRREVIEWER = {I.\ Ya.\ Novikov},
       DOI = {10.5186/aasfm.2017.4204},
       URL = {https://doi.org/10.5186/aasfm.2017.4204},
}

@article {JanaOrtho,
    AUTHOR = {Janakiraman, Prabhu},
     TITLE = {Orthogonality in complex martingale spaces and connections
              with the {B}eurling-{A}hlfors transform},
   JOURNAL = {Illinois J. Math.},
  FJOURNAL = {Illinois Journal of Mathematics},
    VOLUME = {54},
      YEAR = {2010},
    NUMBER = {4},
     PAGES = {1509--1563},
      ISSN = {0019-2082,1945-6581},
   MRCLASS = {60G44 (42B20)},
  MRNUMBER = {2981858},
MRREVIEWER = {Ferenc\ Weisz},
       URL = {http://projecteuclid.org/euclid.ijm/1348505539},
}

@article {MR3018958,
    AUTHOR = {Ba\~nuelos, Rodrigo and Os\c ekowski, Adam},
     TITLE = {Sharp inequalities for the {B}eurling-{A}hlfors transform on
              radial functions},
   JOURNAL = {Duke Math. J.},
  FJOURNAL = {Duke Mathematical Journal},
    VOLUME = {162},
      YEAR = {2013},
    NUMBER = {2},
     PAGES = {417--434},
      ISSN = {0012-7094,1547-7398},
   MRCLASS = {60G46 (42B20)},
  MRNUMBER = {3018958},
MRREVIEWER = {Tuomas\ P.\ Hyt\"onen},
       DOI = {10.1215/00127094-1962649},
       URL = {https://doi.org/10.1215/00127094-1962649},
}

@article {DraVol,
    AUTHOR = {Dragicevi\'c, Oliver and Volberg, Alexander},
     TITLE = {Bellman function, {L}ittlewood-{P}aley estimates and
              asymptotics for the {A}hlfors-{B}eurling operator in
              {$L^p(\Bbb C)$}},
   JOURNAL = {Indiana Univ. Math. J.},
  FJOURNAL = {Indiana University Mathematics Journal},
    VOLUME = {54},
      YEAR = {2005},
    NUMBER = {4},
     PAGES = {971--995},
      ISSN = {0022-2518,1943-5258},
   MRCLASS = {30C62 (31B15 35K05 42B25 47B38 47G10)},
  MRNUMBER = {2164413},
MRREVIEWER = {M.\ Nakai},
       DOI = {10.1512/iumj.2005.54.2554},
       URL = {https://doi.org/10.1512/iumj.2005.54.2554},
}

@book {Lehto,
    AUTHOR = {Lehto, O. and Virtanen, K. I.},
     TITLE = {Quasiconformal mappings in the plane},
    SERIES = {Die Grundlehren der mathematischen Wissenschaften},
    VOLUME = {Band 126},
   EDITION = {Second},
      NOTE = {Translated from the German by K. W. Lucas},
 PUBLISHER = {Springer-Verlag, New York-Heidelberg},
      YEAR = {1973},
     PAGES = {viii+258},
   MRCLASS = {30A60},
  MRNUMBER = {344463},
}

@book {Volberg1,
    AUTHOR = {Vasyunin, Vasily and Volberg, Alexander},
     TITLE = {The {B}ellman function technique in harmonic analysis},
    SERIES = {Cambridge Studies in Advanced Mathematics},
    VOLUME = {186},
 PUBLISHER = {Cambridge University Press, Cambridge},
      YEAR = {2020},
     PAGES = {xvii+445},
      ISBN = {978-1-108-48689-7},
   MRCLASS = {42-02 (35J70 42B15 42B20 42B25 60H10)},
  MRNUMBER = {4411371},
MRREVIEWER = {Tuomas\ P.\ Hyt\"onen},
       DOI = {10.1017/9781108764469},
       URL = {https://doi.org/10.1017/9781108764469},
}

@book {Astala1,
    AUTHOR = {Astala, Kari and Iwaniec, Tadeusz and Martin, Gaven},
     TITLE = {Elliptic partial differential equations and quasiconformal
              mappings in the plane},
    SERIES = {Princeton Mathematical Series},
    VOLUME = {48},
 PUBLISHER = {Princeton University Press, Princeton, NJ},
      YEAR = {2009},
     PAGES = {xviii+677},
      ISBN = {978-0-691-13777-3},
   MRCLASS = {30C62 (30G20 35J46 35J60 35J92)},
  MRNUMBER = {2472875},
MRREVIEWER = {Olli\ Martio},
}

@article{BanKim2026,
  author  = {Rodrigo Ba{\~n}uelos and Daesung Kim},
  title   = {Discrete analogues of second-order Riesz transforms},
  journal = {Journal of the London Mathematical Society},
  volume  = {113},
  number  = {3},
  year    = {2026},
  pages   = {e70498},
  doi     = {10.1112/jlms.70498}
}

@article{BanKimKwa2026,
  author  = {Rodrigo Ba{\~n}uelos and Daesung Kim and Mateusz Kwa{\'s}nicki},
  title   = {Sharp {$\ell^p$} inequalities for discrete singular integrals on the lattice},
  journal = {Journal of Functional Analysis},
  volume  = {290},
  number  = {9},
  year    = {2026},
  pages   = {111359},
  doi     = {10.1016/j.jfa.2026.111359}
}

@article {Riesz,
    AUTHOR = {Riesz, Marcel},
     TITLE = {Sur les fonctions conjugu\'{e}es},
   JOURNAL = {Math. Z.},
  FJOURNAL = {Mathematische Zeitschrift},
    VOLUME = {27},
      YEAR = {1928},
    NUMBER = {1},
     PAGES = {218--244},
      ISSN = {0025-5874},
   MRCLASS = {DML},
  MRNUMBER = {1544909},
       DOI = {10.1007/BF01171098},
}

@misc{BanKimKwa,
      title={Sharp $\ell^p$ inequalities for discrete singular integrals}, 
      author={Rodrigo Bañuelos and Daesung Kim and Mateusz Kwaśnicki},
      year={2022},
      eprint={2209.09737},
      archivePrefix={arXiv},
      primaryClass={math.PR},
      JOURNAL = {arXiv e-print},
      url={https://arxiv.org/abs/2211.10762},
}

@article {GesMonSak,
    AUTHOR = {Geiss, Stefan and Montgomery-Smith, Stephen and Saksman, Eero},
     TITLE = {On singular integral and martingale transforms},
   JOURNAL = {Trans. Amer. Math. Soc.},
  FJOURNAL = {Transactions of the American Mathematical Society},
    VOLUME = {362},
      YEAR = {2010},
    NUMBER = {2},
     PAGES = {553--575},
      ISSN = {0002-9947},
   MRCLASS = {60G46 (42B15 46B20)},
  MRNUMBER = {2551497},
       DOI = {10.1090/S0002-9947-09-04953-8},
}

@article {Iwa82,
    AUTHOR = {Iwaniec, T.},
     TITLE = {Extremal inequalities in {S}obolev spaces and quasiconformal
              mappings},
   JOURNAL = {Z. Anal. Anwendungen},
  FJOURNAL = {Zeitschrift f\"{u}r Analysis und ihre Anwendungen},
    VOLUME = {1},
      YEAR = {1982},
    NUMBER = {6},
     PAGES = {1--16},
      ISSN = {0232-2064},
   MRCLASS = {30C60 (30C75 46E35)},
  MRNUMBER = {719167},
MRREVIEWER = {V. M. Gol\cprime dshte\u{\i}n},
       DOI = {10.4171/ZAA/37},
       URL = {https://doi-org.ezproxy.lib.purdue.edu/10.4171/ZAA/37},
}

@article {BanJan,
    AUTHOR = {Ba\~{n}uelos, Rodrigo and Janakiraman, Prabhu},
     TITLE = {{$L^p$}-bounds for the {B}eurling-{A}hlfors transform},
   JOURNAL = {Trans. Amer. Math. Soc.},
  FJOURNAL = {Transactions of the American Mathematical Society},
    VOLUME = {360},
      YEAR = {2008},
    NUMBER = {7},
     PAGES = {3603--3612},
      ISSN = {0002-9947},
   MRCLASS = {42B20 (44A15)},
  MRNUMBER = {2386238},
MRREVIEWER = {Wilfredo O. Urbina},
       DOI = {10.1090/S0002-9947-08-04537-6},
       URL = {https://doi-org.ezproxy.lib.purdue.edu/10.1090/S0002-9947-08-04537-6},
}

@incollection {Bak1,
    AUTHOR = {Bakry, Dominique},
     TITLE = {The {R}iesz transforms associated with second order
              differential operators},
 BOOKTITLE = {Seminar on {S}tochastic {P}rocesses, 1988 ({G}ainesville,
              {FL}, 1988)},
    SERIES = {Progr. Probab.},
    VOLUME = {17},
     PAGES = {1--43},
 PUBLISHER = {Birkh\"{a}user Boston, Boston, MA},
      YEAR = {1989},
   MRCLASS = {58G32 (22E30 35J15 47B38)},
  MRNUMBER = {990472},
MRREVIEWER = {Hiroshi Akiyama},
       DOI = {10.1214/aop/1176991490},
       URL = {https://doi-org.ezproxy.lib.purdue.edu/10.1214/aop/1176991490},
}

@incollection {Mey1,
    AUTHOR = {Meyer, P.-A.},
     TITLE = {Transformations de {R}iesz pour les lois gaussiennes},
 BOOKTITLE = {Seminar on probability, {XVIII}},
    SERIES = {Lecture Notes in Math.},
    VOLUME = {1059},
     PAGES = {179--193},
 PUBLISHER = {Springer, Berlin},
      YEAR = {1984},
   MRCLASS = {60H07},
  MRNUMBER = {770960},
MRREVIEWER = {Naresh C. Jain},
       DOI = {10.1007/BFb0100043},
       URL = {https://doi-org.ezproxy.lib.purdue.edu/10.1007/BFb0100043},
}

@article {SteSome,
    AUTHOR = {Stein, E. M.},
     TITLE = {Some results in harmonic analysis in {${\bf R}^{n}$}, for
              {$n\rightarrow \infty $}},
   JOURNAL = {Bull. Amer. Math. Soc. (N.S.)},
  FJOURNAL = {American Mathematical Society. Bulletin. New Series},
    VOLUME = {9},
      YEAR = {1983},
    NUMBER = {1},
     PAGES = {71--73},
      ISSN = {0273-0979},
   MRCLASS = {42B25 (42B20)},
  MRNUMBER = {699317},
       DOI = {10.1090/S0273-0979-1983-15157-1},
       URL = {https://doi-org.ezproxy.lib.purdue.edu/10.1090/S0273-0979-1983-15157-1},
}

@article {DraVol4,
    AUTHOR = {Dragi\v{c}evi\'{c}, Oliver and Volberg, Alexander},
     TITLE = {Bellman function, {L}ittlewood-{P}aley estimates and
              asymptotics for the {A}hlfors-{B}eurling operator in
              {$L^p(\Bbb C)$}},
   JOURNAL = {Indiana Univ. Math. J.},
  FJOURNAL = {Indiana University Mathematics Journal},
    VOLUME = {54},
      YEAR = {2005},
    NUMBER = {4},
     PAGES = {971--995},
      ISSN = {0022-2518},
   MRCLASS = {30C62 (31B15 35K05 42B25 47B38 47G10)},
  MRNUMBER = {2164413},
MRREVIEWER = {M. Nakai},
       DOI = {10.1512/iumj.2005.54.2554},
       URL = {https://doi-org.ezproxy.lib.purdue.edu/10.1512/iumj.2005.54.2554},
}

@incollection{Pis,
  author    = {Pisier, Gilles},
  title     = {Riesz transforms: a simpler analytic proof of {P.-A.} {M}eyer's inequality},
  booktitle = {S\'{e}minaire de Probabilit\'{e}s, {XXII}},
  series    = {Lecture Notes in Mathematics},
  volume    = {1321},
  pages     = {485--501},
  publisher = {Springer},
  address   = {Berlin},
  year      = {1988},
  mrnumber  = {960544},
  mrclass   = {60J35 (42B20)},
  reviewer  = {R. F. Gundy},
  doi       = {10.1007/BFb0084154},
}

@article {Ban86,
    AUTHOR = {Ba\~{n}uelos, Rodrigo},
     TITLE = {Martingale transforms and related singular integrals},
   JOURNAL = {Trans. Amer. Math. Soc.},
  FJOURNAL = {Transactions of the American Mathematical Society},
    VOLUME = {293},
      YEAR = {1986},
    NUMBER = {2},
     PAGES = {547--563},
      ISSN = {0002-9947},
   MRCLASS = {60G44 (42B20 60G46 60H05)},
  MRNUMBER = {816309},
MRREVIEWER = {Martin L. Silverstein},
       DOI = {10.2307/2000021},
       URL = {https://doi.org/10.2307/2000021},
}

@article {BBLS2021,
    AUTHOR = {Ba\~{n}uelos, Rodrigo and Baudoin, Fabrice and Chen, Li and Sire,
              Yannick},
     TITLE = {Multiplier theorems via martingale transforms},
   JOURNAL = {J. Funct. Anal.},
  FJOURNAL = {Journal of Functional Analysis},
    VOLUME = {281},
      YEAR = {2021},
    NUMBER = {9},
     PAGES = {Paper No. 109188, 37},
      ISSN = {0022-1236},
   MRCLASS = {60G44 (35A22 42B15 58J65)},
  MRNUMBER = {4295971},
       DOI = {10.1016/j.jfa.2021.109188},
}

@article {BanKwa,
    AUTHOR = {Ba\~{n}uelos, Rodrigo and Kwa\'{s}nicki, Mateusz},
     TITLE = {On the {$\ell^p$}-norm of the discrete {H}ilbert transform},
   JOURNAL = {Duke Math. J.},
  FJOURNAL = {Duke Mathematical Journal},
    VOLUME = {168},
      YEAR = {2019},
    NUMBER = {3},
     PAGES = {471--504},
      ISSN = {0012-7094},
   MRCLASS = {60J45 (60G42 60J70)},
  MRNUMBER = {3909902},
MRREVIEWER = {Volkert Paulsen},
       DOI = {10.1215/00127094-2018-0047},
}

@article {Burk84,
    AUTHOR = {Burkholder, D. L.},
     TITLE = {Boundary value problems and sharp inequalities for martingale transforms},
   JOURNAL = {Ann. Probab.},
  FJOURNAL = {The Annals of Probability},
    VOLUME = {12},
      YEAR = {1984},
    NUMBER = {3},
     PAGES = {647--702},
      ISSN = {0091-1798},
   MRCLASS = {60G42 (46N05 60G46)},
  MRNUMBER = {744226},
MRREVIEWER = {Jean-Yves Ouvrard},
}

@article{IwaMar,
	Author = {Iwaniec, T. and Martin, G.},
	Fjournal = {Journal f{\"u}r die Reine und Angewandte Mathematik. [Crelle's Journal]},
	Issn = {0075-4102},
	Journal = {J. Reine Angew. Math.},
	Mrclass = {42B20 (30E20 47G10)},
	Mrnumber = {1390681},
	Mrreviewer = {S. K. Vodop\cprime yanov},
	Pages = {25--57},
	Title = {Riesz transforms and related singular integrals},
	Volume = {473},
	Year = {1996},
}

@book {Stein70,
    AUTHOR = {Stein, E. M.},
     TITLE = {Singular integrals and differentiability properties of functions},
    SERIES = {Princeton Mathematical Series, No. 30},
 PUBLISHER = {Princeton University Press, Princeton, N.J.},
      YEAR = {1970},
     PAGES = {xiv+290},
   MRCLASS = {46.38 (26.00)},
  MRNUMBER = {0290095},
MRREVIEWER = {R. E. Edwards},
}

@article{Bur84,
	Author = {Burkholder, D. L.},
	Fjournal = {The Annals of Probability},
	Issn = {0091-1798},
	Journal = {Ann. Probab.},
	Mrclass = {60G42 (46N05 60G46)},
	Mrnumber = {744226},
	Mrreviewer = {Jean-Yves Ouvrard},
	Number = {3},
	Pages = {647--702},
	Title = {Boundary value problems and sharp inequalities for martingale transforms},
	Url = {https://mathscinet.ams.org/mathscinet-getitem?mr=744226},
	Volume = {12},
	Year = {1984},
}

@article{GV79,
	Author = {Gundy, R. F. and Varopoulos, N. Th.},
	Fjournal = {Comptes Rendus Hebdomadaires des S{\'e}ances de l'Acad{\'e}mie des Sciences. S{\'e}ries A et B},
	Issn = {0151-0509},
	Journal = {C. R. Acad. Sci. Paris S{\'e}r. A-B},
	Mrclass = {60H05 (60J65)},
	Mrnumber = {545671},
	Mrreviewer = {Maurizio Pratelli},
	Number = {1},
	Pages = {A13--A16},
	Title = {Les transformations de {R}iesz et les int{\'e}grales stochastiques},
	Volume = {289},
	Year = {1979},
}

@article{BW95,
	Author = {Ba\~nuelos, R. and Wang, G.},
	Doi = {10.1215/S0012-7094-95-08020-X},
	Fjournal = {Duke Mathematical Journal},
	Issn = {0012-7094},
	Journal = {Duke Math. J.},
	Mrclass = {60G44 (30C80 60G46)},
	Mrnumber = {1370109},
	Mrreviewer = {Tao Qian},
	Number = {3},
	Pages = {575--600},
	Title = {Sharp inequalities for martingales with applications to the {B}eurling-{A}hlfors and {R}iesz transforms},
	Url = {http://dx.doi.org/10.1215/S0012-7094-95-08020-X},
	Volume = {80},
	Year = {1995},
}

@article {BanOsc15,
    AUTHOR = {Ba\~{n}uelos, Rodrigo and Os\c{e}kowski, Adam},
     TITLE = {Sharp martingale inequalities and applications to {R}iesz
              transforms on manifolds, {L}ie groups and {G}auss space},
   JOURNAL = {J. Funct. Anal.},
  FJOURNAL = {Journal of Functional Analysis},
    VOLUME = {269},
      YEAR = {2015},
    NUMBER = {6},
     PAGES = {1652--1713},
      ISSN = {0022-1236,1096-0783},
   MRCLASS = {42B25 (60G44)},
  MRNUMBER = {3373431},
MRREVIEWER = {Steven\ Michael\ Heilman},
       DOI = {10.1016/j.jfa.2015.06.015},
       URL = {https://doi.org/10.1016/j.jfa.2015.06.015},
}

@article {BanBau13,
    AUTHOR = {Ba\~{n}uelos, Rodrigo and Baudoin, Fabrice},
     TITLE = {Martingale transforms and their projection operators on
              manifolds},
   JOURNAL = {Potential Anal.},
  FJOURNAL = {Potential Analysis. An International Journal Devoted to the
              Interactions between Potential Theory, Probability Theory,
              Geometry and Functional Analysis},
    VOLUME = {38},
      YEAR = {2013},
    NUMBER = {4},
     PAGES = {1071--1089},
      ISSN = {0926-2601,1572-929X},
   MRCLASS = {58J65 (60G46 60H05)},
  MRNUMBER = {3042695},
MRREVIEWER = {Stavros\ Vakeroudis},
       DOI = {10.1007/s11118-012-9307-8},
       URL = {https://doi.org/10.1007/s11118-012-9307-8},
}

\
\end{document}